\documentclass[11pt]{amsart}
\usepackage{xypic}
\usepackage{amssymb}
\newcommand\Rl{{\mathbb{R}}}

\catcode`@=11

\let\latex@newcommand\newcommand
\let\latex@renewcommand\renewcommand

\def\if@undefined#1{\ifx#1\un@@@defined@@@}
\def\newcommand#1{\if@undefined{#1}
  \let\next@=\latex@newcommand \else
  \let\next@=\latex@renewcommand \fi
  \next@{#1}}

\catcode`@=12

\begin{document}
\theoremstyle{plain}
\newtheorem{thm}{Theorem}[section]
\newtheorem*{thm*}{Theorem}
\newtheorem{prop}[thm]{Proposition}
\newtheorem*{prop*}{Proposition}
\newtheorem{lemma}[thm]{Lemma}
\newtheorem{cor}[thm]{Corollary}
\newtheorem*{conj*}{Conjecture}
\newtheorem*{cor*}{Corollary}
\newtheorem{defn}[thm]{Definition}
\newtheorem{cond}{Condition}
\theoremstyle{definition}
\newtheorem*{defn*}{Definition}
\newtheorem{rems}[thm]{Remarks}
\newtheorem*{rems*}{Remarks}
\newtheorem*{proof*}{Proof}
\newtheorem*{not*}{Notation}
\newcommand{\npartial}{\slash\!\!\!\partial}
\newcommand{\Heis}{\operatorname{Heis}}
\newcommand{\Solv}{\operatorname{Solv}}
\newcommand{\Spin}{\operatorname{Spin}}
\newcommand{\SO}{\operatorname{SO}}
\newcommand{\ind}{\operatorname{ind}}
\newcommand{\Index}{\operatorname{index}}
\newcommand{\ch}{\operatorname{ch}}
\newcommand{\rank}{\operatorname{rank}}
\newcommand{\abs}[1]{\lvert#1\rvert}
 \newcommand{\A}{{\mathcal A}}
 \newcommand{\E}{{\mathcal E}}
        \newcommand{\D}{{\mathcal D}}\newcommand{\HH}{{\mathcal H}}
        \newcommand{\LL}{{\mathcal L}}
        \newcommand{\B}{{\mathcal B}}
         \newcommand{\S}{{\mathcal S}}
 \newcommand{\F}{{\mathcal F}}

        \newcommand{\K}{{\mathcal K}}
\newcommand{\oo}{{\mathcal O}}
         \newcommand{\PP}{{\mathcal P}}
        \newcommand{\s}{\sigma}
\newcommand{\al}{\alpha}
        \newcommand{\coker}{{\mbox coker}}
        \newcommand{\p}{\partial}
        \newcommand{\dd}{|\D|}
        \newcommand{\n}{\parallel}
\newcommand{\bma}{\left(\begin{array}{cc}}
\newcommand{\ema}{\end{array}\right)}
\newcommand{\bca}{\left(\begin{array}{c}}
\newcommand{\eca}{\end{array}\right)}
\def\clsp{\overline{\operatorname{span}}}
\def\T{\mathbb T}
\def\Aut{\operatorname{Aut}}

\newcommand{\sr}{\stackrel}
\newcommand{\da}{\downarrow}
\newcommand{\tD}{\tilde{\D}}
        \newcommand{\Q}{\mathbb Q}
        \newcommand{\R}{\mathbb R}
        \newcommand{\C}{\mathbb C}
        \newcommand{\h}{\mathbf H}
\newcommand{\Z}{\mathbb Z}
\newcommand{\N}{\mathbb N}
\newcommand{\tto}{\longrightarrow}
\newcommand{\ben}{\begin{displaymath}}
        \newcommand{\een}{\end{displaymath}}
\newcommand{\be}{\begin{equation}}
\newcommand{\ee}{\end{equation}}

        \newcommand{\bean}{\begin{eqnarray*}}
        \newcommand{\eean}{\end{eqnarray*}}
\newcommand{\nno}{\nonumber\\}
\newcommand{\bea}{\begin{eqnarray}}
        \newcommand{\eea}{\end{eqnarray}}

\def\cross#1{\rlap{\hskip#1pt\hbox{$-$}}}
        \def\intcross{\cross{0.3}\int}
        \def\bigintcross{\cross{2.3}\int}

\newcommand{\supp}[1]{\operatorname{#1}}
\newcommand{\norm}[1]{\parallel\, #1\, \parallel}
\newcommand{\ip}[2]{\langle #1,#2\rangle}
\setlength{\parskip}{.3cm}
\newcommand{\nc}{\newcommand}
\nc{\nt}{\newtheorem} \nc{\gf}[2]{\genfrac{}{}{0pt}{}{#1}{#2}}
\nc{\mb}[1]{{\mbox{$ #1 $}}} \nc{\real}{{\mathbb R}}
\nc{\comp}{{\mathbb C}} \nc{\ints}{{\mathbb Z}}
\nc{\Ltoo}{\mb{L^2({\mathbf H})}} \nc{\rtoo}{\mb{{\mathbf R}^2}}
\nc{\slr}{{\mathbf {SL}}(2,\real)} \nc{\slz}{{\mathbf
{SL}}(2,\ints)} \nc{\su}{{\mathbf {SU}}(1,1)} \nc{\so}{{\mathbf
{SO}}} \nc{\hyp}{{\mathbb H}} \nc{\disc}{{\mathbf D}}
\nc{\torus}{{\mathbb T}}
\newcommand{\tk}{\widetilde{K}}
\newcommand{\boe}{{\bf e}}\newcommand{\bt}{{\bf t}}
\newcommand{\vth}{\vartheta}
\newcommand{\CGh}{\widetilde{\CG}}
\newcommand{\db}{\overline{\partial}}
\newcommand{\tE}{\widetilde{E}}
\newcommand{\tr}{\mbox{tr}}
\newcommand{\ta}{\widetilde{\alpha}}
\newcommand{\tb}{\widetilde{\beta}}
\newcommand{\txi}{\widetilde{\xi}}
\newcommand{\hV}{\hat{V}}
\newcommand{\IC}{\mathbf{C}}
\newcommand{\IZ}{\mathbf{Z}}
\newcommand{\IP}{\mathbf{P}}
\newcommand{\IR}{\mathbf{R}}
\newcommand{\IH}{\mathbf{H}}
\newcommand{\IG}{\mathbf{G}}
\newcommand{\CC}{{\mathcal C}}
\newcommand{\CS}{{\mathcal S}}
\newcommand{\CG}{{\mathcal G}}
\newcommand{\CL}{{\mathcal L}}
\newcommand{\CO}{{\mathcal O}}
\nc{\ca}{{\mathcal A}} \nc{\cag}{{{\mathcal A}^\Gamma}}
\nc{\cg}{{\mathcal G}} \nc{\chh}{{\mathcal H}} \nc{\ck}{{\mathcal
B}} \nc{\cl}{{\mathcal L}} \nc{\cm}{{\mathcal M}}
\nc{\cn}{{\mathcal N}} \nc{\cs}{{\mathcal S}} \nc{\cz}{{\mathcal
Z}} \nc{\cM}{{\mathcal M}}
\nc{\sind}{\sigma{\rm -ind}}
\newcommand{\la}{\langle}
\newcommand{\ra}{\rangle}

\renewcommand{\labelitemi}{{}}

\def\title#1{\begin{center}\bf\large #1\end{center}\vskip 0.3 in}
\def\author#1{\begin{center}#1\end{center}}

\title{Families of Type {\rm III KMS} States on a Class of $C^*$-Algebras
containing $O_n$ and $\mathcal{Q}_\N$}

\centerline {A.L. Carey$^a$, J. Phillips$^{b,c}$, 
I.F. Putnam$^{b}$,
A. Rennie$^{a}$}
\vspace*{0.1in}

\centerline {$^a$ Mathematical Sciences Institute, Australian National 
University,
Canberra, ACT, AUSTRALIA}

\centerline{ $^b$ Department of Mathematics and Statistics,
University of Victoria,
Victoria, BC,  CANADA}

\centerline{ $^c$ Corresponding Author. email: johnphil@uvic.ca;
telephone: 250-721-7450; fax: 250-721-8962}
\email{ acarey@maths.anu.edu.au, phillips@math.uvic.ca, putnam@math.uvic.ca,
        rennie@maths.anu.edu.au}

\centerline{\bf Abstract}
We construct a family of purely infinite $C^*$-algebras, $\mathcal{Q}^\lambda$ 
for $\lambda\in (0,1)$ 
that are classified by their $K$-groups.
There is an action of the circle 
$\T$ with a unique ${\rm KMS}$ state $\psi$ on each $\mathcal{Q}^\lambda.$
For $\lambda=1/n,$ $\mathcal{Q}^{1/n}\cong O_n$, with its 
usual $\T$ action and ${\rm KMS}$ state. 
For $\lambda=p/q,$  
rational in lowest terms, $\mathcal{Q}^\lambda\cong O_n$ ($n=q-p+1$) with  
UHF fixed point algebra of type $(pq)^\infty.$ For any $n>0,$ 
$\mathcal{Q}^\lambda\cong O_n$ for infinitely many $\lambda$ 
with distinct KMS 
states
and UHF fixed-point algebras. For any $\lambda\in (0,1),$ 
$\mathcal{Q}^\lambda\neq O_\infty.$ For $\lambda$ irrational
the fixed point algebras, are NOT AF and the $\mathcal{Q}^\lambda$ are usually 
NOT Cuntz algebras. For $\lambda$ transcendental, 
$K_1\cong K_0\cong\Z^\infty$, so that 
$\mathcal{Q}^\lambda$ is Cuntz' $\mathcal Q_{\N}$, \cite{Cu1}. If 
$\lambda^{\pm 1}$ are both algebraic integers, 
the {\bf only} $O_n$ which appear satisfy $n\equiv 3(mod\;4).$  
For each $\lambda$, 
the representation of $\mathcal{Q}^\lambda$ defined by the KMS state $\psi$ 
generates a type 
${\rm III}_\lambda$ factor. These algebras fit into the framework of modular 
index (twisted cyclic) theory of \cite{CPR2,CRT} and \cite{CNNR}.

\vspace{.05in}
{\bf Keywords:} KMS state, $III_\lambda$ factor, modular index,
twisted cyclic theory, K-Theory. 

\vspace{.05in}
AMS Classification codes: 46L80, 58J22, 58J30.
\section{Introduction}

In this paper we introduce some new examples of KMS states on a large
class of purely infinite $C^*$-algebras that were
motivated by the `modular index theory' of  \cite{CPR2,CNNR}. We were aiming to 
find examples of algebras that were not Cuntz-Krieger algebras (or the CAR 
algebra)
and were not previously known
in order to explore the possibilities opened by \cite{CNNR}. These algebras,
denoted by $\mathcal{Q}^\lambda$ for $0<\lambda<1$, 
are not constructed as graph algebras, but as
``corner algebras'' of certain crossed product $C^*$-algebras.
The $\mathcal{Q}^\lambda$ have similar structural properties to the Cuntz algebras,  
however
there are important new features, such as

1) when $\lambda=p/q$ is rational in lowest terms, then
$\mathcal{Q}^\lambda\cong O_{q-p+1}$ as mentioned in the Abstract,

2) when $\lambda$ is algebraic, the $K$-groups depend on the minimal polynomial
(and its coefficients) 
of 
$\lambda$,

3) when $\lambda$ is transcendental, $\mathcal{Q}^\lambda\cong \mathcal{Q}_\N$, 
Cuntz' 
algebra, \cite{Cu1}. 

We prove in Section 3 that the $\mathcal{Q}^\lambda$ are purely infinite, simple, 
separable, nuclear $C^*$-algebras, so there is no nontrivial trace
on them. Also in Section 3 we determine in many cases the $K$-groups of these 
algebras and use classification
theory to identify them when these algebras have the same $K$-groups as others 
found previously (these facts are summarised in the Abstract). 
As each $\mathcal{Q}^\lambda$ comes equipped with a gauge action of the circle,
our results thus give an uncountable family of distinct circle actions on
$\mathcal Q_{\N}$, each with its own unique KMS state. Indeed, for all 
$0<\lambda<1$,
we find a unique
KMS state, \cite{BR2}, for this gauge action, and we prove in Section 4 that the 
GNS representation of $\mathcal {Q}^\lambda$ associated to our KMS state generates 
a type ${\rm III}_\lambda$ von Neumann algebra.
The result of \cite{CPR2} that motivated this paper was the construction of 
a `modular spectral triple' with which one may
compute an index pairing using the KMS state. 
In \cite{CNNR} it was shown 
how modular spectral triples arise naturally for KMS states of circle actions 
and lead to
`twisted residue cocycles'
using a variation on the semifinite residue cocycle
of \cite{CPRS2}. 
It is well known that such twisted cocycles can not pair with ordinary $K_1$. 
In \cite{CPR2, CRT}  a substitute was introduced which is called `modular $K_1$'.
The correct definition of modular $K_1$ was found in \cite{CNNR}, and there is
a general spectral flow formula which defines
the pairing of modular $K_1$
with our `twisted residue cocycle'.

There is a strong analogy with the local index formula
of noncommutative geometry 
in the $\LL^{1,\infty}$-summable
 case, however, there are important
differences: the usual residue cocycle is replaced
by a twisted residue cocycle and the Dixmier trace 
arising in the standard
situation is replaced by a KMS-Dixmier functional. 
The common ground with \cite{CPRS2} stems from the
use of the spectral flow formula of \cite{CP2} to derive
the twisted residue cocycle and this has the corollary that we have a homotopy 
invariant.
To illustrate the theory for these examples we compute, for particular modular 
unitaries
in matrix algebras over the algebras $\mathcal{Q}^\lambda$, the precise 
numerical values arising from the general formalism. 
\section{The algebras $\mathcal{Q}^\lambda$ for $0<\lambda<1$.}
\label{lambda-alg}
\subsection{The  $C^*$-algebras $C^*(\Gamma_\lambda) = 
C(\hat{\Gamma}_\lambda)$ and their $K$-theory}
We will construct our algebras $\mathcal{Q}^\lambda$ as ``corner'' algebras in
certain crossed product $C^*$-algebras but first we need some preliminaries.
For $0<\lambda<1$, let $\Gamma_\lambda$ be the countable additive
abelian subgroup of $\R$ defined by:
$$
\Gamma_\lambda=\left\{\left.\sum_{k=-N}^{k=N} n_k\lambda^{k}\;
\right|\; N\geq 0\;\;{\rm and}\;\; n_k\in\Z
\right\}.
$$
Loosely speaking, $\Gamma_\lambda$ consists of Laurent
polynomials in $\lambda$ and $\lambda^{-1}$ with integer coefficients.
It is not only a dense subgroup of $\R$, but is clearly a unital 
{\bf subring} of $\R.$

\begin{prop}\label{Zgroups}
Let $0<\lambda<1.$\\
{\rm (1)}\hspace{.1in} If
$\lambda= p/q$ where $0<p<q$ are integers in lowest terms,
then $\Gamma_\lambda= \Z[1/n],$ where $n=pq.$\\
{\rm (2)}\hspace{.1in} If $\lambda$ and $\lambda^{-1}$ are both algebraic 
integers, then 
$\Gamma_\lambda=\Z+\Z\lambda+\cdots+\Z\lambda^{d-1}$ is an internal direct sum
where $d\geq 2$ is the
degree of the minimal (monic) polynomial in $\Z[x]$ 
satisfied by $\lambda.$\\
{\rm (3)}\hspace{.1in} If $\lambda$ is transcendental then,
$\Gamma_\lambda=\bigoplus_{k\in\Z} \Z\lambda^k$ is an internal direct sum.\\
{\rm (4)}\hspace{.1in} If $\lambda=1/\sqrt{n}$ with $n\geq 2$ a 
square-free positive integer,
then $\Gamma_\lambda=\Z[1/n]+\Z[1/n]\cdot\sqrt{n}$ is an internal direct sum.\\
{\rm (5)}\hspace{.1in} In general, if $\lambda$ is algebraic with minimal
polynomial, $n\lambda^d+\cdots+m=0$ over $\Z,$ then
$$\Z\oplus\Z\lambda\oplus\cdots\oplus\Z\lambda^{d-1}\subseteq\Gamma_\lambda
\subseteq\Z[\frac{1}{mn}]
\oplus\Z[\frac{1}{mn}]\lambda\oplus\cdots\oplus\Z[\frac{1}{mn}]\lambda^{d-1}.$$
Hence, $\rank(\Gamma_\lambda):=\dim_{\Q}(\Gamma_\lambda\otimes_{\Z}\Q)=
d.$
\end{prop} 

\begin{proof}
In case {\rm (1)}, since $gcd(p,q)=1,$ there exist $a,b\in\Z$
so that $1=ap + bq.$ Therefore,
$\frac{1}{q} = \frac{ap+bq}{q}= a\lambda +b\in
\Gamma_\lambda;$ and similarly, $\frac{1}{p}
\in\Gamma_\lambda.$
Since, $\Gamma_\lambda$ is a commutative ring, for any $k,m\in\Z$ 
with $k\geq 1$ we have:
$\frac{m}{n^k}=\frac{m}{(pq)^k}$ is in $\Gamma_\lambda.$
That is, $\Z[1/n]\subseteq\Gamma_\lambda.$ On the other hand, for
$k\geq 1$ we have 
$$
\lambda^k=\frac{p^k}{q^k}=p^{2k}\frac{1}{(pq)^k}=p^{2k}\frac{1}{n^k}
\in\Z[1/n]\;\;{\rm and}\;\;\lambda^{-k}=\frac{q^k}{p^k}=q^{2k}\frac{1}{(pq)^k}
=q^{2k}\frac{1}{n^k}\in\Z[1/n].
$$
That is, $\Z[1/n]=\Gamma_\lambda.$ 

In case {\rm (2)}, it is not hard to see the minimal
polynomial of $\lambda$ in $\Z[x]$ is not only monic, but also has
constant term $=\pm 1;$
say, $p(\lambda)= \lambda^d + a\lambda^{d-1} + \cdots \pm 1=0.$ Clearly,
$\lambda\in \Z+\Z\lambda+\cdots +\Z\lambda^{d-1}.$ Since 
$\lambda^{-1} p(\lambda) =0,$ we also have
$\lambda^{-1}\in \Z+\Z\lambda+\cdots +\Z\lambda^{d-1}.$ By an easy induction,
we have $\lambda^k\in \Z+\Z\lambda+\cdots +\Z\lambda^{d-1},$ for all
$k\in \Z.$ Hence, 
$\Gamma_\lambda=\Z+\Z\lambda+\cdots+\Z\lambda^{d-1}.$ The sum is 
direct by the minimality of the degree of the minimal polynomial.

In case {\rm (3)} the sum is direct because if $\lambda$
satisfied a Laurent polynomial over $\Z$, then by multipling by a high power
of $\lambda$ it would also satisfy a genuine polynomial over $\Z.$

Case {\rm (4)} is an easy calculation which we leave to the 
reader. Case {\rm (5)} is proved by similar methods used in case {\rm (2)}.
Again, the sum $\Z[\frac{1}{mn}]
+\Z[\frac{1}{mn}]\lambda +\cdots +\Z[\frac{1}{mn}]\lambda^{d-1}$
is direct by the minimality of the degree of the minimal polynomial.
\end{proof}

\begin{prop}\label{hatKtheory}
Let $0<\lambda<1.$\\
{\rm (1)}\hspace{.1in} If $\lambda =p/q$ is rational in lowest terms so that
$\Gamma_\lambda= \Z[1/n],$ where $n=pq,$ then
$$
K_0(C(\hat{\Gamma}_\lambda))=\Z[1_{\hat{\Gamma}_\lambda)}]\;\;\;{\rm and}
\;\;\; 
K_1(C(\hat{\Gamma}_\lambda)) = \Z[1/n].
$$
{\rm (2)}\hspace{.1in} If $\lambda$ and $\lambda^{-1}$ are both algebraic 
integers, so that 
$\Gamma_\lambda=\Z+\Z\lambda+\cdots+\Z\lambda^{d-1}$ is an internal direct sum
as above, then 
$$
K_0(C(\hat{\Gamma}_\lambda))=\bigwedge^{even}(\Gamma_\lambda)
={\bigoplus_{k=0, k\; even}^d} \bigwedge^k(\Gamma_\lambda)\;\;\;
{\rm and}\;\;\; K_1(C(\hat{\Gamma}_\lambda))=\bigwedge^{odd}(\Gamma_\lambda)
=\bigoplus_{k=1, k\; odd}^d \bigwedge^k(\Gamma_\lambda).
$$
{\rm (3)}\hspace{.1in} If $\lambda$ is transcendental then,
$$
K_0(C(\hat{\Gamma}_\lambda))=\bigwedge^{even}(\Gamma_\lambda)
={\bigoplus_{k=0, k\; even}^\infty} \bigwedge^k(\Gamma_\lambda)\;\;\;
{\rm and}\;\;\; K_1(C(\hat{\Gamma}_\lambda))=\bigwedge^{odd}(\Gamma_\lambda)
=\bigoplus_{k=1, k\; odd}^\infty \bigwedge^k(\Gamma_\lambda).
$$
{\rm (4)}\hspace{.1in} If $\lambda=1/\sqrt{n}$ with $n\geq 2$ a 
square-free positive integer,
then 
$$
K_0(C(\hat{\Gamma}_\lambda))\cong \Z\oplus\Z[1/n]\;\;\;{\rm and}\;\;\;
K_1(C(\hat{\Gamma}_\lambda))\cong \Z[1/n]\oplus\Z[1/n]    
$$
{\rm (5)}\hspace{.1in} In general, if $\lambda$ is algebraic with
$n\lambda^d+\cdots+m=0$ over $\Z$ then the composition of the inclusions
$$
\Z\oplus\Z\lambda\oplus\cdots\oplus\Z\lambda^{d-1}
\subseteq\Gamma_\lambda\subseteq\Z[\frac{1}{mn}]
\oplus\Z[\frac{1}{mn}]\lambda\oplus\cdots\oplus\Z[\frac{1}{mn}]\lambda^{d-1}
$$
induces an inclusion on $K$-Theory, so that both of the following maps are 
one-to-one
$$
\bigwedge^{even}(\Z^d)\cong K_0(C^*(\Z\oplus\cdots\Z\lambda^{d-1}))
\hookrightarrow 
K_0(C(\hat{\Gamma}_\lambda))\;\;{\rm and}\;\; 
\bigwedge^{odd}(\Z^d)\cong K_1(C^*(\Z\oplus\cdots\Z\lambda^{d-1}))
\hookrightarrow 
K_1(C(\hat{\Gamma}_\lambda)).
$$
\end{prop}

\begin{proof}
In case (1), $\Gamma_\lambda = \underrightarrow{\lim}\,\Z$ 
where each map 
is multiplication by $n,$ so that $\hat{\Gamma}_\lambda
=\underleftarrow{\lim}\, \T.$ Since $K_0(C(\T))=\Z[1]$ 
is generated by multiples of the 
trivial rank one bundle, the maps in the direct limit
$K_0(C(\hat{\Gamma}_\lambda)=\underrightarrow{\lim}\,K_0(C(\T))$ are
the identity map in each case, so that $K_0(C(\hat{\Gamma}_\lambda))=\Z[1].$
On the other hand, $K_1(C(\T))$ is generated by the maps on $C(\T)$ 
$z\mapsto z^k$, and each map in the direct limit is the same map induced by
$z\mapsto z^n.$ Thus, $K_1(C(\hat{\Gamma}_\lambda))=\Z[1/n].$ 

Cases (2) and (3) are well-known facts about the $K$-Theory of
tori.

Case (4): first one uses item (4) of the previous Proposition, 
then the proof of case (1) above in order to apply Proposition 2.11 of 
\cite{Sc}. The proof is finished off with the easily proved observation 
that $\Z[1/n]\otimes\Z[1/n] = \Z[1/n].$

Case (5) the composed embedding is just containment:\\ 
$\Z\oplus\Z\lambda\oplus\cdots\oplus\Z\lambda^{d-1}
\subseteq\Z[\frac{1}{mn}]
\oplus\Z[\frac{1}{mn}]\lambda\oplus\cdots\oplus\Z[\frac{1}{mn}]\lambda^{d-1}.$
Since we know that $K_*(C^*(\Z))\to K_*(C^*(\Z[1/mn]))$ is one-to-one
(even an isomorphism after tensoring with $\Q$),
an application of C. Schochet's K\"{u}nneth Theorem, \cite{Sc} shows that 
the induced map on $K$-Theory:
$$K_*(C^*(\Z\oplus\Z\lambda\oplus\cdots\oplus\Z\lambda^{d-1}))\longrightarrow
K_*(C^*(\Z[\frac{1}{mn}]\oplus\Z[\frac{1}{mn}]\lambda\oplus\cdots\oplus
\Z[\frac{1}{mn}]\lambda^{d-1}))$$
is one-to-one (even an isomorphism after tensoring with $\Q$). 
\end{proof}

\begin{cor}
If $\lambda$ is algebraic with minimal polynomial of degree $d$ so that 
${\rm rank}(\Gamma_\lambda) = d$ then
$${\rm rank}(K_0(C(\hat{\Gamma}_\lambda)))={\rm rank}(\bigwedge^{even}(\Z^d))
=2^{d-1}={\rm rank}(\bigwedge^{odd}(\Z^d))={\rm rank}
(K_1(C(\hat{\Gamma}_\lambda))).$$
\end{cor}
\begin{proof}
For each $N\geq d-1,$ let $\Gamma_N=\Z\lambda^{-N}+\cdots+\Z\lambda^{N}
\subseteq
\Gamma_\lambda.$ Then each $\Gamma_N$ is a finitely generated torsion free 
(and hence free abelian) subgroup of $\Gamma_\lambda.$ Moreover,
$$\Z\oplus\Z\lambda\oplus\cdots\oplus\Z\lambda^{d-1}
\subseteq\Gamma_N\subseteq\Gamma_\lambda\subseteq\Z[\frac{1}{mn}]
\oplus\Z[\frac{1}{mn}]\lambda\oplus\cdots\oplus\Z[\frac{1}{mn}]\lambda^{d-1},$$ 
so that by tensoring with $\Q$ the induced inclusions are all equalities,
and hence all are $\Q$-vector spaces of dimension $d.$ 
Since $\Gamma_N$ is free abelian, $\Gamma_N\cong\Z^d.$
 Now,
$$K_0(C^*(\Gamma_N))\cong K_0(C(\T^d))\cong\bigwedge^{even}(\Z^d)
\cong\Z^{2^{d-1}}
\;\;\;{\rm and}\;\;\;K_1(C^*(\Gamma_N))\cong K_1(C(\T^d))
\cong\bigwedge^{odd}(\Z^d)\cong
\Z^{2^{d-1}}.$$ So, each $K_i(C^*(\Gamma_N))\otimes_{\Z}\Q$ is a 
$\Q$-vector space of dimension $2^{d-1}$ and the map: 
$$K_*(C^*(\Z\oplus\Z\lambda\oplus\cdots\oplus\Z\lambda^{d-1}))\otimes_{\Z}\Q
\longrightarrow K_*(C^*(\Gamma_N))\otimes_{\Z}\Q$$ is one-to-one and hence an
isomorphism of $\Q$-vector spaces. Since the corresponding isomorphism onto
$K_*(C^*(\Gamma_{N+1}))\otimes_{\Z}\Q$ factors through 
$K_*(C^*(\Gamma_N))\otimes_{\Z}\Q$ the maps 
$$K_*(C^*(\Gamma_N))\otimes_{\Z}\Q\to K_*(C^*(\Gamma_{N+1}))\otimes_{\Z}\Q$$
are all isomorphisms. Now, $C^*(\Gamma_\lambda)=\lim_N C^*(\Gamma_N)$ 
and so $K_i(C^*(\Gamma_\lambda))=\lim_N K_i(C^*(\Gamma_N)),$ and therefore,
$$K_i(C^*(\Gamma_\lambda))\otimes_{\Z}\Q=
\lim_N K_i(C^*(\Gamma_N))\otimes_{\Z}\Q \cong \Q^{2^{d-1}}$$
for each $i=1,2.$
\end{proof}

Now, let $G_\lambda\supset G^0_\lambda$ be the following countable discrete 
groups of matrices:
$$
G_\lambda =\left\{\left.\bma \lambda^n & a\\0 & 1 \ema\;\right|\; 
a\in \Gamma_\lambda,\;n\in\Z\right\}\;\supset\;
G^0_\lambda =\left\{\left.\bma 1 & a\\0 & 1 \ema\;\right|\; 
a\in \Gamma_\lambda\right\}.
$$
Of course, $G^0_\lambda$ is isomorphic to the additive group
$\Gamma_\lambda$, and $G_\lambda$ is semidirect product of $\Z$ acting on
$G^0_\lambda\cong\Gamma_\lambda.$
We let $G_\lambda$ act on $\R$ as an ``ax+b'' group, noting that the action 
leaves $\Gamma_\lambda$ invariant. That is,
$$
\text{for}\;\;t\in\R\;\;\text{and}\;\;g=\bma \lambda^n & a\\0 & 1 \ema
\in G_\lambda\;\;\text{define}\;\; g\cdot t:=\lambda^n  t + a.
$$
\noindent{\bf Notation.} For such an element $g\in G_\lambda$ we will use the 
notation 
$g:=[\lambda^n\, :\,a]$ in place of 
the matrix for $g$ and $|g|:=\det(g)=\lambda^n$ 
for the determinant of $g.$ Note: $G^0_\lambda=\{g\in G_\lambda\;|\;|g|=1\}
\lhd G_\lambda.$ 

We use this action on $\R$ to define the transpose action $\alpha$ of 
$G_\lambda$ on $\mathcal{L}^\infty(\R):$
$$
\alpha_g(f) (t)=f(g^{-1}t)\;\;\text{for}\;\;f\in\mathcal{L}^\infty(\R)\;\;
\text{and}\;\; t\in\R.
$$

Now let $C_0^\lambda(\R)$ be the separable $C^*$-subalgebra of 
$\mathcal{L}^\infty(\R)$ generated by the countable family of projections
$\mathcal X_{[a,b)}$ where $a,b\in\Gamma_\lambda.$ That is,

$$
C_0^\lambda(\R)={\rm closure}\left(
\left\{\left.\sum_{k=1}^{n}c_k\mathcal X_{[a_k,b_k)}\;\right|
\;c_k\in\C;\;\;a_k,b_k\in\Gamma_\lambda\right\}\right).
$$
We observe that $C_0^\lambda(\R)$ is a commutative AF-algebra.
Clearly, $C_0(\R)\subset C_0^\lambda(\R)$ and since 
$\alpha_g(\mathcal X_{[a,b)})=\mathcal X_{[g(a),g(b))}$ both are invariant 
under 
the action $\alpha$ of $G_\lambda.$ We define the separable $C^*$-algebras
$A^\lambda\supset A_0^\lambda$  as the crossed products:
$$
A^\lambda = G_\lambda \rtimes_\alpha C_0^\lambda(\R)
=\Z\rtimes (G^0_\lambda \rtimes_\alpha C_0^\lambda(\R))\;\supset\;
A_0^\lambda = G^0_\lambda \rtimes_\alpha C_0^\lambda(\R).
$$
Since $G_\lambda$ and $G^0_\lambda$ are amenable these equal 
the reduced crossed products by \cite[Theorem 7.7.7 ]{Ped}. 
Let $C_{00}^\lambda(\R)$ denote the dense $*$-subalgebra of 
$C_0^\lambda(\R)$ consisting of finite linear combinations of the generating 
projections, $\mathcal X_{[a,b)},$ and let 
$A_c^\lambda\subset l^1_\alpha(G_\lambda,C_0^\lambda(\R))\subset A^\lambda$ 
denote the dense
$*$-subalgebra of $A^\lambda$ consisting of finitely supported functions
$x:G_\lambda\to C_{00}^\lambda(\R).$ Similarly we define 
$A_{0,c}^\lambda\subset A_0^\lambda.$
\begin{prop}\label{bootstrap1}
For any $\lambda\in (0,1)$ 
$A_0^\lambda$ and $A^\lambda$ are in the bootstrap class $\mathfrak N_{nuc}.$
\end{prop}
\begin{proof}
Since $A^\lambda=\Z\rtimes A_0^\lambda,$ it suffices to see that $A_0^\lambda$
is in $\mathfrak N_{nuc}.$ By the proof of the previous Corollary,
we can write $\Gamma_\lambda$
as an increasing union of finitely generated torsion-free abelian groups
$\Gamma_N$ which are free abelian group of finite rank so that 
$A_0^\lambda$ is the direct limit of crossed products of the separable
commutative $C^*$-algebra $C_0^\lambda(\R)$ by $\Z^{m_i}$ and
hence is in $\mathfrak N_{nuc}.$ 
\end{proof}
\noindent{\bf Notation:}
We remind the reader of the crossed product operations in our setting 
(Definition 7.6.1 of \cite{Ped}) together 
with some particular notations we use. To this end, let 
$x,y\in l^1_\alpha(G_\lambda,C_0^\lambda(\R))$ then we have the product and 
adjoint formulas:
\bean
(x\cdot y)(g)&=&\sum_{h\in G_\lambda}x(h)\alpha_h(y(h^{-1}g))
\;\;\text{for}\;\;
g\in G_\lambda;\\
x^*(g)&=&\alpha_g((x(g^{-1}))^*)\;\;\text{for}\;\;
g\in G_\lambda.
\eean

If $x\in l^1_\alpha(G_\lambda,C_0^\lambda(\R))$ is supported on the single
element $g\in G_\lambda$ and $x(g)=f\in C_0^\lambda(\R)$, then we write
$x=f\cdot\delta_g$. Since $A_c^\lambda$ (respectively, $A_{0,c}^\lambda$)  
is dense in $A^\lambda$ (respectively, $A_0^\lambda$) we often do our 
calculations with these elements and we have the
following easily verified calculus for them.

\begin{lemma}\label{lambdacalc}
Let $f_1\cdot\delta_{g_1}, f_2\cdot\delta_{g_2},f\cdot\delta_g\in A_c^\lambda,$ 
then:\\
{\rm (1)}\hspace{.1in} $(f_1\cdot\delta_{g_1})\cdot(f_2\cdot\delta_{g_2})=
f_1\alpha_{g_1}(f_2)\cdot\delta_{g_1 g_2}$\\
{\rm (2)}\hspace{.1in} $(f\cdot\delta_g)^*=
\alpha_{g^{-1}}(\bar{f})\cdot\delta_{g^{-1}}.$\\
{\rm (3)}\hspace{.1in} $f\cdot\delta_g$ is self-adjoint if and only if $f$ is
self-adjoint and $g=1.$\\
{\rm (4)}\hspace{.1in} $f\cdot\delta_g$ is a projection if and only if $f$ is
a projection and $g=1.$\\
{\rm (5)}\hspace{.1in} $f\cdot\delta_g$ is a partial isometry if and only if 
$|f|$ is a projection.\\
{\rm (6)}\hspace{.1in} The product of partial isometries of the form
$\mathcal X_{[a,b)}\cdot\delta_g$ is a partial isometry of the same form.\\
{\rm (7)}\hspace{.1in} Consider the partial isometry,
$v=\mathcal X_{[a,b)}\cdot\delta_g.$ Any two of the following: $vv^*,\;v^*v,\;g$
completely determine the interval $[a,b)$ and the element $g.$ 
\end{lemma}

\begin{defn}
Let $e\in A_{0,c}^\lambda$ be the projection 
$e=\mathcal X_{[0,1)}\cdot\delta_1.$
We define the separable unital $C^*$-algebras:
$$
\mathcal{Q}^\lambda:=eA^\lambda e\supset eA_0^\lambda e =:F^\lambda.
$$
We will also have occasion to use the dense subalgebras
$\mathcal Q_c^\lambda:=eA_c^\lambda e,$ and 
$F_c^\lambda:= eA_{0,c}^\lambda e.$
\end{defn}

\begin{prop}\label{stable}
The orthogonal family of projections 
$e_n=\mathcal X_{[n,n+1)}\cdot\delta_1 \in A_0^\lambda$
for $n\in\Z$ are mutually equivalent by partial isometries in $A_0^\lambda$
of the form $V_{n,k}:=\mathcal X_{[n,n+1)}\cdot\delta_{g_{n-k}}$
where $g_{n-k}=[1\;:\;(n-k)].$ Moreover, the finite sums 
$E_N:=\sum_{n=-N}^{N-1} e_n=\mathcal X_{[-N,N)}\cdot\delta_1$ form an 
approximate identity for $A^\lambda$
so that 
$$
A^\lambda\cong \mathcal{Q}^\lambda\otimes \mathcal K(l^2(\Z)) 
\;\;{\rm and}\;\;
A_0^\lambda\cong F^\lambda\otimes\mathcal K(l^2(\Z)).
$$
\end{prop}

\begin{proof}
By Lemma \ref{lambdacalc}, one easily calculates that:
$$
\text{for\;\;each\;\;pair}\;\;n,k\in\Z,\;\;V_{n,k}V_{n,k}^*=e_n\;\;\text{and}
\;\;V_{n,k}^*V_{n,k}=e_k.
$$  
Now for each positive integer $N$ if we have
$y\in A^\lambda_c$ that satisfies $\text{supp}(y_h)\subseteq [-N,N)$ 
for all $h,$ then using Lemma \ref{lambdacalc} again we see that $E_N\cdot y=y.$
Since the collection of all such
elements $y\in A^\lambda_c$ is dense in $A^\lambda$, we see that the 
increasing sequence of projections $\{E_N\}$ form an approximate identity for
$A^\lambda.$
\end{proof}

\begin{cor}\label{bootstrap2}It follows from Proposition 2.4.7 of \cite{RS}
and Proposition \ref{bootstrap1} that for any $\lambda\in (0,1),$ 
$\mathcal{Q}^\lambda$ and $F^\lambda$ are both in $\mathfrak N_{nuc}.$
\end{cor}

\begin{lemma}(cf. \cite[Proposition 3.1, Lemma 3.6]{PhR})
The algebra $C_0^\lambda(\R)$ is a commutative separable AF algebra consisting 
of all
functions $f:\R\to\C$ which vanish at $\infty$ and:
are right continuous at each $x\in \Gamma_\lambda$; have a finite 
left-hand limit 
at each $x\in \Gamma_\lambda;$ and are continuous at each
$x\in (\R\setminus\Gamma_\lambda).$ Moreover, if 
$\phi\in\widehat{C_0^\lambda(\R)},$ (the space 
of all nonzero $*$-homomorphisms: $C_0^\lambda(\R)\to\C$) then there exists
a unique $x_0\in\R$ such that:\\
{\rm (1)} if $x_0\in(\R\setminus\Gamma_\lambda)$ \;then\; 
$\phi(f)=f(x_0)$\;
for all\; $f\in C_0^\lambda(\R),$\\ 
{\rm (2)} if $x_0\in\Gamma_\lambda$ then either
$\left\{\begin{array}{ll} \phi(f)=f(x_0) \;\;\mbox{for all} \;\;
f\in C_0^\lambda(\R),\;\;\mbox{or}\\ 
\phi(f)=f^-(x_0)=\lim_{x\to x_0^-} f(x)\;\;
\mbox{for all}\;\; f\in C_0^\lambda(\R). \end{array}\right.$
\end{lemma}
\begin{proof}
Since generating functions for $C_0^\lambda(\R)$ satisfy each of
the properties above which are clearly preserved by passing to 
uniform limits, we see that any function in $C_0^\lambda(\R)$ satisfies these
properties. Conversely, it is easy to show that any function 
satisfying these properties can be 
uniformly approximated by a finite linear combination of the generators.
The remainder of the proof is given in \cite[Lemma 3.6]{PhR}.
\end{proof}

\noindent{\bf Notation.} We denote the dual space,
$\widehat{C_0^\lambda(\R)}$ by $\R_\lambda$ and endow it with
the relative weak-$*$ topology, 
that is the topology of 
pointwise convergence on $C_0^\lambda(\R).$ Of course, $\R_\lambda$
is a locally compact Hausdorff space, and 
$C_0^\lambda(\R)\cong C_0(\R_\lambda).$

\begin{prop}\label{simple}
The algebras
$A^\lambda$ and $A^\lambda_0$ (and hence $\mathcal{Q}^\lambda$ and $F^\lambda$)
are simple $C^*$-algebras. 
Moreover, $A^\lambda$ is purely infinite and hence so is $\mathcal{Q}^\lambda$. 
\end{prop}

\begin{proof}
Now, both $G_\lambda$ and $G_\lambda^0$ act on $C_0^\lambda(\R)$ 
as countable discrete groups of outer automorphisms. Thus, we can apply
Theorem 3.2 of \cite{E} once we check that neither action has any nontrivial
invariant ideals in $C_0^\lambda(\R)$ and that the actions are {\bf properly
outer} in the sense of Definition 2.1 of \cite{E}. 

To do this we look at the induced action of $G_\lambda$ and 
$G_\lambda^0$
on $\R_\lambda.$ 
So, for $g\in G_\lambda$ we have $g$ acting on $\R_\lambda$
via $g(\phi)=\phi\circ\alpha_g^{-1}$ so that for $\phi=\phi_x$
given by evaluation at $x\in\R,$ we have as expected 
$g(\phi_x)=\phi_{g(x)}.$ Now, if $x\in\Gamma_\lambda$ we use
the notation $\phi_{x_{-}}$ to denote the $*$-homomorphism 
$\phi_{x_{-}}(f)=f^-(x)=f(x_{-})=\lim_{y\to x_{-}}f(y).$ One easily checks 
that since
$g(x)\in \Gamma_\lambda,$ we have
$g(\phi_{x_{-}})=\phi_{g(x)_{-}}.$\\
Next we claim that each of the sets 
$\{\phi_m\;|\; m\in \Gamma_\lambda\}$ and
$\{\phi_{m_{-}}\;|\; m\in \Gamma_\lambda\}$ is dense in 
$\R_\lambda$ in the relative weak-$*$ topology.
For example, we show that the second set is dense.
To approximate $\phi_x$ for some $x\in\R$ we let $\{m_n\}$ be a sequence in
$\Gamma_\lambda$ converging to $x$ from the right in $\R.$
Let $f\in C_0^\lambda(\R)$ so that $f$ is right continuous at $x.$ 
One easily shows that $|\phi_{m_{n_{-}}}(f)-\phi_x(f)|\to 0;$ that is, 
the sequence $\{\phi_{m_{n_{-}}}\}$
converges to $\phi_x$ in the relative weak-$*$ topology.

It is easy to see that the action of $G_\lambda^0$ on 
$\R_\lambda$ has dense orbits, and so, of course, the 
action of $G_\lambda$ has dense orbits also. 
This implies that the actions of $G_\lambda^0$ and $G_\lambda$ 
on $C_0^\lambda(\R)$ have no nontrivial invariant ideals since 
the induced action on
$\R_\lambda$ has no nontrivial invariant closed sets. 
We complete the proof by 
showing that the action is {\bf properly outer} in the sense of
Definition 2.1 of \cite{E}. Since there are no 
nontrivial $\alpha$-invariant ideals and $C_0^\lambda(\R)$ is commutative this
is the condition that for each $g\neq 1$ and each nonzero closed
two sided ideal $\mathcal I$ invariant under $\alpha_g$ we have 
$\Vert(\alpha_g - Id)_{|_{\mathcal I}}\Vert=2.$ Since $\mathcal I$ is nonzero
there is a nonempty open subset, $\mathcal O$ of $\R_\lambda$
so that $\hat{\mathcal I} =\mathcal O.$
But since $g\neq 1$ and $\mathcal O$ is not finite there exists 
$y\in\mathcal O$ such that $g(y)\neq y$
and $g(y)\in \mathcal O.$ Let $x=g(y)\in\mathcal O$ so that 
$g^{-1}(x)=y\in\mathcal O$ and $x\neq g^{-1}(x).$
So we can choose a continuous compactly supported real-valued function
$f$ on $\mathcal O$ with $f(x)=1,$ $f(g^{-1}(x))=-1$ and $\Vert f\Vert=1.$ 
But then
$f\in \mathcal I$ and
$$
2\geq \Vert(\alpha_g - Id)_{|_{\mathcal I}}\Vert|\geq 
\Vert(\alpha_g - Id)(f)\Vert
= \Vert\alpha_g(f) - f\Vert\geq |f(g^{-1}(x))-f(x)|=2.
$$

Now that we know $A^\lambda$ is simple, we can easily apply
Theorem 9 of \cite{LS} to conclude that $A^\lambda$ satisfies hypothesis (v)
of Proposition 4.1.1 (page 66) of \cite{RS}. For simple $C^*$-algebras,
this is equivalent to being purely infinite by Definition 4.1.2 of \cite{RS}:
the authors of \cite{LS} had used one of the earlier definitions of purely
infinite in their paper (namely, hypothesis (v)). By Proposition 4.1.8 of
\cite{RS} $\mathcal{Q}^\lambda$ is also purely infinite.	 
\end{proof}

\begin{cor}\label{bootstrap3}It follows from Corollaries 8.2.2 and 8.4.1
(Kirchberg-Phillips) of \cite{RS} and the fact that $A^\lambda$ is stable
that for any $\lambda\in (0,1),$  $A^\lambda$ is classified up to isomorphism 
(among Kirchberg algebras in $\frak N_{nuc}$) by its $K$-theory.
\end{cor}
Since we need to calculate with elements of $\mathcal{Q}^\lambda$ and 
$F^\lambda,$ we make 
the following observations.

\begin{lemma}\label{lambdagens}
Now, $\mathcal{Q}^\lambda$ (respectively, $F^\lambda$)
is the norm closure of finite linear combinations of the elements
of the form
$e(\mathcal X_{[a,b)}\cdot\delta_g) e,$ where $g\in G_\lambda$ (respectively,
$g\in G^0_\lambda$), henceforth called the {\em generators}. Thus, we calculate\\
(1) If $f\cdot\delta_g\in A^\lambda,$ (respectively, 
$f\cdot\delta_g\in A_0^\lambda$) where $f\in C_0^\lambda(\R)$, then
$$
e(f\cdot\delta_g)e=\mathcal X_{[a,b)}f\cdot\delta_g\;\;\text{where}\;\;
[a,b)=[0,1)\cap[g(0),g(1)).
$$
(2) Thus, for $g\in G_\lambda,$ (respectively, $g\in G^0_\lambda$) 
$f\cdot\delta_g$ is in $\mathcal{Q}^\lambda$ (respectively, $F^\lambda$) iff 
$supp(f)\subseteq[0,1)\cap[g(0),g(1)).$ In particular, for $g\in G_\lambda,$ 
(respectively, $g\in G^0_\lambda$) 
$\mathcal X_{[a,b)}\cdot\delta_g$ is in $\mathcal{Q}^\lambda$ (respectively, 
$F^\lambda$)
iff 
$[a,b)\subseteq[0,1)\cap[g(0),g(1)).$
\end{lemma}

\begin{proof}
The first item is an easy calculation using part (1) of Lemma
\ref{lambdacalc}
and the fact that $\alpha_g(\mathcal X_{[a,b)})=\mathcal X_{[g(a),g(b))}.$
The second item follows easily from the first.
\end{proof}

\begin{prop}\label{AF}
If $\lambda$ is rational, then $A_0^\lambda$ and $F^\lambda$ are AF-algebras. 
In particular, if $\lambda= p/q$ where $0<p<q$ are in lowest terms,
then $F^\lambda$ is the UHF algebra $n^\infty$ where $n=pq.$ 
Moreover, the minimal projections
in the finite-dimensional subalgebras can all be chosen from the canonical
commutative subalgebra $C_0^\lambda(\R)\cdot\delta_I.$ 
\end{prop}

\begin{proof}
We have shown in Proposition \ref{Zgroups} that if $\lambda= p/q$ where $0<p<q$ 
are in lowest terms, then $\Gamma_\lambda= \Z[1/n],$ where $n=pq.$ 
 Now, any element in $\Z[1/n]$ has the form $m/n^k=m(1/n^k)$
where $k\geq 1.$ Therefore any of the generating partial isometries
$\mathcal X_{[a,b)}\cdot\delta_{[1:c]}\in A^\lambda_0$ can (by bringing
$a,b$ and $c$ to a common denominator) be written (assuming $c>0$) as a 
finite linear combination of partial isometries of the form
$\mathcal X_{[l/n^k,(l+1)/n^k)}\cdot\delta_{[1:1/n^k]}.$ For partial isometries
in $F^\lambda$ we would have to restrict $0\leq l\leq n^k-1$ and such 
partial isometries generate an $n^k\;{\rm by}\;n^k$ matrix subalgebra of
$F^\lambda.$ It should now be clear that $F^\lambda$ is a UHF algebra of type
$n^\infty.$
\end{proof}

At this point we define some special elements in $\mathcal{Q}^\lambda$ which 
behave
very much like the isometries $S_\mu\in O_n,$ except for the fact that 
{\bf some} of them are {\bf not} isometries.

\begin{defn}\label{Sk,m}
Fix $0<\lambda <1$ and let $k$ be a positive integer. Define $m_k$ to be the
unique positive integer satisfying: $m_k\lambda^k <1\leq (m_k +1)\lambda^k.$
For $0\leq m \leq m_k$ define {\bf partial isometries} 
$S_{k,m}\in \mathcal{Q}^\lambda$ 
via:
$$
S_{k,m}=\mathcal X_{[m\lambda^k,(m+1)\lambda^k)}\cdot\delta_{g_{k,m}}\;\;
{\rm where}\;\; g_{k,m}=[\lambda^k \; :\; m\lambda^k].
$$
Note: for $m<m_k$ the $S_{k,m}$ are actually {\bf isometries}, and
$S_{k,m_k}$ is an isometry iff $1=(m_k +1)\lambda^k.$
\end{defn}

\begin{rems*}
The defining inequalities $m_k\lambda^k <1\leq (m_k +1)\lambda^k$
for the positive integer $m_k$ are equivalent to:
$0<\lambda^{-k} - m_k \leq 1.$ In particular, these differences  are
positive and bounded above by $1$. In the case of $\mathcal{Q}^{1/n}$ we have
$m_k=n^k-1.$ Generally we have $m_1^k\leq m_k<1\leq(m_k+1)\leq (m_1+1)^k.$
\end{rems*}

\begin{lemma}\label{Smusub}
With the previously defined elements we have:
$$
S_{k,m}^*=\mathcal X_{[0,1)}\cdot\delta_{g^{-1}_{k,m}}\;\;
{\rm and}\;\;S_{k,m_k}^*=
\mathcal X_{[0,\lambda^{-k}-m_k)}\cdot\delta_{g^{-1}_{k,m_k}}\;\;
{\rm where\;\;for\;\;all\;\;}m,\;\; g^{-1}_{k,m}=[\lambda^{-k}\; :\;-m].
$$
Moreover, for $0\leq m<m_k,\; S_{k,m}^*S_{k,m}=
\mathcal X_{[0,1)}\cdot\delta_1=e$
while $S_{k,m_k}^*S_{k,m_k}=
\mathcal X_{[0,\lambda^{-k} -m_k)}\cdot\delta_1.$\\
Finally, for
$0\leq m<m_k,\; S_{k,m}S_{k,m}^*=
\mathcal X_{[m\lambda^k,(m+1)\lambda^k)}\cdot\delta_1$
while $S_{k,m_k}S_{k,m_k}^*=
\mathcal X_{[m_k\lambda^k,1)}\cdot\delta_1,$
so that 
$$
\sum_{m=0}^{m_k}S_{k,m}S_{k,m}^*=\mathcal X_{[0,1)}\cdot\delta_1 =e.
$$
\end{lemma}

\begin{proof}
These are just straightforward calculations based on Lemma \ref{lambdacalc}
which we leave to the reader.
\end{proof}

\begin{thm}\label{cuntzalg}
For each $\lambda$ with $0<\lambda<1,$ 
consider the partial isometries $S_{1,m}$ for $m=0,1,...,m_1$ where 
$m_1\lambda<1\leq (m_1 +1)\lambda.$ For $m<m_1,$ $S_{1,m}$ is an isometry and
$\sum_{m=0}^{m_1}S_{1,m}S^*_{1,m}=1.$
For $\lambda = 1/n$, $m_1=n-1,$ $S_{1,m_1}$ is also an isometry,  and 
$\mathcal Q^{1/n}\cong O_n,$ the usual Cuntz algebra. 
\end{thm}

\begin{proof}
The first statement is clear.
With $\lambda =1/n$ we have inside $\mathcal Q^{1/n},$ $n$ isometries
one for each $m=0,1,...,(n-1)$ defined by: 
$$
S_m=\mathcal X_{[\frac{m}{n},\frac{m+1}{n})}\cdot\delta_{g_m}\;\;\text{where}
\;\;g_m=[1/n\;:\;m/n]\;\;\text{and\;\;so}\;\;
S_m^*=\mathcal X_{[0,1)}\cdot\delta_{g_m^{-1}}\;\;\text{where}\;\;
g_m^{-1}=[n\; : \;-m].
$$
Using Lemma \ref{lambdagens}, we easily see that for each $m,$
$S_m\in \mathcal Q^{1/n}.$ 

Then, using item (1) of Lemma \ref{lambdacalc} we calculate:
$$
S_m^*S_m=\mathcal X_{[0,1)}\cdot\delta_1=e\;\;\text{and}\;\;
S_mS_m^*=\mathcal X_{[\frac{m}{n},\frac{m+1}{n})}\cdot\delta_1\;\;
\text{and\;\;so}\;\;\sum_{m=0}^{n-1} S_mS_m^*=
\mathcal X_{[0,1)}\cdot\delta_1=e.
$$
Since $e$ is the identity of $\mathcal Q^{1/n},$ we have constructed a 
unital copy of $O_n$ inside $\mathcal Q^{1/n}.$
Now one shows by induction that for each $k>0$ the product 
of exactly
$k$ of these $n$ isometries has the form $S_{k,m}$ where $S_{k,m}$
has the same defining equation as $S_m$ above but with $n^k$
in place of $n$ and $m=0,1,...,(n^k-1).$ These new isometries have
range projections 
$S_{k,m}S^*_{k,m} = \mathcal X_{[\frac{m}{n^k},\frac{m+1}{n^k})}\cdot\delta_1$
which therefore lie in this copy of $O_n.$ By adding up some of these 
projections,
we can get any projection of the form $\mathcal X_{[a,b)}\cdot\delta_1$
where $0\leq a<b\leq1$ and both $a,b$ have the form $m/n^k.$
But any element $a\in\Gamma_{1/n}$ can be written as
$a=\frac{m}{n^k}$ for a sufficiently large $k\geq 0$ and some $m\in\Z$
depending on $k,$ and any pair $a,b$ can be brought to a common denominator
$n^k.$ Hence any projection of the form $\mathcal X_{[a,b)}\cdot\delta_1$
in $\mathcal Q^{1/n}$ is in this copy of $O_n.$

Now, a straightforward calculation gives us:
\begin{eqnarray}
\label{SumS}\sum_{m=1}^{n^k -1} S_{k,m}S^*_{k,m-1}=
\sum_{m=1}^{n^k -1}\mathcal X_{
[\frac{m+1}{n^k},\frac{m+2}{n^k})}\cdot\delta_{[1\;:\;1/n^k]}
=\mathcal X_{
[1/n^k,1)}\cdot\delta_{[1\;:\;1/n^k]}\in O_n.
\end{eqnarray}

Finally, let $\mathcal X_{[a,b)}\cdot\delta_g \in \mathcal Q^{1/n}$ be 
an arbitrary
generator. By taking adjoints if necessary we can assume that $g$ has the
form $g=[n^k\; :\; *]$ where $k\geq0.$ Since $S_{k,0}$ is an isometry in
$O_n$ it suffices to prove that $S_{k,0}(\mathcal X_{[a,b)}\cdot\delta_g) \in 
O_n.$ That is, we are reduced to the case $g=[1\;:\; c]$ and again by taking 
adjoints if necessary we can assume that $c\geq 0.$ The case $c=0$ is done
and so we can assume that $c>0.$ So (with possibly new $a,b$)
we have $\mathcal X_{[a,b)}\cdot\delta_
{[1\;:\;c]}$ where $0<c\leq 1$ and $[a,b)\subseteq [0,1)\cap[c,c+1)=[c,1).$
But, $\mathcal X_{[a,b)}\cdot\delta_{[1\;:\;c]}=
\mathcal X_{[a,b)}\mathcal X_{[c,1)}\cdot\delta_{[1\;:\;c]}=
\mathcal X_{[a,b)}\cdot\delta_1\mathcal X_{[c,1)}\cdot\delta_{[1\;:\;c]}$
and we already know that $\mathcal X_{[a,b)}\cdot\delta_1\in O_n.$ Therefore
it suffices to see that $\mathcal X_{[c,1)}\cdot\delta_{[1\;:\;c]}\in O_n.$
However, $c=l/n^k$ for some $0<l<n^k$ and so:
$$
\mathcal X_{[c,1)}\cdot\delta_{[1\;:\;c]}=
\mathcal X_{[l/n^k,1)}\cdot\delta_{[1\;:\;l/n^k]}=
\left(\mathcal X_{[1/n^k,1)}\cdot\delta_{[1\;:\;1/n^k]}\right)^l
$$
which is in $O_n$ by Equation \ref{SumS}. Since all generators for 
$\mathcal Q^{1/n}$
are in $O_n$ we're done.
\end{proof}

\subsection{$K$-Theory of $\mathcal{Q}^\lambda$ for $\lambda$ 
rational}: 
Since $A^\lambda_0$ is stable and 
stably isomorphic to the UHF algebra $F_\lambda,$ each 
of its projections is equivalent to one in some finite-dimensional subalgebra
and hence to some projection in $C_0^\lambda(\R)$, and in this case
the trace induces an isomorphism from $K_0(A^\lambda_0)$ onto 
$\Gamma_\lambda=\Z[1/(pq)]\subset \R.$ This isomorphism carries the projection
$e=\mathcal X_{[0,1)}\cdot\delta_{1}$ which is the identity of 
$\mathcal{Q}^\lambda$
and $F^\lambda$ onto $1\in\Z[1/(pq)].$
Now, since $A^\lambda_0$ is AF, $K_1(A^\lambda_0)=\{0\},$ and since
$A^\lambda=\Z\rtimes_\lambda A^\lambda_0$ we can use the Pimsner-Voiculescu
exact sequence to calculate $K_*(A^\lambda)=K_*(\mathcal{Q}^\lambda).$ When we 
do this
we get:
$$
K_1(\mathcal{Q}^\lambda)=\{0\},\;\;{\rm and}\;\; K_0(\mathcal{Q}^\lambda)=
\Z[1/(pq)]/(1-\lambda)\Z[1/(pq)].
$$

\begin{prop}\label{ratgroups}
For $\lambda$ rational with $\lambda=p/q$ in lowest terms, we have
$$
K_1(\mathcal{Q}^\lambda)=\{0\},\;\;{\rm and}\;\;
K_0(\mathcal{Q}^\lambda)\cong\Z[1/(pq)]/(1-\lambda)\Z[1/(pq)]\cong\Z_{(q-p)}.
$$
\end{prop}

\begin{proof}
By Proposition \ref{Zgroups}, $\Gamma_\lambda=\Z[1/(pq)],$
so we must show that 
$$
\Z[1/(pq)]/(1-(1/(pq))\Z[1/(pq)]\cong\Z_{(q-p)}.
$$
Since $(q-p)=(1-p/q)q$ and every element of $\Z[1/(pq)]$ is of the form
$m/(pq)^N,$ it is easy to see that $(q-p)\Z[1/(pq)]=(1-p/q)\Z[1/(pq)].$
Now, $(q-p)$ and $(pq)^N$ are relatively prime for any $N$ and so there exist
$a,b\in\Z$ so that $1=a(q-p)+b(pq)^N$ and hence $m/(pq)^N=(q-p)am/(pq)^N +mb.$
That is, $m/(pq)^N$ and $mb$ represent the same element in the quotient.
So, every element in the quotient has an integer representative. Two integers
$c,d$ represent the same element in the quotient if and only if
$c-d=(p-q)n/(pq)^N,$ or $(c-d)(pq)^N=n(q-p).$ But then:
$$
(c-d)=(c-d)[a(q-p)+b(pq)^N]=(c-d)a(q-p)+b(c-d)(pq)^N=
[(c-d)a+bn](q-p).
$$
That is, $c,d$ represent the same element in $\Z/(q-p)\Z=\Z_{(q-p)}.$
On the other hand if $(c-d)$ is in $(q-p)\Z$ then 
clearly, $[c]=[d]$ in $\Z[1/(pq)]/(1-(1/(pq))\Z[1/(pq)]$
and we are done.
\end{proof}

\begin{cor}\label{funnyF's}
If $\lambda=p/q$ in lowest terms, then 
$$
F^\lambda=F^{p/q}\cong UHF((pq)^\infty)\;\;\;{\rm and}\;\;\;
\mathcal{Q}^\lambda=\mathcal Q^{p/q}\cong O_{(q-p+1)}.
$$
In particular, if $\lambda =\frac{k}{k+1}$ then
$$
F^\lambda\cong UHF((k(k+1))^\infty)\;\;\;{\rm and}\;\;\;
\mathcal{Q}^\lambda\cong O_2.
$$
\end{cor} 

\begin{proof}
Since each $\mathcal{Q}^\lambda$ is separable, nuclear, simple, purely 
infinite and in the bootstrap category $\mathcal N_{nuc}$
once we show that the class of the identity
$e \in \mathcal{Q}^\lambda$ is a generator for 
$K_0(\mathcal{Q}^\lambda)=\Z/(q-p)\Z,$ the
Kirchberg-Phillips Classification Theorem, Theorem 8.4.1 of \cite{RS}, 
shows that
$\mathcal{Q}^\lambda\cong O_{(q-p+1)}.$ To this end we observe that since $e$ is 
mapped to $1$ in $\Z[1/pq],$ we must show that $[1]$ is a generator 
for $K_0(\mathcal{Q}^\lambda)=\Z[1/pq]/(1-(p/q))\Z[1/pq].$ Now, by the proof of 
the
previous proposition, $k[1]=[k\cdot 1] = 0\in\Z[1/pq]/(1-(p/q))\Z[1/pq]$ if and 
only if $[k\cdot 1] = 0\in\Z/(q-p)\Z$ if and only if $k-0=m(q-p)$ for some
$m\in\Z$ if and only if $k$ is a multiple of $(q-p).$ That is,
$[1],[2\cdot 1],\dots,[(q-p-1)\cdot 1]$ are all nonzero in 
$K_0(\mathcal{Q}^\lambda)=\Z/(q-p)\Z$ and hence $[1]$ is a generator.
\end{proof}

\subsection{The $K$-Theory of the Algebras $A_0^\lambda$ for $\lambda$
irrational} The case $\lambda$ rational is much simpler, and while it does fit
into the following scheme, it does not need this deeper machinery.
Initially, we (and others) believed that the algebras $A_0^\lambda$ were 
AF algebras
when $\lambda$ is irrational. In fact we will show that $A_0^\lambda$ is 
{\bf never} AF
when $\lambda$ is irrational. We will set up our examples to fit the 
situation on
page 1487 of \cite{Put2} so that we can apply the six-term exact sequence
of Theorem 2.1 on page 1489 of \cite{Put2}.

We let $\Gamma=\Gamma_\lambda\cong G^0_\lambda.$ Thus, 
$\Gamma\subset
\R$ is a countable dense subgroup of $\R$ which acts on $\R$ by translations. 
Before looking at the crossed product of $\Gamma$ acting on 
$C_0^\lambda(\R)=C_0(\R_\lambda)$ (which gives us $A_0^\lambda$)
we first consider the crossed product of $\Gamma$ acting on $C_0(\R).$ Since
$\Gamma$ acts on $\R$ by translation we can Fourier transform to get an 
isomorphism:
$$
\Gamma\rtimes C_0(\R)\cong \hat{\R}\rtimes C(\hat{\Gamma}).
$$
Then, by Connes' Thom isomorphism we get for $i=0,1$:
$$
K_i(\Gamma\rtimes C_0(\R))\cong K_i(\hat{\R}\rtimes C(\hat{\Gamma}))
\cong K_{i+1}(C(\hat{\Gamma})).
$$

\begin{prop}\label{Iantrick}
The composition:
$$
K_1(C_0(\R))\stackrel{i_*}{\longrightarrow}K_1(\Gamma\rtimes C_0(\R))
\stackrel{\widehat{}}{\longrightarrow} K_1(\hat{\R}\rtimes C(\hat{\Gamma}))
\stackrel{\cong}{\longrightarrow} K_0(C(\hat{\Gamma}))
$$
takes the generator $[u]\in K_1(C_0(\R))=\Z\cdot[u]$; where $u$ is the Bott 
element in $C_0(\R)^1$
defined by $u(t)=\frac{1+it}{1-it}$; to $[1_{\hat{\Gamma}}]$ where 
$1_{\hat{\Gamma}}$ is the identity function in $C(\hat{\Gamma}).$
\end{prop}

\begin{proof} We first work on the right hand side of this sequence of maps.
Let $u(t)= 1 + \varepsilon(t),$ then by the proof of Connes' Thom isomorphism 
from
$$
K_0(C(\hat{\Gamma}))\otimes_{\Z} K_1(C_0(\R))\longrightarrow 
K_1(\R\rtimes C(\hat{\Gamma}))
$$
we see that $[1_{\hat{\Gamma}}]\otimes [u]$ gets mapped to the class 
$[1+(convolution\; by\; \hat{\varepsilon}\cdot1_{\hat{\Gamma}})].$ 
Now in this displayed equation, $K_0(C(\hat{\Gamma}))\otimes_{\Z} K_1(C_0(\R))=
K_0(C(\hat{\Gamma}))\otimes_{\Z} \Z\cdot[u]=K_0(C(\hat{\Gamma}))\cdot[u]
\cong K_0(C(\hat{\Gamma})).$
Thus, $[1_{\hat{\Gamma}}]$ in $K_0(C(\hat\Gamma))$ gets mapped to
the class $[1+(convolution\; by\; \hat{\varepsilon}\cdot1_{\hat{\Gamma}})]$
by the Thom isomorphism.

On the other hand, the map
$K_1((C_0(\R)^1)\longrightarrow K_1((\Gamma\rtimes C_0(\R))^1)$ takes 
$[u]\longmapsto [\delta_0\cdot\varepsilon + 1]$ and by the Fourier transform
this goes to $[(convolution\; by\; \hat{\varepsilon}\cdot1_{\hat{\Gamma}})+1]$
in $K_1(\R\rtimes C(\hat\Gamma)).$ Combining these we get:
$$1\in\Z\longmapsto [u]\in\Z\cdot[u]=K_1((C_0(\R))^1)=K_1(C_0(\R))
\longmapsto [1_{\hat{\Gamma}}]\in K_0((C(\hat\Gamma)).$$
\end{proof}
Now, by Proposition \ref{Zgroups} we know $\Gamma$ in many cases
so that these last groups are quite computable.
In the notation of \cite{Put2} we define the transformation 
groupoids:  
$$
G:=\R_\lambda\rtimes\Gamma,\;\;\;G^\prime :=\R\rtimes\Gamma,\;\;\;
{\rm and}\;\;\;H:=\Gamma\rtimes\Gamma.
$$
Then, $A^\lambda_0=C^*_r(G)$ is the reduced $C^*$-algebra of $G$;
$\Gamma\rtimes C_0(\R) =C^*_r(G^\prime)$ is the reduced $C^*$-algebra of 
$G^\prime$; and $\mathcal K(l^2(\Gamma))$ is the reduced $C^*$-algebra of $H$.
 By the proof of Proposition \ref{simple}
there is a continuous proper surjective map: $\R_\lambda \to \R,$ 
where
points in $\R$ which are not in $\Gamma$ each have a single pre-image, while 
points $\gamma\in\Gamma$ have exactly two pre-images in 
$\R_\lambda,$
which we denote by $\gamma^{-}$ and $\gamma^+.$ Thus, there are two disjoint 
embeddings of $\Gamma$ in $\R_\lambda:$
$$
i_0 , i_1 : \Gamma\to\R_\lambda\;:\;\;\;\; 
i_0(\gamma)=\gamma^- ,\;\;\;i_1(\gamma)=\gamma^+.
$$
Now in order to mesh with the notation of \cite{Put2}, we let $Y:=\Gamma$
with the equivalence relation, ``$=$'';
$X:= \R_\lambda,$ with the equivalence relation
$(i_0(\gamma)\sim i_1(\gamma));$ and quotient $\pi:X\to X^\prime:=\R$ where
$X^\prime=X/(i_0(\gamma)\sim i_1(\gamma))=\R;$ while the ``factor groupoid'' 
of $G=\R_\lambda\times\Gamma=X\times\Gamma$ is 
$G^\prime:=\R\times\Gamma=X^\prime\times\Gamma.$

We represent each of these three $C^*$-algebras on 
$\mathcal H:= l^2(\Gamma^+)\oplus l^2(\Gamma^-)$ where 
$\Gamma^\pm=\{\gamma^\pm\;|\;\gamma\in\Gamma\}$ in the following way.
First we denote the natural orthonormal basis elements of $\mathcal H$
by $\delta_{a^+}$ and $\delta_{a^-}$ for each $a\in \Gamma.$ Now the unitary
representation $U$ of $\Gamma$ on $\mathcal H$ is
$U_\gamma(\delta_{a^\pm})=\delta_{(a-\gamma)^\pm}.$ The actions of 
$C_0(\R_\lambda),$ $C_0(\R)$, and $C_0(\Gamma)$ on $\mathcal H$
are as follows for $f_1\in C_0(\R_\lambda),$ $f_2\in C_0(\R)$,  
$f_3\in C_0(\Gamma)$, and $\delta_{a^\pm}\in\mathcal H$
$$
\pi_1(f_1)(\delta_{a^\pm})=f_1(a^\pm) \delta_{a^\pm}\;\;\;
\pi_2(f_2)(\delta_{a^\pm})=f_2(a) \delta_{a^\pm}\;\;\;
\pi_3(f_3)(\delta_{a^\pm})=f_3(a) \delta_{a^\pm}.
$$
These three covariant pairs of representations, $(\pi_1,U)$, $(\pi_2,U)$, 
and $(\pi_3,U)$ define representations  of $C^*_r(G)=A^\lambda_0$,
$C^*_r(G^\prime)=\Gamma\rtimes C_0(\R)$, and $C^*_r(H)=\mathcal K(l^2(\Gamma))$
respectively on $\mathcal H.$ Since each of these $C^*$-algebras is simple
these representations are faithful.

Now, one  
checks that the hypotheses of Theorem 2.1 of \cite{Put2} are satisfied. As in
\cite{Put1,Put2} one shows that the two {\bf mapping cone algebras}
 of the inclusions:
$$
C^*_r(G^\prime)=\Gamma\rtimes C_0(\R)\longrightarrow A^\lambda_0 =
C^*_r(G)\;\;\;{\rm and}\;\;\;C^*_r(H)\longrightarrow C^*_r(H)\oplus C^*_r(H):
\;\;\;(\; x\mapsto (x,x)\;)
$$
have isomorphic $K$-Theory. One then pastes these isomorphisms into the 
mapping cone long exact sequence for 
$C^*_r(G^\prime)=\Gamma\rtimes C_0(\R)\longrightarrow A^\lambda_0 =C^*_r(G).$
Next one observes that for any $C^*$-algebra, $B$ the diagonal embedding 
$B\longrightarrow B\oplus B$
induces the diagonal embedding $K_*(B)\longrightarrow
 K_*(B)\oplus K_*(B)$ with quotient
isomorphic to $K_*(B)$ (this is true for any abelian group). This implies that
$K_*(B)\cong K_{*+1}(M(B,B\oplus B))$ so that
we get the six-term exact sequence from \cite{Put2}:
$$
\xymatrix{K_1(C_r^*(H)) \ar[r] &K_0(C_r^*(G^\prime))\ar[r]  &
K_0(C_r^*(G))\ar[d]\\
K_1(C_r^*(G))\ar[u] & K_1(C_r^*(G^\prime)) \ar[l]  &K_0(C_r^*(H))\ar[l]}
$$
In our set-up this becomes:
$$
\xymatrix{\{0\}\ar[r] & K_0(\Gamma\rtimes C_0(\R))\ar[r] & 
K_0(\Gamma\rtimes C_0(\R_\lambda))\ar[d]\\
K_1(\Gamma\rtimes C_0(\R_\lambda))\ar[u] & 
K_1(\Gamma\rtimes C_0(\R))\ar[l]&\Z\ar[l]}
$$
Which by Connes' Thom isomorphism becomes:
$$
\xymatrix{\{0\}\ar[r] & K_1(C(\hat{\Gamma}))\ar[r] & 
K_0(A_0^\lambda)\ar[d]\\
K_1(A_0^\lambda)\ar[u] & K_0(C(\hat{\Gamma}))\ar[l]&\Z\ar[l]}
$$

By Proposition \ref{Iantrick}, the nonzero element $[1_{\hat\Gamma}]$ in
$K_0(C_0(\hat\Gamma))\cong K_1(\Gamma\rtimes C_0(\R))$ is mapped to the
image of the class $[u]$ in $K_1(\Gamma\rtimes C_0(\R))$ by Connes' Thom 
isomorphism, and then the image of $[1_{\hat\Gamma}]$ in 
$K_1(\Gamma\rtimes C_0(\R_\lambda))$ is the same as the image of $[u]$
under the inclusion $K_1(\Gamma\rtimes C_0(\R))\longrightarrow
K_1(\Gamma\rtimes C_0(\R_\lambda)).$
However, this is clearly the same as the image of $[u]$ under the inclusion
$K_1(C_0(\R))\to K_1(C_0(\R_\lambda))
\to K_1(\Gamma\rtimes C_0(\R_\lambda)).$ This composition is $0$ since
$C_0(\R_\lambda)$ is an AF-algebra.
That is, the element $[1_{\hat\Gamma}]$ in
$K_0(C_0(\hat\Gamma))$ is mapped to $0$ in $K_1(A^\lambda_0)$ and hence is in 
the image of the map $\Z\longrightarrow K_0(C_0(\hat\Gamma)).$ 
Since 
$[1_{\hat\Gamma}]$ generates a copy of $\Z$ in $K_0(C_0(\hat\Gamma))$,
we have a nonzero homomorphism from $\Z$ to $\Z[1_{\hat\Gamma}]$ which is
onto and hence one-to-one. By the exactness, the map 
$K_0(A^\lambda_0)\longrightarrow \Z$ is the zero map.
$$
{\rm CONCLUSION:}\;\;\;
K_0(A^\lambda_0)\cong K_1(C(\hat\Gamma_\lambda))\;\;\;{\rm and}\;\;\;
K_1(A^\lambda_0)\cong K_0(C(\hat\Gamma_\lambda))/[1_{{\hat\Gamma}_\lambda}]\Z.
$$ 

\begin{prop}\label{irrational}
If $\lambda$ is irrational, then $K_1(A^\lambda_0)\neq \{0\}$ so that
$A^\lambda_0$ is not an AF-algebra.
\end{prop}

\begin{proof}
By items (3) and (5) of Proposition \ref{hatKtheory} we see that when $\lambda$ 
is irrational, $K_0(C(\hat{\Gamma}_\lambda))$ is not singly generated so that
$K_1(A^\lambda_0)\cong K_0(C(\hat\Gamma_\lambda))/[1_{{\hat\Gamma}_\lambda}]\Z
\neq \{0\}.$
\end{proof}
\subsection{$K$-theory computations of particular $\mathcal{Q}^\lambda$ for 
$\lambda$ irrational.}

\noindent{\bf Example(s) $\lambda=1/\sqrt{n}$:}\;\; for $n>1$ a square-free 
integer. Using
Proposition \ref{Zgroups}, we get:
$$
K_0(F^\lambda)=K_0(A^\lambda_0)=K_1(C(\hat{\Gamma_\lambda}))=
\Z[1/n]\oplus\Z[1/n]
$$
$$
K_1(F^\lambda)=K_1(A^\lambda_0)
=(K_0(C(\hat{\Gamma_\lambda})))\big/\Z[1]=
(\Z[1]\oplus\Z[1/n])\big/\Z[1] = \Z[1/n].
$$
To compute the $K$-theory of $\mathcal{Q}^\lambda$ in this case using the 
Pimsner-Voiculescu exact sequence, one must first compute the induced 
automorphism 
$\lambda_*$ on $K_1(C(\hat{\Gamma_\lambda}))$ and on 
$K_0(C(\hat{\Gamma_\lambda}))$ by a more detailed analysis of the proof of
\cite[Proposition 2.11]{Sc}. In the case of $K_1(C(\hat{\Gamma_\lambda}))$
we get a copy of the group $\Gamma_\lambda=\Z[1/n] + \Z[1/n]\sqrt{n}$ and the 
action on $\Gamma_\lambda$ is just multiplication by $\lambda = 1/\sqrt{n}.$
As an action translated to the abstract group $\Z[1/n]\oplus\Z[1/n],$ the 
automorphism becomes $\lambda_*(a,b) = (b, a/n).$ Therefore, $id_*-\lambda_*$
on  $K_0(A^\lambda_0)=\Z[1/n]\oplus\Z[1/n]$ to itself is clearly $1:1$.
Now it is an instructive exercise to show that the kernel of the homomorphism
$$
(a,b)\in\Z[1/n]\oplus\Z[1/n]\mapsto [a+b]\in \Z[1/n]\big/(1-1/n)\Z[1/n]
$$
is exactly the range of the homomorphism
$$
id_*-\lambda_* : \Z[1/n]\oplus\Z[1/n]\longrightarrow \Z[1/n]\oplus\Z[1/n].
$$ 
Hence, we have the isomorphisms:
$$
(\Z[1/n]\oplus\Z[1/n])\big/(id_*-\lambda_*)(\Z[1/n]\oplus\Z[1/n])\cong
\Z[1/n]\big/(1-1/n)\Z[1/n]\cong \Z\big/(n-1)\Z.
$$
where the last isomorphism follows from the proof of Proposition 
\ref{ratgroups} with $p=1$ and $q=n.$ 

Once we have computed the action of
$\lambda_*$ on $K_1(A^\lambda_0)=\Z[1/n]$ we will be ready to compute 
$K_*(\mathcal{Q}^\lambda).$ Now, by Proposition 2.11 of \cite{Sc} we have the 
isomorphism:
$$
K_0(C(\hat{\Gamma}_\lambda))\cong (\Z[1]\otimes_{\Z}\Z[1]) 
\oplus (\Z[1/n]\otimes_{\Z} \Z[1/n])=\Z[1]\oplus(\Z[1/n]\otimes_{\Z} \Z[1/n]).
$$
The action of $\lambda_*$ on $\Z[1]$ is of course the identity. However, the 
action of $\lambda_*$ on $(\Z[1/n]\otimes_{\Z} \Z[1/n])$ is just 
$x\otimes y\mapsto y\otimes x/n.$ If one combines this with the multiplication
isomorphism $x\otimes y\mapsto xy: \Z[1/n]\otimes_{\Z} \Z[1/n]\longrightarrow
\Z[1/n]$ we see that $\lambda_*$ acts as multiplication by $1/n$ on
$\Z[1/n]=\Z[1/n]\otimes_{\Z} \Z[1/n].$ Thus, $\lambda_*$ on the 
quotient $K_1(A^\lambda_0)=\Z[1/n]$ is just multiplication by $1/n.$
Therefore, $id_*-\lambda_*$ becomes multiplication by $(1-1/n)$ on $\Z[1/n]$
which is clearly $1:1.$ Applying the Pimsner-Voiculescu exact sequence and
recalling that $K_i(\mathcal{Q}^\lambda)=K_i(A^\lambda)$ we get
the isomorphisms: 
$$
K_0(\mathcal{Q}^\lambda)
\cong \Z\big/(n-1)\Z,\;\;{\rm and}\;\;K_1(\mathcal{Q}^\lambda)\cong
\Z\big/(n-1)\Z,\;\;{\rm for}\ \lambda=1/\sqrt{n}.
$$
For $n>2$ we get $K_1\neq 0$ and so these are not Cuntz algebras, in fact
not even Cuntz-Krieger algebras since $K_1$ has nonzero torsion. For 
$\lambda =1/\sqrt{2}$ however we get $K_0=0=K_1$ and by classification theory,
we must have $\mathcal Q^{1/\sqrt{2}}\cong O_2!$ However, even in this case 
the fixed 
point algebra, is NOT AF since it has $K_1=\Z[1/2]$, the tape-measure group.
So for the simplest irrational number $1/\sqrt{2}$ we get the Cuntz algebra,
$O_2$ with a strange gauge action of $\T.$

\begin{rems*}
In the examples below it is important to note that any polynomial 
of the form $f(x)=x^n +ax^{n-1}+\cdots+bx\pm 1$ has at most $n-1$ roots in
the open interval $(0,1)$ because the product of all the roots of $f$ must 
equal $\pm 1.$
\end{rems*}

\noindent{\bf Example(s)  quadratic integers and an algorithm:} 
If both $\lambda$ and 
$\lambda^{-1}$ 
are quadratic integers
with $\lambda\in (0,1),$ then $\lambda^2+a\lambda \pm 1=0$ where the 
integer polynomial
$f(x)=x^2+ax\pm 1$ is irreducible over $\mathbb{Q}.$ With these 
restrictions there 
are two cases, either $f(x)=x^2+ax-1$ where $a>0$ and 
$\lambda=1/2\cdot(\sqrt{a^2 +4}-a)\in (0,1);$ or $f(x)=x^2+ax+1$ where 
$a\leq -3$ and $\lambda=1/2\cdot(-\sqrt{a^2 -4}-a)\in (0,1).$ 

In the first case, $\lambda^2 +a\lambda-1=0,$ with $a>0,$ so 
that 
$\lambda +a-\lambda^{-1}=0$ and $\lambda^{-1}= a + \lambda.$ 
For this case we outline an algorithm using the ideas of the Smith Normal 
Form and the Pimsner-Voiculescu exact sequence to calculate the $K$-Theory.
By Proposition \ref{hatKtheory}, and the CONCLUSION before Proposition
\ref{irrational}, $\Gamma_\lambda =\Z + \Z\lambda$
and $K_0(A^\lambda_0)\cong K_1(C(\hat{\Gamma}_\lambda))=
\bigwedge^1(\Gamma_\lambda)=\Gamma_\lambda\cong \Z^2.$
Giving $\Gamma_\lambda$ its  $\Z$-basis $\{1,\lambda\}$ we see that 
the action of the automorphism $\lambda_*$ on 
$K_0(A^\lambda_0)\cong\Gamma_\lambda$ has matrix: 
$\left[\begin{array}{cc} 0 & 1\\1 & -a\end{array}\right].$ 
So, $(id-\lambda_*)=\left[\begin{array}{cc} 1& -1\\-1& (a+1) \end{array}\right]
:= M.$ To compute the kernel and cokernel of
this matrix mapping $\Z^2\to\Z^2$ we row and column-reduce $M$ over $\Z$
to obtain matrices $P,Q\in GL(2,\Z)$ so that $PMQ=D$ where $D$ is 
diagonal over $\Z.$ Then $\ker(M)\cong \ker(D)$ and 
$\Z^2/M(\Z^2)\cong\Z^2/D(\Z^2).$ In this case, we get 
$D=\left[\begin{array}{cc} 1 & 0\\0 & a\end{array}\right].$
Hence, on $K_0(A^\lambda_0)$ we have 
$$
\ker(id-\lambda_*)=\ker(M)\cong \ker(D) =\{0\}\;\;{\rm and}\;\;
{\rm coker}(id-\lambda_*)\cong {\rm coker}(D) =\Z/a\Z.
$$
Now we compute $(id-\lambda_*)$ on 
$$
K_1(A^\lambda_0)\cong
K_0(C(\hat{\Gamma}_\lambda))/\Z\cdot 1_o=(\Z\cdot 1_o\oplus\Z(1\wedge\lambda))
/\Z\cdot 1_o=\Z(1\wedge\lambda).
$$
Now, $\lambda_*(1\wedge\lambda)=\lambda\wedge\lambda^2 =
\lambda\wedge(1-a\lambda)=\lambda\wedge 1=(-1) 1\wedge\lambda.$ That is,
$\lambda_* = -id$ on $K_1(A^\lambda_0)\cong\Z.$ Therefore, 
$(id-\lambda_*) = {\rm multiplication\; by}\; 2 $ on $\Z(1\wedge\lambda)$
which has $\ker(id-\lambda_*)=\{0\}$ and ${\rm cokernel}(id-\lambda_*)\cong
\Z/2\Z$.

Applying these results to the Pimsner-Voiculescu exact
sequence we obtain:
$$
K_0(\mathcal{Q}^\lambda)=\Z/a\Z\;\;{\rm and}\;\;
K_1(\mathcal{Q}^\lambda)=\Z/2\Z,\;\;{\rm for}\;\; \lambda^2 +a\lambda -1 = 0,
\;n\geq 1.
$$
None of these examples are Cuntz-Krieger algebras since $K_1$
is not torsion-free. In particular, when $\lambda=(1/2)(\sqrt{5} -1)$ is the 
inverse of the golden mean, we get $K_0=\{0\}$ and $K_1=\Z/2\Z.$

In the second case, $\lambda^2 +a\lambda+1=0,$ we have as above,
$K_0(A^\lambda_0)\cong \Gamma_\lambda=\Z+\Z\lambda$ with 
$\Z$-basis $\{1,\lambda\};$ the diagonal version of $(id-\lambda_*)$ is
$D={\rm diag}[1,(a+2)]$ so that 
$\ker(id-\lambda_*)=\{0\}$ and ${\rm coker}(id-\lambda_*)\cong \Z/(a+2)\Z.$
On the other hand, $K_1(A^\lambda_0)\cong\Z(1\wedge\lambda)$ only now,
$\lambda_*=id$ here so that $(id-\lambda_*)=0$ and hence 
$\ker(id-\lambda_*)\cong\Z$ while ${\rm coker}(id-\lambda_*)\cong\Z.$ By 
Pimsner-Voiculescu we get
$$
K_0(\mathcal{Q}^\lambda)=\Z\oplus(\Z/(a+2)\Z)\;\;
{\rm and}\;\;K_1(\mathcal{Q}^\lambda)=\Z,\;\;{\rm for}\;\;
\lambda^2 + a\lambda +1=0,\;a\leq -3.
$$
We note that in this case, $\mathcal{Q}^\lambda$ has the correct $K$-theory
to be a Cuntz-Krieger algebra (and is therefore stably isomorphic to one),
and that in the case $a=-3$ 
(i.e., $\lambda = (1/2)(3-\sqrt{5})$) we have $K_0=\Z=K_1.$

\noindent{\bf Example cubic integers:} If $\lambda$ {\bf and} $\lambda^{-1}$ 
are cubic integers
with $\lambda\in (0,1),$ then $\lambda^3+a\lambda^2+b\lambda \pm 1=0$ where the 
integer polynomial
$f(x)=x^3+ax^2+bx\pm 1$ is irreducible over $\mathbb{Q}.$ Such an $f$ 
is irreducible
if and only if $f(1)\neq 0\neq f(-1).$ 
There are two cases depending on the constant, $\pm 1.$

First, consider $f(x)=x^3+ax^2+bx-1=0$ with $f(1)=a+b\neq 0$
and $f(-1)=a-b-2\neq 0$ so that $f$ is irreducible.
Now assume $a+b$ is {\bf positive} (but $a\neq b+2$). Then $f(0)=-1$ 
and $f(1)=a+b>0$ so that $f$ has a {\bf unique} root in 
$(0,1)$ since it is a cubic. 

Next consider the same polynomial, $f(x)=x^3+ax^2+bx-1=0,$
with $a+b$ {\bf negative} (but $a\neq b+2$). Since both $f(0)$ and $f(1)$ are 
negative, in
order to have a solution the function $f$ must have a local maximum on $(0,1).$
There are
examples with no solutions in $(0,1);$ for example, $f(x)=x^3-3x-1.$
In order to have a {\bf unique solution}, then considering $f^\prime(x)$,
one would need $4a^2-12b=0:$ while this has many solutions, they all satisfy
$|a|\leq b$ and so we can not have $a+b<0.$  So solutions are {\bf not} unique 
in this case.
But, there are infinitely many cubics with {\bf two distinct} solutions
in $(0,1);$ 
eg.,
$f(x)=x^3-(a+k)x^2+ax-1$ for $a\geq k+4$ and $k\geq 1$ has two solutions
in $(0,1),$ since $f(.5)>0.$ 

We now calculate the $K$-theory of 
$\mathcal{Q}^\lambda$ assuming that $\lambda$ satisfies 
$f(x)=x^3+ax^2+bx-1=0,$ where $a+b\neq 0,$ and $a-b\neq 2.$
Now, $\lambda^3 +a\lambda^2 +b\lambda-1=0,$ so that 
$\lambda^3 =1-a\lambda^2 -b\lambda$ and 
$\lambda^{-1}= \lambda^2+a\lambda+b.$ 
Then, 
$\Gamma_\lambda =\Z + \Z\lambda+\Z\lambda^2$
and $K_0(A^\lambda_0)\cong K_1(C(\hat{\Gamma}_\lambda))=
\bigwedge^{odd}(\Gamma_\lambda)
=\Gamma_\lambda\oplus (\Gamma_\lambda\wedge\Gamma_\lambda\wedge\Gamma_\lambda)
=\Gamma_\lambda\oplus(\Z(1\wedge\lambda\wedge\lambda^2))\cong\Z^4.$
Giving $\Gamma_\lambda$ its natural $\Z$-basis $\{1,\lambda,\lambda^2\}$  
the induced homomorphism $(id-\lambda_*)$ on 
$K_0(A^\lambda_0)\cong\Z^4$ yields
the diagonal matrix, 
$D={\rm diag}[1,1,(a+b),0]$
so that on $K_0(A^\lambda_0)$ we have 
$$
\ker(id-\lambda_*)\cong \ker(D)\cong\Z\;\;{\rm and}\;\;
{\rm coker}(id-\lambda_*)\cong {\rm coker}(D) =(\Z/(a+b)\Z)\oplus \Z.
$$
Now, $K_1(A^\lambda_0)\cong K_0(C(\hat{\Gamma}_\lambda))/\Z \cdot 1_o=
\bigwedge^2(\Gamma_\lambda)=\Z(1\wedge\lambda) + \Z(1\wedge\lambda^2) +
\Z(\lambda\wedge\lambda^2)\cong\Z^3.$ By similar computations we get
for $K_1(A^\lambda_0)\cong\Z^3;$ the matrix 
$D={\rm diag}[1,1,(a+b)].$
Hence, on $K_1(A^\lambda_0)$ we have
$$
\ker(id-\lambda_*)\cong \ker(D)=\{0\}\;\;{\rm and}\;\;
{\rm coker}(id-\lambda_*)\cong {\rm coker}(D) =\Z/(a+b)\Z.
$$
Applying these results to the Pimsner-Voiculescu exact
sequence we obtain: 
$$
K_0(\mathcal{Q}^\lambda)=\Z\oplus(\Z/(a+b)\Z)  \;\;{\rm and}\;\;
K_1(\mathcal{Q}^\lambda)= \Z\oplus(\Z/(a+b)\Z)\;\;{\rm for}\;\;
\lambda^3+a\lambda^2+b\lambda -1=0.$$

\begin{rems*} In case $a+b=1$ (which has infinitely many solutions 
corresponding to infinitely 
many distinct invertible cubic integers $\lambda\in (0,1)$)
we get $K_0(\mathcal{Q}^\lambda)=\Z=K_1(\mathcal{Q}^\lambda),$ which as 
noted above 
is also true for the invertible quadratic integer, 
$\lambda = (1/2)(3-\sqrt{5}).$ In the general cubic case with constant term 
 $-1$ we always have 
{\bf non-torsion} elements in both $K_0$ and $K_1:$ this is the opposite of
the case where the constant term is $+1,$ where we see below that $K_0$
and $K_1$ are both {\bf torsion} groups. A similar phenomenon occurs in the
quadratic case above, except that there we get {\bf torsion} in the $-1$ case
and {\bf non-torsion} in the $+1$ case! That this may be a periodic phenomenon
is supported by a calculation of two quartic examples: first,
the unique solution $\lambda\in(0,1)$ 
to the irreducible quartic $f(x)=x^4-3x^3+1$ gives us
$K_0=\Z$ and 
$K_1=\Z\oplus(\Z/3\Z)\oplus(\Z/3\Z);$ while, second,
the unique solution $\lambda\in(0,1)$ 
to the irreducible quartic $f(x)=x^4+3x^3-1$ gives us
$K_0=(\Z/3\Z)\oplus(\Z/3\Z)$ and 
$K_1=(\Z/9\Z)\oplus(\Z/2\Z),$ similar to the quadratic case. Proposition  
\ref{periodic} is further evidence.

When an irreducible polynomial 
$f(x)=x^n+ax^{n-1}+\cdots\pm1$ has two 
roots, $\lambda_1, \lambda_2\in (0,1),$  then $\Gamma_{\lambda_1}\cong
\Gamma_{\lambda_2}$ as {\bf rings} (but {\bf not}
as ordered rings, for that would imply equality). Still, 
$\mathcal Q^{\lambda_1}\cong\mathcal Q^{\lambda_2}$ (at least stably)
since the calculation of their $K$-groups are identical.
Their KMS states are not equivalent since the type ${\rm III}$ factors that they
generate are not isomorphic, as we will see below.
\end{rems*}

\begin{prop}\label{periodic}
Suppose $\lambda$ satisfies the irreducible (over $\Z$) polynomial,
$f(x)=x^n+\cdots\pm 1=0$.\\
{\rm (1)} For $n$ {\bf odd}, if
$f(x)=x^n+\cdots+1$ then $K_0(\mathcal{Q}^\lambda)$ has $\Z/2\Z$ as a 
summand. \\
While, if $f(x)=x^n+\cdots-1$ then $K_0(\mathcal{Q}^\lambda)$ has $\Z$ as a 
summand (so, by the next Proposition, ${\rm rank}(K_0)={\rm rank}(K_1)\geq 1$ 
in this case).\\
{\rm (2)} For $n$ {\bf even}, if
$f(x)=x^n+\cdots+1$ then $K_1(\mathcal{Q}^\lambda)$ has $\Z$ as a 
summand (so, by the next Proposition, ${\rm rank}(K_0)={\rm rank}(K_1)\geq 1$ 
in this case). \\
While, if $f(x)=x^n+\cdots-1$ then $K_1(\mathcal{Q}^\lambda)$ has 
$\Z/2\Z$ as a summand.
\end{prop}

\begin{proof}
In $K_*(A^\lambda_0)$ there is a $\lambda_*$-invariant summand,
$(1\wedge\lambda\wedge\lambda^2\wedge\cdots\wedge\lambda^{n-1})\Z.$
Depending on $n({\rm mod}\;2)$ and the constant term $\pm 1,$ the action of
$\lambda_*$ on this summand is $\pm id.$ Hence,
$(id-\lambda_*)$ here is either $0$ or $2 (id).$ Applying
Pimsner-Voiculescu gives a summand in 
$K_*(\mathcal{Q}^\lambda)$ of either $\Z$ or $\Z/2\Z.$
\end{proof}
\begin{prop}Suppose $\lambda$ is algebraic.\\
{\rm (1)} Then, $rank(K_0(\mathcal{Q}^\lambda))=
rank(K_1(\mathcal{Q}^\lambda))$ so that $\mathcal{Q}^\lambda$ is not
stably isomorphic to $O_\infty.$\\
{\rm (2)} If $\lambda$ and $\lambda^{-1}$ are both algebraic integers and 
$\mathcal {Q}^\lambda$ is stably isomorphic to a Cuntz algebra $O_n,$
then the minimal polynomial of $\lambda$ has odd degree and constant term $+1.$
Moreover, $n$ is congruent to $3(mod\;4)$ and all such Cuntz 
algebras appear this way.
\end{prop}
\begin{proof}
To see part (1) we tensor the Pimsner-Voiculescu exact sequence by $\Q$
(which preserves exactness)
to obtain an exact hexagon of $\Q$-vector spaces:
$$
\xymatrix{V_1 \ar[r]^{\theta_1} & V_1\ar[r]^{\tau_1}  &
K_1^{\Q}\ar[d]^{\mu_1}\\
K_0^{\Q}\ar[u]^{\mu_0} & V_0 \ar[l]^{\tau_0}  & V_0\ar[l]^{\theta_0}}
$$
where $V_i=K_i(A^\lambda_0)\otimes\Q,$ and $K_i^{\Q}=K_i(A^\lambda)\otimes\Q.$ 
Then $$\dim(K_0^{\Q})=\rank(\mu_0)+{\rm nullity}(\mu_0)
=\rank(\mu_0)+\rank(\tau_0)\;\;{\rm and}\;\;
\dim(K_1^{\Q})=\rank(\mu_1)+\rank(\tau_1)$$
and
$$\rank(\tau_0)+\rank(\theta_0)=\dim(V_0)=\rank(\theta_0)+\rank(\mu_1)
\;\;{\rm so\;\;that}\;\;\rank(\tau_0)=\rank(\mu_1).$$
Similarly, $\rank(\tau_1)=\rank(\mu_0),$ so that:
$$\dim(K_0^{\Q})=\rank(\mu_0)+\rank(\tau_0)=\rank(\mu_1)+\rank(\tau_1)=
\dim(K_1^{\Q}).$$
That is, $\rank(K_0(\mathcal{Q}^\lambda))=
\dim K_0(\mathcal{Q}^\lambda)\otimes_{\Z}\Q=\dim K_0(A^\lambda)\otimes_{\Z}\Q=
\cdots \rank(K_1(\mathcal{Q}^\lambda)).$
By Proposition \ref{periodic}, if the minimal polynomial of $\lambda$ has even 
degree, then $K_1(\mathcal {Q}^\lambda)\neq \{0\},$ and so 
$\mathcal {Q}^\lambda$ cannot be stably isomorphic to a Cuntz algebra. If
$\mathcal {Q}^\lambda$ is stably isomorphic to $O_n$ then $n$ is finite by part
(1)and by Proposition \ref{periodic}, the order of $K_0(\mathcal {Q}^\lambda)$
must be even and therefore $n$ must be odd. Furthermore, the minimal
polynomial must have constant term $+1.$ In order for 
$K_0(\mathcal {Q}^\lambda)$ to be a finite cyclic group of even order,
it must be of the form $\Z/m\Z\oplus\Z/2\Z$ where $m$ is {\bf odd} since
$\Z/2\Z$ is a summand. Let $m=2k+1$ then
$$n=\sharp[\Z/m\Z\oplus\Z/2\Z] + 1 =2m+1=4k+3$$
as claimed.\\
\hspace*{.2in} In the examples below where $\lambda^3+a\lambda^2+b\lambda+1=0,$
and either $a-b=1$ and $b\leq -2$ OR $a=b=-1$ and $b\leq -1$, we obtain
(stably, at least) all the Cuntz algebras $O_n$ where $n\equiv 3( mod\;4).$
\end{proof}
Now consider the case of irreducible cubics
of the form $f(x)=x^3+mx^2+nx+1;$ so $f(1)=m+n+2\neq 0$ and 
$f(-1)=m-n\neq 0.$ 
Since $f(0)=1,$ if we have $f(1)=m+n+2 <0,$ then we have as 
above a {\bf unique} root in $(0,1).$

If $f(1)=m+n+2 >0,$ we can have distinct roots. 
For example, if $n=-4$ and $m=3$, then,
$f(x)=x^3+3x^2-4x+1$ has two roots in $(0,1).$  
If $n<<0,$ we get several solutions $m$ for each $n$: eg., $n=-7$ implies 
that any $m$ with $6\leq m\leq 9$ will yield a polynomial with 
two roots in $(0,1).$  

We now calculate the $K$-Theory of 
$\mathcal{Q}^\lambda$ assuming $\lambda$ satisfies $f(\lambda)=
\lambda^3+m\lambda^2+n\lambda+1=0.$ 
Again, 
$K_0(A^\lambda_0)\cong \Z^4,$ but now the diagonal matrix 
$D={\rm diag}[1,1,(m+n+2),2].$
On $K_1(A^\lambda_0)\cong \Z^3,$ 
the matrix $D={\rm diag}[1,1,(m-n)].$
Both matrices are $1:1$ since 
$m+n+2\neq 0\neq m-n.$ We get:
$$
K_0(\mathcal{Q}^\lambda)=\Z/(n+m+2)\Z\oplus \Z/2\Z\;\;{\rm and}\;\;
K_1(\mathcal{Q}^\lambda)=\Z/(m-n)\Z\;\;{\rm for}\;\; 
\lambda^3+m\lambda^2+n\lambda+1=0.
$$
To obtain Cuntz algebras, we need $m-n=\pm 1.$
It turns out $f(1)>0$ can not occur, so we must have $f(1)=m+n+2<0$ 
hence there is a unique root $\lambda$ in $(0,1).$ Combining this
inequality with $m-n=\pm 1$ we get exactly two infinite families of solutions;
$m=n+1$ for $n\leq -2$, OR $m=n-1$ for $n\leq -1.$
In either case, the sequence of numbers  $\{|m+n+2|\}$ 
is the same: $\{2k+1|k\geq 0\}.$
For this sequence we get the $K_0$ groups: 
$\Z/(2k+1)\Z\oplus\Z/2\Z\cong \Z/(4k+2)\Z.$ Since the $K_1$ groups are all 
$\{0\}$, by construction, the algebras $\mathcal{Q}^\lambda$ are 
(at least stably)
the Cuntz algebras, $O_{4k+3}$ for $k\geq 0.$ 
That is, $O_3$, $O_7$, $O_{11}$, etc.\\

\noindent{\bf Example, $\lambda$ transcendental:}
\begin{lemma}
Let $\varphi : \bigoplus_{n\in\Z} \Z\longrightarrow \Z$ be the surjective 
homomorphism, $\phi(\{a_n\}):=\sum_{n\in\Z} a_n;$ and let 
$S\in Aut(\bigoplus_{n\in\Z} \Z)$ be the shift $S(\{a_n\}_{n\in\Z}):
=\{a_{n-1}\}_{n\in\Z}.$ Then, $(id-S)$ is $1:1$ and
$\ker(\varphi)={\rm Im}(id-S),$ so that
$(\bigoplus_{n\in\Z} \Z)/{\rm Im}(id-S)\cong\Z.$
\end{lemma}

\begin{proof}
As a model for $\bigoplus_{n\in\Z} \Z$ we use $\Z[x,x^{-1}]=
\bigoplus_{n\in\Z} \Z x^n,$ the ring of Laurent polynomials over 
$\Z$ (i.e., the group ring over $\Z$ of the group
$\{x^n| n\in\Z\}$). Here, $\varphi$ is the augmentation map,
$S$ is multiplication by $x,$ and $(id-S)$ is multiplication by $(1-x)$ which
is $1:1.$ Now, 
$$
\sum_{n=-N}^{N}a_n x^n \in\ker(\varphi) \Leftrightarrow \sum_{n=-N}^{N} a_n =0
\Leftrightarrow \sum_{n=-N}^{N}a_n x^{n+N} \in\ker(\varphi).
$$
Let $p(x)=\sum_{n=-N}^{N}a_n x^{n+N} \in \Z[x]$ so
$p(1)=\sum_{n=-N}^{N} a_n =0.$ Hence, $p(x)$ factors:
$p(x)=(1-x)q(x)$ where initially $q(x)\in\mathbb{Q}[x].$ Since $p(x)\in\Z[x]$ it
is easy to see that in fact, $q(x)\in\Z[x]$ also. Then,
$$
\sum_{n=-N}^{N}a_n x^n = x^{-N}p(x)=(1-x)x^{-N}q(x)\in(1-x)\Z[x,x^{-1}]
={\rm Im}(id-S).
$$
That is, $\ker(\varphi)\subseteq {\rm Im}(id-S),$ and the other 
containment is immediate.
\end{proof}

\begin{prop}
If $\lambda$ is transcendental then 
$$
K_0(\mathcal{Q}^\lambda)\cong \bigoplus_{n=1}^\infty \Z 
\cong K_1(\mathcal{Q}^\lambda).
$$
\end{prop}

\begin{proof}
In this case, by Proposition \ref{hatKtheory} and the CONCLUSION before
Proposition \ref{irrational} we have:
$$
K_0(A^\lambda_0)=\bigoplus_{k=1}^\infty\bigwedge^{2k-1}(\Gamma_\lambda)\;\;
{\rm and}\;\;
K_1(A^\lambda_0)=\bigoplus_{k=1}^\infty\bigwedge^{2k}(\Gamma_\lambda)\;\;
{\rm where}\;\; \Gamma_\lambda=\bigoplus_{n=-\infty}^{\infty}\Z\lambda^n.
$$
Now, each individual summand 
$\bigwedge^{m}(\Gamma_\lambda)$ is invariant under $\lambda_*$ and yields 
(for $m>1$) an 
infinite direct sum of ($\lambda_*$-invariant) examples of the previous lemma 
where the action of $\lambda_*$ is just the shift.
The general case is notation-heavy, so we do 
the examples, $\bigwedge^2$ and $\bigwedge^3.$ Letting $\Z_+$ denote the 
positive integers, we have:
$$
\bigwedge^2(\Gamma_\lambda)=\bigoplus_{k\in\Z_+}\left(\bigoplus_{n\in\Z}
(\lambda^n\wedge\lambda^{n+k})\Z\right)\;\;{\rm and}\;\; 
\bigwedge^3(\Gamma_\lambda)=\bigoplus_{(k_1,k_2)\in\Z^2_+}
\left(\bigoplus_{n\in\Z}
(\lambda^n\wedge\lambda^{n+k_1}\wedge\lambda^{n+k_1+k_2})\Z\right).
$$
The case $m=1$ is just the group 
$\Gamma_\lambda=\bigoplus_{n\in\Z} \Z\lambda^n$ which yields a single instance
of the lemma.

Applying the lemma we see that $(id-\lambda_*)$ is $1:1$
on both $K_0(A^\lambda_0)$ and $K_1(A^\lambda_0)$ and that both 
$K_0(A^\lambda_0)/(id-\lambda_*)(K_0(A^\lambda_0))$ and
$K_1(A^\lambda_0)/(id-\lambda_*)(K_1(A^\lambda_0))$ are isomorphic to a 
countable direct sum of copies of $\Z.$ An application
of the Pimsner-Voiculescu exact sequence completes the proof. 
\end{proof}

{\bf Remark}. The classification theory of Kirchberg algebras implies that for 
$\lambda$ transcendental  we have a new realisation of the algebras found 
in \cite{Cu1} and denoted $\mathcal Q_{\N}$ there.

\subsection{The dual action of $\T^1$ on $A^\lambda$ and its restriction
to the gauge action on $\mathcal{Q}^\lambda$}

Recall, $G_\lambda^0=\{g\in G_\lambda\,|\,|g|=1\}$ is a normal subgroup
of $G_\lambda.$ The subgroup of $G_\lambda$ of elements of the form
$[\lambda^n\, :\,0]$ is isomorphic to $\Z$ and acts on $G_\lambda^0$
by conjugacy:
$$
[\lambda^n\,:\,0] [1\,:\,b] [\lambda^{-n}\,:\,0]
=[1\,:\,\lambda^n b].
$$
Thus $G_\lambda=\Z\rtimes G_\lambda^0$ is a semidirect product and we can
write $A^\lambda$ as an iterated crossed product:
$$
A^\lambda=G_\lambda \rtimes_\alpha C_0^\lambda(\R)=
\Z\rtimes(G^0_\lambda \rtimes_{\alpha} C_0^\lambda(\R))=\Z\rtimes A_0^\lambda.
$$
The dual action $\gamma$ of $\T^1$ on $A^\lambda$ is relative to this 
latter crossed 
product so that for each $z\in\T^1$ and $x$ in 
the Banach $*$-algebra, $l^1_\alpha(G_\lambda,C_0^\lambda(\R))$ we have:
$$
\gamma_z(x)(g)=z^nx(g)\ \text{if}
\ x\in l^1_\alpha(G_\lambda,C_0^\lambda(\R));\ g\in G_\lambda\ 
\text{and}\ |g|=\lambda^n.
$$
Since $A^\lambda$ is defined to be the completion of this Banach $*$-algebra
in its universal representation, the action $\gamma$ extends uniquely to an
action (also denoted by $\gamma$) of $\T^1$ as automorphisms of $A^\lambda.$
The fixed point subalgebra of the dual action is, of course, exactly
$A_0^\lambda=G^0_\lambda \rtimes_\alpha C_0^\lambda(\R).$

Since the projection $e$ is in $A_0^\lambda,$ the
action $\gamma$ restricts to an action of $\T^1$ on
$\mathcal{Q}^\lambda=eA^\lambda e,$ which we will also denote by $\gamma.$
We call this the {\bf gauge action} of $\T^1$ on $\mathcal{Q}^\lambda.$
Now, $\gamma$ is clearly a strongly continuous action of $\T^1$ on 
$\mathcal{Q}^\lambda$.  Averaging over $\gamma$ with respect to 
normalised Haar measure
gives a positive, faithful expectation $\Phi$ of $\mathcal{Q}^\lambda$ onto the 
fixed-point algebra which is clearly $F^\lambda$:

$$
\Phi(a):=\frac{1}{2\pi}\int_{\T^1} \gamma_z(a)\,d\theta\ \text{for}
\ a\in \mathcal{Q}^\lambda,\ \text{and}\  z=e^{i\theta}.
$$

\begin{prop}\label{fixedpoint}
The fixed point algebra, $F^\lambda=eA_0^\lambda e$ is the norm closure 
of finite 
linear combinations of elements of the form: 
$$
\mathcal X_{[a,b)}\cdot\delta_g\ \text{where}\ 
g=[1\,:\,c]\ \text{and}\ [a,b)\subseteq[0,1)\cap[c,1+c),
$$
for $a,b,c\in\Gamma_\lambda.$ 
Recall, $A_0^\lambda\cong\mathcal K(l^2(\Z))\otimes F^\lambda.$
\end{prop}

\begin{proof}
Applying the integral formula for $\Phi$ to a finite linear combination of 
the generators for $\mathcal{Q}^\lambda$ we see that the only terms that 
survive are those where $|g|=1:$ that is, $g$ has the above form. Then
we apply item (2) of Lemma \ref{lambdagens} to obtain the condition on the
interval $[a,b).$ 
\end{proof}

\begin{cor}
The stabilised algebra $ \mathcal{Q}^\lambda\otimes\mathcal K$ is a crossed 
product
of the stabilised fixed-point algebra $F^\lambda\otimes\mathcal K $ by an
action of $\Z.$ For $\lambda=1/n$ this is a theorem of J. Cuntz.
\end{cor}

\begin{proof}
By Proposition \ref{stable}, 
$A_0^\lambda\cong F^\lambda\otimes\mathcal K ,$ and 
$A^\lambda\cong  \mathcal{Q}^\lambda\otimes\mathcal K.$
 By the discussion at the beginning of subsection 3.1,
$A^\lambda\cong\Z\rtimes A_0^\lambda$ 
 and the proof is complete.
See \cite[Section 2]{Cu}.
\end{proof}

\begin{rems*}
If we combine the previous observation that $F^\lambda$ is the fixed point
subalgebra of $\mathcal{Q}^\lambda$ under the gauge action
with Corollary \ref{funnyF's}
we get, for example,  $O_2\cong \mathcal Q^{2/3}$ with a gauge action whose 
fixed point
subalgebra $F^{2/3}$ is a UHF algebra of type $6^\infty.$ Interestingly,
$F^{3/4}$ is UHF of type $12^\infty=6^\infty$ which is therefore 
isomorphic to $F^{2/3}.$ So we have two gauge actions on $O_2$ with isomorphic
UHF fixed point subalgebras, with distinct, inequivalent KMS states:
one where $\beta=\log(3/2)$ and the other where $\beta=\log(4/3)$ by 
Proposition \ref{KMS} below. Moreover, the two von Neumann algebras 
generated by the
GNS representations of $O_2$ are not isomorphic as they are 
type ${\rm III}_\lambda$ factors for $\lambda$ equalling $2/3$ and $3/4,$ 
respectively, by Theorem \ref{III-lambda} below.
\end{rems*}
\subsection{The $\gamma$-invariant semifinite weight on $A^\lambda$ and its
restriction to  $\mathcal{Q}^\lambda$}

The aim of this subsection is to exhibit the unique KMS states for the 
gauge action on
$\mathcal{Q}^\lambda$. We first recall the definition of KMS states.

\begin{defn}
Let $A$ be a $C^*$-algebra with a continuous action 
$\gamma:\R\rightarrow Aut(A).$ Let $\psi$ be a state on $A$ and $\beta\in\R$
a real number. We define $\psi$ to be a ${\rm KMS}_\beta$ state for the action
$\gamma$ if
$$
\psi(x\,\gamma_{i\beta}(y))=\psi(yx)
$$
for all $x,y\in \mathcal A$ a dense $\gamma$-invariant $*$-subalgebra of 
$A_\gamma,$ the subalgebra of analytic elements for the action $\gamma.$
We refer to \cite[Section 2.5]{BR} for basic information on the
subalgebra of analytic elements, $A_\gamma$ and to \cite[Section 5.3]{BR2} 
for all the basic information on {\rm KMS} states.
\end{defn}
Since $G_\lambda$ is discrete it is well-known that the map
$$
x\mapsto x(1): l^1_\alpha(G_\lambda,C_0^\lambda(\R))\to C_0^\lambda(\R)
$$
extends uniquely to a faithful conditional expectation 
$E: A^\lambda\to C_0^\lambda(\R).$ Composing $E$ with the densely defined
(norm) lower semicontinuous weight on $C_0^\lambda(\R)$ given by integration, 
gives us a densely defined (norm) lower semicontinuous weight on $A^\lambda$
which we denote by $\bar\psi$. In particular, for 
$x\in l^1_\alpha(G_\lambda,C_0^\lambda(\R))$ we have:
$$
\bar\psi(xx^*) = \int_{\R} xx^*(1)(t)dt=\int_{\R}\left(\sum_{h\in G_\lambda}
x(h)\overline{x(h)}\right)(t)dt=\sum_{h\in G_\lambda}
\left(\int_{\R}|x(h)(t)|^2 dt\right).
$$
So that $\bar\psi$ is faithful.
We observe that $\bar\psi$ {\bf is not a trace}, since
$\bar\psi(x^*x) =\sum_{h\in G_\lambda}|h^{-1}|\int_{\R}|x(h)(t)|^2dt.$

\begin{prop}
The weight $\bar\psi$ on $A^\lambda$ restricts to a faithful semifinite trace
$\bar\tau$ on $A_0^\lambda$ and also restricts to a state denoted by $\psi$ 
on $\mathcal{Q}^\lambda$ satisfying:\\
\noindent(1) The gauge action $\gamma$ of $\T^1$ on $\mathcal{Q}^\lambda$
leaves the state $\psi$ invariant.\\
\noindent(2) The state $\psi$ restricted to the fixed point algebra,
$F^\lambda$ is a faithful (finite) trace denoted by $\tau;$ which is, of course,
the restriction of $\bar\tau$ on $A_0^\lambda$ to $F^\lambda.$\\
\noindent(3) With $\Phi: \mathcal{Q}^\lambda \to F^\lambda$ the
canonical expectation, we have $\psi=\tau\circ\Phi.$ 
\end{prop}

\begin{proof}
Since $\bar\psi (e)= \int_{\R}\mathcal X_{[0,1)}(t) dt =1,$ we see
that $\bar\psi$ restricted to $\mathcal{Q}^\lambda$ is a faithful state.
To see item (1), it suffices to see that the gauge action on 
$l^1_\alpha(G_\lambda,C_0^\lambda(\R))$ 
leaves $\bar\psi$ invariant. To this end, let $x\geq 0$ be in
$l^1_\alpha(G_\lambda,C_0^\lambda(\R))$, and let $z\in \T^1$. Then
$$
E(\gamma_z(x))=\gamma_z(x)(1)=z^0x(1)=x(1)=E(x)
$$
and so 
$$ 
\bar\psi(\gamma_z(x))=\int_\R\gamma_z(x)(1)(t) dt=\int_\R E(x)(t) dt
=\bar\psi(x).
$$
To see item (2) we use Proposition \ref{fixedpoint} and the above computation
that shows that while $\bar\psi$ is not generally a trace, to see that {\bf it is
a trace} when the group elements all have determinant $1$.

To see item (3), it suffices to see that for any 
$x\in \mathcal{Q}^\lambda$ we have $E(x)=E(\Phi(x))$, but this is the same as
$x(1)=\Phi(x)(1)$ which is clear since $\det(1)=1.$
\end{proof}

Now, since the state $\psi$ is invariant under the action $\gamma$, this 
action is unitarily implemented on $\mathcal L^2(\mathcal{Q}^\lambda,\psi).$  For
$z\in\T^1$ and $x\in \mathcal{Q}^\lambda_c$ we define:
$$
(u_z(x))_h=z^n x_h\ \text{for}\ h\in G_\lambda\ \text{with}\ 
|h|=\lambda^n.
$$ 
We define the spectral subspaces of this unitary group on 
$\mathcal L^2(\mathcal{Q}^\lambda,\psi)$ in the usual way. For each $k\in\Z$
let $\Phi_k$ be the operator on $\mathcal L^2(\mathcal{Q}^\lambda,\psi):$
$$
\Phi_k(x)=\frac{1}{2\pi}\int_{\T^1}z^{-k}u_z(x) d\theta,\ z=e^{i\theta},\ 
x\in \mathcal L^2(\mathcal{Q}^\lambda,\psi).
$$ 
We observe that if $x=f\cdot\delta_g$ is a typical generator of 
$\mathcal{Q}^\lambda$
considered as a vector in $\mathcal L^2(\mathcal{Q}^\lambda,\psi)$ then we have:
\bean
\Phi_k(f\cdot\delta_g)=\left\{\begin{array}{ll} f\cdot\delta_g & 
\mbox{if $|g|=\lambda^k$}\\ 0 & \mbox{otherwise}\end{array}\right.
\eean
More generally, on $\mathcal H:=\mathcal L^2(\mathcal{Q}^\lambda,\psi),$ 
we have
$\Phi_k(\mathcal H)=\{x\in \mathcal H\;|\; u_z(x)=
z^kx\;\;\text{for\;\;all}\;\;z\in\T^1\}.$ 
\begin{lemma}
For each $k\in\Z$ the subspace $\text{span}\{f\cdot\delta_g\in 
\mathcal{Q}^\lambda\;|\; 
|g|=\lambda^k\}$ is dense in the range of $\Phi_k.$
The operators $\Phi_k$ are mutually orthogonal projections on
$\mathcal H$ which sum to the identity operator 
$1=\pi(e).$
\end{lemma}

\begin{proof}
The proof of the first statement is similar to the proof of Proposition 
\ref{fixedpoint}. The mutual orthogonality of the $\Phi_k$ follows from 
the fact that $\langle f_1\cdot\delta_{g_1}|f_2\cdot\delta_{g_2}\rangle_\psi=0$
unless $g_1=g_2.$
\end{proof} 

\begin{prop}\label{KMS}
The dense $*$-subalgebra of $\mathcal{Q}^\lambda$ consisting of finite linear 
combinations
of the partial isometries $\mathcal X_{[a,b)}\cdot \delta_g$ is contained
in the subset of entire elements, $\mathcal{Q}^\lambda_\gamma,$ for the action 
$\gamma$
considered as an action of $\R: t\mapsto \gamma_{e^{it}}.$ Moreover, $\psi$
is a ${\rm KMS}_\beta$ state for this action where $\beta=\log(\lambda^{-1}).$
In fact, $\psi$ is the {\bf unique} {\rm KMS} state for this action 
(regardless of $\beta$).
\end{prop}

\begin{proof}
Let $y=\mathcal X_{[a,b)}\cdot\delta_g\in \mathcal{Q}^\lambda$ where 
$\det(g)=\lambda^k.$
Then, $t\mapsto \gamma_{e^{it}}(y)=e^{ikt}y;\ t\in\R$ obviously extends to 
the entire
function $w\mapsto \gamma_{e^{iw}}(y)=e^{ikw}y;\ w\in\C$. For 
$w=\log(\lambda^{-1})i$, this equation becomes 
$\gamma_{e^{iw}}(y)=\gamma_{\lambda}(y)=\lambda^k y$. Letting 
$\beta=\log(\lambda^{-1})$, we have $\gamma_{\beta i}(y)=\lambda^k y$.
Now, let $x=\mathcal X_{[c,d)}\cdot h$ so we want to see that: 
$\lambda^k \psi(xy) = \psi(x\gamma_{\beta i}(y))=\psi(yx).$ That is, we 
{\bf want} 
$\lambda^k \psi(xy) =\psi(yx).$ Now both sides of this equation are zero
unless $h=g^{-1}.$ But, when $h=g^{-1},$ we have
$$
xy=\mathcal X_{[c,d)}\cdot\mathcal X_{[g^{-1}(a),g^{-1}(b))}\cdot\delta_I\ \ 
{\rm while}\ \ 
yx=\mathcal X_{[a,b)}\cdot\mathcal X_{[g(c),g(d))}\cdot\delta_I.
$$
Moreover,
 $$
 s\in [c,d)\cap [g^{-1}(a),g^{-1}(b)) \Longleftrightarrow
 g(s)\in [g(c),g(d))\cap [a,b).
 $$
Since $\det(g)=\lambda^k$ the transformation $g$ increases the measure by a 
factor of $\lambda^k$ and the result follows. That is, $\psi$ is a 
${\rm KMS}_\beta$
state for the action $\gamma$ of $\R$ for $\beta=\log(\lambda^{-1}).$

Now let $\phi$ be a KMS state on $\mathcal{Q}^\lambda$ for 
the action
$\gamma.$ Since $\mathcal{Q}^\lambda$ is purely infinite it has no nontrivial 
traces
and so $\phi$ must be KMS for some nonzero $\beta.$ Hence by 
\cite[Proposition 5.3.3]{BR2}, $\phi$ is invariant under the action of $\gamma$.
Now, if $\mathcal X_{[a,b)}\cdot\delta_g\in \mathcal{Q}^\lambda$ with 
$\det(g)=\lambda^k$,
then we have for all $z\in\T$:
$$
\phi(\mathcal X_{[a,b)}\cdot\delta_g)=
\phi(\gamma_z(\mathcal X_{[a,b)}\cdot\delta_g))=
z^k\phi(\mathcal X_{[a,b)}\cdot\delta_g).
$$
That is, if $\det(g)\neq 1$ we must have 
$\phi(\mathcal X_{[a,b)}\cdot\delta_g)=0,$ and so $\phi$ is supported on
$F^\lambda.$ Since $F^\lambda$ is $\gamma$-invariant and $\phi$ is KMS
for some nonzero $\beta$, $\phi$ is a trace on $F^\lambda$ by 
\cite[5.3.28]{BR2}.

Now, if $x=\mathcal X_{[a,b)}\cdot\delta_g\in F^\lambda$ and 
$g\neq I$, then we claim that $\phi(x)=0$. For
suppose $g=[1\,:\,c]$ with $c>0$. Then there is a positive integer $n$ such 
that $a+nc<b\leq a+(n+1)c$ and so
$$
x=\mathcal X_{[a,b)}\cdot\delta_g = \mathcal X_{[a,c)}\cdot\delta_g +
\mathcal X_{[a+c,a+2c)}\cdot\delta_g +\cdots + 
\mathcal X_{[a+nc,b)}\cdot\delta_g:= v_0 + v_1 +\cdots + v_n.
$$
Now each of these partial isometries $v_k$ satisfies $v_k^2=0,$ and so
$\phi(v_k)=\phi(v_k v_k^* v_k)= \phi(v_k^2 v_k^*)=0$ since $\phi$ is a trace
on $F^\lambda.$ Thus, $\phi(x)=0$ as claimed.

Hence $\phi$ is supported on the commutative subalgebra 
$$
\mathcal C:=\overline{\mbox{span}}\{f\cdot\delta_I\;|\; 
f\in C_0^\lambda(\R)\;\;{\rm and}\;\; supp(f)\subseteq [0,1)\}.
$$ 
Morever,
if $f_1,f_2$ are characteristic functions of subintervals of $[0,1)$ with
endpoints in $\Gamma_\lambda$ and having the same length they give equivalent 
elements $f_i\cdot\delta_I$ in $F^\lambda$ and therefore have the same value 
under $\phi.$ 

Now, since $A^\lambda_0\cong F^\lambda\otimes \mathcal K$
we can define a lower semicontinuous, densely defined trace, $\tilde{Tr}$ on
$A^\lambda_0$ via $\tilde{Tr}=\phi\otimes Tr,$ where $Tr$ is the trace on
$\mathcal K.$ So, for $\mathcal X_{[a,b)}\cdot \delta_I \in F^\lambda$ we have
$\tilde{Tr}(\mathcal X_{[a,b)}\cdot \delta_I)=
\phi(\mathcal X_{[a,b)}\cdot \delta_I).$
Then, for $k_1 < k_2 \in\Z$ the element 
$\mathcal X_{[k_1,k_2)}\cdot \delta_I$ is the sum of $(k_2-k_1)$ projections
in $A^\lambda_0$ each equivalent to $\mathcal X_{[0,1)}\cdot \delta_I$ which
has trace equal to $1;$ that is, 
$\tilde{Tr}(\mathcal X_{[k_1,k_2)}\cdot \delta_I)= (k_2-k_1).$ Now, for any
$a<b$ in $\Gamma_\lambda,$ we have $\mathcal X_{[a,b)}\cdot \delta_I\sim
\mathcal X_{[0,b-a)}\cdot \delta_I$ and so 
$\tilde{Tr}(\mathcal X_{[a,b)}\cdot \delta_I) =
\tilde{Tr}(\mathcal X_{[0,b-a)}\cdot \delta_I),$ and these values are finite
since both these projections are dominated by 
$\mathcal X_{[-N,N)}\cdot \delta_I$ for a sufficiently large integer $N.$
It now suffices to prove the following. {\bf Claim:} 
$\tilde{Tr}(\mathcal X_{[a,b)}\cdot \delta_I) = 
b-a\;\;{\rm for}\;\; a<b\in\Gamma_\lambda.$

By the previous discussion we can assume that $a=0$ so that
$b>0.$ Given $\varepsilon >0$ we choose positive integers $m,n$ such that
$$
\frac{1}{m} \leq \varepsilon\;\;{\rm and}\;\; \frac{n-1}{m}\leq b <
\frac{n}{m},
$$
so that $(n-1)\leq bm < n$ and $(n-1), bm, n \in\Gamma_\lambda.$ Hence
$$
(n-1)=\tilde{Tr}(\mathcal X_{[0,(n-1))}\cdot \delta_I)\leq
\tilde{Tr}(\mathcal X_{[0,bm)}\cdot \delta_I)\leq
\tilde{Tr}(\mathcal X_{[0,n)}\cdot \delta_I)=n.
$$
But, 
$$
\mathcal X_{[0,bm)}\cdot \delta_I=\mathcal X_{[0,b)}\cdot \delta_I+
\mathcal X_{[b,2b)}\cdot \delta_I+\cdots+
\mathcal X_{[(m-1)b,bm)}\cdot \delta_I
$$
and these projections are mutually equivalent in $A^\lambda_0.$ That is,
$$
(n-1)\leq m\tilde{Tr}(\mathcal X_{[0,b)}\cdot \delta_I)\leq n\;\;\; 
{\rm so\;\;that}\;\;\;
\frac{n-1}{m}\leq\tilde{Tr}(\mathcal X_{[0,b)}\cdot \delta_I)\leq\frac{n}{m}.
$$
Hence, $|\tilde{Tr}(\mathcal X_{[0,b)}\cdot \delta_I) -b|<\frac{1}{m}\leq
\varepsilon.$ That is, $b=\tilde{Tr}(\mathcal X_{[0,b)}\cdot \delta_I)
=\phi(\mathcal X_{[0,b)}\cdot \delta_I)$ and $\phi$ agrees with the given trace
$\tau$ on $F^\lambda$ and therefore $\phi$ agrees with $\psi$ on 
$\mathcal{Q}^\lambda.$ 
\end{proof}

\begin{rems*}
The above proof shows that the algebra $F^\lambda$ has a unique (faithful)
tracial state
$\tau,$ and that $A^\lambda_0$ has a unique (faithful) lower semicontinuous,
densely defined trace normalized so that it has value $1$ at
$e=\mathcal X_{[0,1)}\cdot\delta_I.$
\end{rems*}
\subsection{The von Neumann algebra $\pi(A^\lambda)^{-wo}$ acting on 
$\mathcal L^2(A^\lambda,\bar\psi)$ is a type ${\rm III}_\lambda$ factor.}
To prove this we will show that it is unitarily equivalent
to a version of the Murray-von Neumann ``group-measure space'' construction
of type ${\rm III}$ factors on $l^2(G_\lambda)\otimes\mathcal L^2(\R):$
see \cite[Chapter 1, Section 9]{Dix}.
We conclude that it is a ${\rm III}_\lambda$ factor by an appeal to Connes'
thesis \cite{C0}. In order to be consistent with our use of right
$C^*$-modules later, we will do our GNS constructions so that our inner products
are linear in the {\bf second} variable.

\begin{prop}
The $*$-algebra $A^\lambda_c$ is a Tomita algebra with the inner product:
$$\langle y|x\rangle_{\bar\psi} =\bar\psi(y^*x)=\sum_{h\in G_\lambda}
|h|^{-1}\langle x_h|y_h\rangle_{\mathcal L^2(\R)}.$$ 
Here we denote $x_h$ in place of $x(h)$ to simplify notation. In this setting 
we have for $x\in A^\lambda_c$:\\
$(1)\;\;\;\text{Sharp:}\;\; S(x)_h=\alpha_h(\overline{x_{h^{-1}}});$\\
$(2)\;\;\;\text{Flat:}\;\;  F(x)_h=|h|\alpha_h(\overline{x_{h^{-1}}});$\\
$(3)\;\;\;\text{Delta:}\;\; \Delta(x)_h=|h|x_h.$ 
\end{prop}\label{tomita1}
\begin{proof}
We refer to \cite{Ta} for Takesaki's version of the axioms for a Tomita algebra.
Since $Sharp$ is defined to be the adjoint operation on the algebra, item (1) is
immediate. A straightforward calculation shows that for all $x,y\in A^\lambda_c$
we have that the defining equation for $Flat$ holds, namely:$
\langle S(y)|x\rangle_{\bar\psi}=\langle F(x)|y\rangle_{\bar\psi}$
so that item (2) holds. By definition, $\Delta=FS$ and so a simple
calculation shows that $\Delta(x)_h=|h|x_h$ and (3) holds. From this formula for
$\Delta$ we see that for each $z\in \C$ we have $\Delta^z(x)_h=|h|^z x_h$
and a straightforward calculation shows that 
$\Delta^z(x\cdot y)=(\Delta^z(x))\cdot(\Delta^z(y))$ so that each $\Delta^z$
is an algebra homomorphism of $A^\lambda_c$ as required. That each left
multiplication $\pi(x)$ is bounded when $x$ is supported on a single group
element is straightforward and the generalization to finitely supported
elements is then trivial. The fact that it is a $*$-representation holds as it 
does for the GNS representation for any weight.\\
\hspace*{.2in}In order to see that products are dense we recall that we have
local units. That is, for each positive integer $N$ we have defined 
$E_N=\mathcal X_{[-N,N)}\cdot\delta_1,$ and have noted that for each
$y\in A^\lambda_c$ that satisfies $\text{supp}(y_h)\subseteq [-N,N)$ 
for all $h,$
we have $E_N\cdot y=y.$ Axioms IV, V, VI in \cite{Ta} are simple calculations
involving the definitions of $S$, $F$, and $\Delta$ which we leave to 
the reader.\\
\hspace*{.2in}Since our inner products are linear in the second variable, we 
modify Tomita's Axiom VIII to read: 
$z\mapsto\langle x|\Delta^x(y)\rangle_{\bar\psi}$
is analytic on $\C$ for all $x,y\in A^\lambda_c.$ We easily calculate that
$\langle x|\Delta^z(y)\rangle_{\bar\psi}=
\sum_h|h|^{z-1}\langle y_h|x_h\rangle_{\mathcal L^2(\R)}.$
This function is analytic since the sum is finite.
\end{proof}

\begin{lemma}
The representation of $A^\lambda_c$ on $\mathcal L^2(A^\lambda_c,\bar\psi)$
decomposes as the integrated form of a covariant pair of representations:
\bean 
&&{\rm (1)}\;\; \tilde\pi:C^\lambda_{00}(\R)\to\mathcal 
B(L^2(A^\lambda_c,\bar\psi)),\;\;{\rm where:}\;\;  (\tilde\pi(f)(y))_h=
f\cdot y_h\;\;{\rm for}\;\; f\in C^\lambda_{00}(\R)\;\;{\rm and}\;\; 
y\in A^\lambda_c;\\
&&{\rm (2)}\;\; U:G_\lambda\to U(L^2(A^\lambda_c,\bar\psi))\;\;{\rm where:}\;\;
(U_g(y))_h=\alpha_g(y_{g^{-1}h})\;\;{\rm for}\;\; g\in G_\lambda\;\;{\rm and}
\;\; y\in A^\lambda_c.
\eean 
\end{lemma}

\begin{proof}
It is straightforward to verify that $U$ is a unitary representation of
$G_\lambda$ and that $\tilde\pi$ is a $*$-representation of 
$C^\lambda_{00}(\R).$
To see the covariance condition:
\bean
(U_g\tilde\pi(f)U_{g^{-1}}(y))_h
=\cdots= \alpha_g(f\cdot \alpha_{g^{-1}}(y_{gg^{-1}}))=\alpha_g(f)\cdot y_h
=(\tilde\pi(\alpha_g(f))y)_h.
\eean
That is, $U_g\tilde\pi(f)U_{g^{-1}} = \tilde\pi(\alpha_g(f)).$

 Now, by Proposition 7.6.4 of \cite{Ped} the integrated form of
this covariant pair is the representation:
$$
(\tilde\pi \times U)(y)=\sum_h \tilde\pi(y_h)U_h\;\;\text{for}\;\;y\in
A^\lambda_c.
$$
Now, we evaluate this operator on the vector $x\in A^\lambda_c:$
\bean 
[((\tilde\pi\times U)(y))(x)]_k
=\sum_h[\tilde\pi(y_h)U_h(x)]_k
= \sum_h y_h\alpha_h(x_{h^{-1}k})=(y\cdot x)(k)=(\pi(y)(x))_k.
\eean
That is, $(\tilde\pi\times U)(y)=\pi(y)$ the operator left multiplication
by $y.$
\end{proof}

\subsubsection{ A representation of $A^\lambda$ on 
$l^2(G_\lambda)\otimes\mathcal L^2(\R)$}
We define a covariant pair of representations of $C^\lambda_0(\R)$ and
$G_\lambda$ on $l^2(G_\lambda)\otimes\mathcal L^2(\R)$ as follows:
$$
\text{(1)\;for}\;\;f\in C^\lambda_0(\R)\;\;\text{let}\;\;
\overline\pi(f)=1\otimes M_f,\;\;
\text{and\;\;(2)\;for}\;\;g\in G_\lambda\;\;\text{let}\;\;
\overline{U}_g=\Lambda(g)\otimes V_g
$$
where $\Lambda$ is the left regular representation of 
$G_\lambda$ on $l^2(G_\lambda):$
$$
(\Lambda(g)\xi)(h)=\xi(g^{-1}h)\;\;\text{for}\;\;\xi\in l^2(G_\lambda);
$$ 
and $V$ is the unitary action of $G_\lambda$ on $\mathcal L^2(\R)$ induced by 
the action of $G_\lambda$ on $\R:$  
$$
(V_g(f))(t)=|g|^{-1/2}f(g^{-1}t)\\;\;\text{for}\;\;f\in\mathcal L^2(\R).
$$
Using these equations one easily checks the covariance condition for 
$g\in G_\lambda$ and $f\in C^\lambda_0(\R):$
$$
\overline{U}_g\overline\pi(f)\overline{U}_g^*=\overline\pi(\alpha_g(f)).
$$
Clearly the representation $\overline\pi$ extends uniquely by weak-operator
continuity to the usual representation $1\otimes M$ of $\mathcal L^\infty(\R)$
on $l^2(G_\lambda)\otimes\mathcal L^2(\R)$ and is covariant with the
unitary representation $\overline{U}$ of $G_\lambda$ for the action $\alpha$
of $G_\lambda$ on $\mathcal L^\infty(\R).$ Clearly, the von Neumann algebra on
$l^2(G_\lambda)\otimes\mathcal L^2(\R)$ generated by the unitaries 
$\overline{U}_g$ and the operators $1\otimes M_f$ for $g\in G_\lambda$
and $f\in C^\lambda_{00}(\R)$, is the same as the von Neumann algebra
generated by the unitaries 
$\overline{U}_g$ and the operators $1\otimes M_f$ for $g\in G_\lambda$
and $f\in \mathcal L^\infty(\R)$.
The second item of the following Proposition is clear.

\begin{prop}\label{prop:typeIII}
(1) The representation $\pi=(\tilde\pi\times U)$ of $A^\lambda$ on 
$\mathcal L^2(A^\lambda_c,\bar\psi)$ is unitarily equivalent to the
representation $(\overline\pi\times\overline{U})$ of $A^\lambda$ on
$l^2(G_\lambda)\otimes\mathcal L^2(\R).$\\ 
(2) $(\overline\pi\times\overline{U})(A^\lambda)^{\prime\prime}$ is the von 
Neumann
crossed product (in the sense of the group-measure space construction
of Chapter 1 Section 9 of \cite{Dix}) 
$G_\lambda \rtimes_\alpha \mathcal L^\infty(\R).$\\
(3) This von Neumann algebra is a type ${\rm III}$ factor.
\end{prop}

\begin{proof}
To see item (3) we use the proof of  
\cite[Theorem 2, Section 9, Chapter 1]{Dix} 
where instead of the $ax+b$ group $G$ with $a,b\in \mathbb{Q}$
and $a>0$ and its subgroup $G_0$ (with $a=1$), we use $G_\lambda$ and 
its subgroup $G_\lambda^0$ (with $|g|=1$), to conclude that our 
von Neumann algebra is a type ${\rm III}$ factor.

To see item (1), we first define a unitary
$W:\mathcal L^2(A^\lambda_c,\bar\psi)\to 
l^2(G_\lambda)\otimes\mathcal L^2(\R)$ as follows:
$$
W\left(\sum_{i=1}^m f_i\cdot \delta_{h_i}\right)=
\sum_{i=1}^m |h_i|^{-1/2}\delta_{h_i}\otimes f_i.
$$
On the left side of this equation we are using the formalism $f\cdot \delta_h$ 
for singly supported elements in $A^\lambda_c$ with $f\in C^\lambda_{00}(\R)$
and $h\in G_\lambda.$ On the right of this equation we are using $\delta_h$
to denote the canonical orthonormal basis elements in $l^2(G_\lambda)$
and regarding $f\in C^\lambda_{00}(\R)\subset \mathcal L^2(\R).$
Clearly, $W$ is well-defined and linear with dense range. One easily checks 
that:
for all $x,y\in A^\lambda_c$ we have
$$
\langle y|x\rangle_{\bar\psi}=
\langle W(x)|W(y)\rangle_{l^2\otimes \mathcal L^2}
$$
recalling that the inner product on $A^\lambda_c$ is linear in the second
coordinate. Thus $W$ is a unitary and its inverse (adjoint) defined at
first on the elements in $l^2(G_\lambda)\otimes\mathcal L^2(\R)$ of the form
$\sum_{i=1}^m \delta_{h_i}\otimes f_i$ with the $f_i\in C^\lambda_{00}(\R)$,
is given by:
$$
W^*\left(\sum_{i=1}^m \delta_{h_i}\otimes f_i\right)=
\sum_{i=1}^m |h_i|^{1/2}f_i\cdot \delta_{h_i}.
$$
One then verifies the following two equations for $f\in C^\lambda_{00}(\R)$ and
$g\in G_\lambda:$
$$
\text{(1)}\;\; W\tilde\pi(f)W^*=1\otimes M_f
=\overline\pi(f)\;\;\;\text{and}\;\;\;
\text{(2)}\;\; WU_gW^*=\Lambda(g)\otimes V_g=
\overline{U}_g.
$$
The second equation is  more subtle and requires the observation:
$U_g(f\cdot\delta_h)=|g|^{1/2}V_g(f)\cdot \delta_{gh}.$\\ 
This completes the proof of the proposition.
\end{proof}

\subsubsection{The factor $\pi(A^\lambda)^{\prime\prime}$ acting on 
$\mathcal L^2(A^\lambda_c,\bar\psi)$ is type ${\rm III}_\lambda$}

We work in the unitarily equivalent setting of 
$(\overline\pi\times\overline{U})(A^\lambda)^{\prime\prime}$ acting on
$l^2(G_\lambda)\otimes\mathcal L^2(\R)$ afforded by Proposition 
\ref{prop:typeIII}.
Recall that the subgroup of $G_\lambda$ of matrices of the form
$[\lambda^n\,:\, 0]$ is isomorphic to $\Z$ and acts on the normal
subgroup $G_\lambda^0$ by conjugacy, and so
$G_\lambda=\Z\rtimes G_\lambda^0$ is a semidirect product and we can
write a canonical right coset decomposition of $G_\lambda:$
$$
G_\lambda =\bigcup_{n\in \Z}G_\lambda^0 \cdot[\lambda^n\,:\, 0].
$$
This gives us an internal orthogonal decomposition of $l^2(G_\lambda):$
$$
l^2(G_\lambda)=\sum_{n\in\Z}\oplus\; l^2\left(G_\lambda^0\cdot
[\lambda^n\,:\,0]\right)\cong l^2(\Z)\otimes l^2(G_\lambda^0).
$$
Here the latter isomorphism is given explicitly on basis elements by the map
which takes the $\delta$-function at $g\cdot[\lambda^n\,:\,0]$
to $\delta_n\otimes \delta_g$ for $n\in\Z$ and $g\in G_\lambda^0.$

One checks that the restriction of the representation 
$(\overline\pi\times\overline{U})$ of $A^\lambda=G_\lambda\rtimes 
C^\lambda_0(\R)$
to $A_0^\lambda:=G_\lambda^0\rtimes C^\lambda_0(\R)$ on
$l^2(G_\lambda)\otimes\mathcal L^2(\R)$ is unitarily equivalent to the
representation on $l^2(\Z)\otimes l^2(G_\lambda^0)\otimes\mathcal L^2(\R)$
via the covariant pair:
$$
1_{\Z}\otimes\Lambda(h)\otimes V_h=1_{\Z}\otimes\overline{U}_h\;\;\text{for}
\;\;h\in G_\lambda^0\;\;\text{and}
$$
$$
1_{\Z}\otimes 1\otimes M_f=1_{\Z}\otimes\overline\pi(f)
\;\;\text{for}\;\;f\in C^\lambda_0(\R).
$$
Therefore, the von Neumann subalgebra of 
$(\overline\pi\times\overline{U})(A^\lambda)^{\prime\prime}$ generated by
$(\overline\pi\times\overline{U})(A_0^\lambda)$ is isomorphic to the 
von Neumann algebra 
on $l^2(G_\lambda^0)\otimes\mathcal L^2(\R)$ generated by the operators
$\Lambda(h)\otimes V_h$ for $h\in G_\lambda^0$ and $1\otimes M_f$ for
$f\in C^\lambda_0(\R).$ This is clearly the same as the von Neumann algebra
generated by the operators
$\Lambda(h)\otimes V_h$ for $h\in G_\lambda^0$ and $1\otimes M_f$ for
$f\in \mathcal L^\infty(\R),$ and this von Neumann algebra is a factor of 
type $II_\infty$
by the methods of \cite[Chapter 1, Section 9]{Dix}. Thus,
$(\overline\pi\times\overline{U})(A_0^\lambda)^{\prime\prime}$ is a 
type $II_\infty$
subfactor of the type ${\rm III}$ factor, 
$(\overline\pi\times\overline{U})(A^\lambda)^{\prime\prime}.$
Moreover, the faithful normal semifinite trace on 
$(\overline\pi\times\overline{U})(A_0^\lambda)^{\prime\prime}$
is given by the restriction of $\bar\psi.$

Finally, conjugation by the unitary, 
$\overline{U}_g$ for $g=[\lambda\,:\,0]$, which lies in our type 
$III$ factor, defines an automorphism $\beta$ of the type $II_\infty$ subfactor 
which scales the trace by $\lambda.$ If $\cn_0$ is our type $II_\infty$ factor
acting on $l^2(G_\lambda^0)\otimes\mathcal L^2(\R)$ then our type ${\rm III}$
factor, say $\mathcal A_\lambda$ acting on 
$l^2(\Z)\otimes l^2(G_\lambda^0)\otimes\mathcal L^2(\R)$
is unitarily equivalent to the von Neumann crossed product
$\mathcal A_\lambda \cong \Z\rtimes_\beta \cn_0$ and hence is a 
type ${\rm III}_\lambda$ factor by \cite[Theorem 4.4.1]{C0}. We have proved 
the following Proposition.

\begin{prop}
The von Neumann algebra $\pi(A^\lambda)^{\prime\prime}$ acting on 
$\mathcal L^2(A^\lambda_c,\bar\psi)$ is a type ${\rm III}_\lambda$ factor.
Moreover, it is unitarily equivalent to 
$(\overline\pi\times\overline{U})(A^\lambda)^{\prime\prime}$ acting on
$l^2(G_\lambda)\otimes\mathcal L^2(\R).$ The von Neumann subalgebra of 
$(\overline\pi\times\overline{U})(A^\lambda)^{\prime\prime}$ generated by
$(\overline\pi\times\overline{U})(A_0^\lambda)$ is a type $II_\infty$ factor.
The space $l^2(G_\lambda)\otimes\mathcal L^2(\R)$ factors as
$l^2(\Z)\otimes l^2(G_\lambda^0)\otimes\mathcal L^2(\R)$ and with this
factorization, our $II_\infty$ factor has the form 
$\cn_0=1_\Z\otimes \tilde\cn_0$
where $\tilde\cn_0$ acts on $l^2(G_\lambda^0)\otimes\mathcal L^2(\R).$
Thus, our type ${\rm III}_\lambda$ factor is unitarily equivalent to the von Neumann 
crossed product $\Z\rtimes_\beta \cn_0$ where the automorphism $\beta$ of
$\cn_0$ is given by $\beta=Ad(\overline{U}_g)$ where 
$g=[\lambda\;:\;0].$
\end{prop}

\subsection{The von Neumann algebra, $\pi_0(\mathcal{Q}^\lambda)^{-wo}$ acting 
on  
$\mathcal L^2(\mathcal{Q}^\lambda,\psi)$ is type ${\rm III}_\lambda$}

\begin{thm}\label{III-lambda}
The von Neumann algebra, $\pi_0(\mathcal{Q}^\lambda)^{-wo}$ acting on 
$\mathcal L^2(\mathcal{Q}^\lambda,\psi)$ is type ${\rm III}_\lambda.$
Moreover, the von Neumann subalgebra, $\pi_0(F^\lambda)^{-wo}$ is a
type $II_1$ factor with unique faithful normal state given by the restriction
of the vector state, $\psi$ which is the same as $\tau$ on $F^\lambda.$ By
the general theory of type ${\rm III}$ factors, $\pi_0(\mathcal{Q}^\lambda)^{-wo}$
is isomorphic to  $\pi(A^\lambda)^{-wo}$ acting on 
$\mathcal L^2(A^\lambda_c,\bar\psi).$ 
\end{thm}

\begin{proof}
Recall that $\mathcal{Q}^\lambda=eA^\lambda e$ where 
$e=\mathcal X_{[0,1)}\cdot\delta_1
\in A^\lambda.$ Then 
$$
\pi(e)(\pi(A^\lambda)^{-wo})\pi(e)=
(\pi(e)\pi(A^\lambda)\pi(e))^{-wo}=\pi(\mathcal{Q}^\lambda)^{-wo}
$$ 
and the 
cut-down
of the type ${\rm III}$ factor $\pi(A^\lambda)^{-wo}$ (on its separable Hilbert 
space) by the nonzero projection $\pi(e)$
is isomorphic to $\pi(A^\lambda)^{-wo}$ since $\pi(e)$ is
Murray-von Neumann equivalent to the identity operator. Of course the
{\bf cut-down mapping} by $\pi(e)$ is {\bf not} an isomorphism.  Moreover, 
by left Hilbert algebra
theory, the operator {\bf right} multiplication by $e$ which is denoted by
$\pi^\prime(e)$ is in the commutant of $\pi(A^\lambda)^{-wo}$ acting on
$\mathcal L^2(\mathcal{Q}^\lambda,\bar\psi)$ and since we are in a factor the 
{\bf mapping}
$\pi(A^\lambda)^{-wo}\to \pi^\prime(e)\pi(A^\lambda)^{-wo}$ {\bf is an 
isomorphism} by \cite[Chapter 1, Section 2, Prop. 2]{Dix}.
Restricting this isomorphism to $\pi(\mathcal{Q}^\lambda)^{-wo}$ gives us an 
isomorphism
$\pi(\mathcal{Q}^\lambda)^{-wo}\to \pi^\prime(e)\pi(\mathcal{Q}^\lambda)^{-wo}$ 
which acts on
the Hilbert space $\pi^\prime(e)\pi(e)(\mathcal L^2(A^\lambda,\bar\psi)),$
which has as a dense subspace 
$\pi^\prime(e)\pi(e)(A^\lambda_c)=eA^\lambda_c e\subset eA^\lambda e=
\mathcal{Q}^\lambda$ 
with the inner product given by $\bar\psi$ which is the same as the inner
product on $eA^\lambda_c e$ given by the state $\psi.$ The completion
of this space is, of course, 
$\mathcal L^2(\mathcal{Q}^\lambda,\psi)$ with the action of $\mathcal{Q}^\lambda$ 
being the
GNS representation afforded by the state $\psi.$ We denote this 
representation of $\mathcal{Q}^\lambda$ on 
$\mathcal L^2(\mathcal{Q}^\lambda,\psi)$ by $\pi_0$
to distinguish it from the representation $\pi$ of $A^\lambda$ on the larger
space, $\mathcal L^2(A^\lambda_c,\bar\psi).$

Similar considerations applied to the type $II_\infty$ subfactor,
$\pi(A_0^\lambda)^{-wo}\subset\pi(A^\lambda)^{-wo}\;\;\text{on}\;\;
\mathcal L^2(A^\lambda_c,\bar\psi),$
show that:
$$
\pi(e)(\pi(A_0^\lambda)^{-wo})\pi(e)=
(\pi(e)\pi(A_0^\lambda)\pi(e))^{-wo}=\pi(F^\lambda)^{-wo}.
$$
Now the projection $\pi(e)$ is actually in
the type $II_\infty$ subfactor $\pi(A_0^\lambda)^{-wo}$ of 
$\pi(A^\lambda)^{-wo}$ and has finite ($\psi$) trace $=1$ there. 
Therefore, $\pi(F^\lambda)^{-wo}$ is a type $II_1$ factor on 
$\mathcal L^2(\mathcal{Q}^\lambda,\psi)$ with trace given
by the vector state $\psi.$ We remark that this is clearly a larger
space than the subspace, 
$\mathcal L^2(F^\lambda,\tau)\subset\mathcal L^2(\mathcal{Q}^\lambda,\psi).$ 
\end{proof}

\begin{prop}
The $*$-algebra $\mathcal{Q}^\lambda_c$ is a Tomita algebra with the inner 
product:
$\langle y|x\rangle_{\psi} =\psi(y^*x).$ 
Again we denote $x_h$ in place of $x(h)$ to simplify notation. In this setting 
we
have for $x\in \mathcal{Q}^\lambda_c$:\\
$(1)\;\;\;\text{Sharp:}\;\; S(x)_h=\alpha_h(\overline{x_{h^{-1}}});$\\
$(2)\;\;\;\text{Flat:}\;\;  F(x)_h=|h|\alpha_h(\overline{x_{h^{-1}}});$\\
$(3)\;\;\;\text{Delta:}\;\; \Delta(x)_h=|h|x_h.$ 
\end{prop}\label{tomita2}

\begin{proof}
This is really a corollary of Proposition \ref{tomita1}, as 
$\mathcal{Q}^\lambda_c$
is just a Tomita-subalgebra of $A^\lambda_c.$
\end{proof}

\section{The modular spectral triple of the algebra $\mathcal{Q}^\lambda$}

Having introduced the main features of the algebras $\mathcal{Q}^\lambda$, we
now turn briefly to the modular index theory of \cite{CNNR,CPR2,CRT}. We 
begin with some
semifinite preliminaries.

\subsection{Semifinite noncommutative geometry}
We need to explain some semifinite versions of standard definitions and
results following \cite{CPRS2}. 
Let $\phi$ be a fixed faithful, normal, semifinite trace
on a von Neumann algebra ${\mathcal N}$. Let ${\mathcal
K}_{\mathcal N }$ be the $\phi$-compact operators in ${\mathcal
N}$ (that is the norm closed ideal generated by the projections
$E\in\mathcal N$ with $\phi(E)<\infty$).

\begin{defn} A {\bf semifinite
spectral triple} $(\A,\HH,\D)$ is given by a Hilbert space $\HH$, a
$*$-algebra $\A\subset \cn$ where $\cn$ is a semifinite von
Neumann algebra acting on $\HH$, and a densely defined unbounded
self-adjoint operator $\D$ affiliated to $\cn$ such that
$[\D,a]$ is densely defined and extends to a bounded operator
in $\cn$ for all $a\in\A$ and $(\lambda-\D)^{-1}\in\K_\cn$ for all 
$\lambda\not\in{\R}.$
The triple is said to be {\bf even} if there is $\Gamma\in\cn$ such
that $\Gamma^*=\Gamma$, $\Gamma^2=1$,  $a\Gamma=\Gamma a$ for all
$a\in\A$ and $\D\Gamma+\Gamma\D=0$. Otherwise it is {\bf odd}.
\end{defn}

Note that if $T\in\cn$ and
$[\D,T]$ is bounded, then $[\D,T]\in\cn$. 

We recall from \cite{FK}
that if $S\in\mathcal N$, the {\bf t-th generalized singular
value} of $S$ for each real $t>0$ is given by
$$
\mu_t(S)=\inf\{\Vert SE\Vert\ : \ E \mbox{ is a projection in }
{\mathcal N} \mbox { with } \phi(1-E)\leq t\}.
$$
The ideal $\LL^1({\mathcal N},\phi)$ consists of those operators 
$T\in{\mathcal N}$ such that $\Vert T\Vert_1:=\phi( |T|)<\infty$ where
$|T|=\sqrt{T^*T}$. In the Type I setting this is the usual trace
class ideal. We will denote the norm on $\LL^1(\cn,\phi)$ by
$\n\cdot\n_1$. An alternative definition in terms of singular
values is that 
$T\in\LL^1(\cn,\phi)$ if $\|T\|_1:=\int_0^\infty \mu_t(T) dt<\infty.$
When ${\mathcal N}\neq{\mathcal
B}({\mathcal H})$, $\LL^1(\cn,\phi)$ need not be complete in this norm but it is
complete in the norm $\Vert\cdot\Vert_1 + \Vert\cdot\Vert_\infty$. (where
$\Vert\cdot\Vert_\infty$ is the uniform norm). We use the notation
$$
{\mathcal L}^{(1,\infty)}({\mathcal N},\phi)=
\left\{T\in{\mathcal N}\ : \Vert T\Vert_{_{{\mathcal L}^{(1,\infty)}}} :=   
\sup_{t> 0}\frac{1}{\log(1+t)}\int_0^t\mu_s(T)ds<\infty\right\}.
$$

 The reader
should note that ${\mathcal L}^{(1,\infty)}(\cn,\phi)$ is often taken to
mean an ideal in the algebra $\widetilde{\mathcal N}$ of
$\phi$-measurable operators affiliated to ${\mathcal N}$. Our
notation is however consistent with that of \cite{C} in the
special case ${\mathcal N}={\mathcal B}({\mathcal H})$. With this
convention the ideal of $\phi$-compact operators, 
${\mathcal  K}({\mathcal N})$,
consists of those $T\in{\mathcal N}$ (as opposed to
$\widetilde{\mathcal N}$) such that 
$\mu_\infty(T):=\lim_{t\to \infty}\mu_t(T)  = 0.$

\begin{defn}\label{summable} A semifinite  spectral triple
$(\A,\HH,\D)$ relative to $(\cn,\phi)$
with $\A$ unital is
$(1,\infty)$-summable if 
$(\D-\lambda)^{-1}\in\LL^{(1,\infty)}(\cn,\phi)\ \mbox{for all}\ 
\lambda\in\C\setminus\R.$ 
\end{defn}

It follows that if $(\A,\HH,\D)$ is $(1,\infty)$-summable then it is
$n$-summable (with respect to the trace $\phi$) for all $n>1$.   
We next need to briefly discuss Dixmier traces. For
more information on semifinite Dixmier traces, see \cite{CPS2,CRSS}.
For $T\in\LL^{(1,\infty)}(\cn,\phi)$, $T\geq 0$, the function 
\ben
F_T:t\to\frac{1}{\log(1+t)}\int_0^t\mu_s(T)ds 
\een 
is bounded. 
There are certain $\omega\in L^\infty(\R_*^+)^*$, \cite{CPS2,C}, which
define (Dixmier) traces on
 $\LL^{(1,\infty)}(\cn,\phi)$ by setting
$$ 
\phi_\omega(T)=\omega(F_T), \ \ T\geq 0
$$
and  extending to all of
$\LL^{(1,\infty)}(\cn,\phi)$ by linearity.
For each such $\omega$ we write $\phi_\omega$ for
the associated Dixmier trace.
Each Dixmier trace $\phi_\omega$ vanishes on the ideal of trace class
operators. Whenever the function $F_T$ has a limit at infinity,
all Dixmier traces return that limit as their value. 
This leads to the notion of a measurable operator \cite{C,LSS},
that is, one on which all Dixmier traces take the same value.

\subsection{The Kasparov module and modular spectral triple}

We have seen that the algebras $\mathcal {Q}^\lambda$ do not possess a 
faithful gauge invariant
trace but that there is a ${\rm KMS}_\beta$ where 
$\beta=-\log(\lambda)$ for the gauge 
action, $\gamma,$ namely $\psi:=\tau\circ\Phi:\mathcal{Q}^\lambda\to\C$,
where $\Phi:\mathcal{Q}^\lambda\to F^\lambda$ is the expectation and 
$\tau:F^\lambda\to\C$ is a faithful normalised trace. In fact, $\psi$ is the 
only KMS state for the gauge action (for any $\beta$), by Proposition 
\ref{KMS}.  
We show below that the generator of the gauge action $\D$
acting on a suitable $C^*$-$F^\lambda$-module $X$ gives us a Kasparov module
$(X,\D)$ whose class lies in 
$KK^{1,\T}(\mathcal{Q}^\lambda,F^{\lambda})$.  In some examples, including 
the case
$\lambda\in\mathbb{Q}$, we have $K_1(\mathcal {Q}^\lambda)=\{0\}$ and so pairing
with ordinary $K_1$ would be fruitless. However, 
following \cite{CPR2,CNNR} we may compute a numerical pairing 
using a `modular spectral triple' constructed from the Kasparov module.

We now review this construction adapted to the present situation.
Let $\HH=\LL^2(\mathcal {Q}^\lambda)$ be the GNS Hilbert space given by the 
faithful state
$\psi$ with the inner product on $\mathcal {Q}^\lambda$ defined by
$\la a,b\ra=\psi(a^*b)=(\tau\circ\Phi)(a^*b).$
Then $\D$ is a self-adjoint unbounded operator on $\HH$, \cite{CPR2}.
The representation of $\mathcal{Q}^\lambda$ on $\HH$ by left multiplication 
(which
we now denote by $\pi$ in place of $\pi_0$) is bounded and
nondegenerate: the left action of an element $a\in \mathcal{Q}^\lambda$ by 
$\pi(a)$ 
satisfies $\pi(a)b=ab$ for all $b\in \mathcal{Q}^\lambda.$  This distinction 
between 
elements of
$\mathcal{Q}^\lambda$ as vectors in $\LL^2(\mathcal{Q}^\lambda)$ and operators on 
$\LL^2(\mathcal{Q}^\lambda)$ is sometimes crucial.
The dense subalgebra $\mathcal Q_c^\lambda:=eA_c^\lambda e$
which is the finite span of elements in $\mathcal{Q}^\lambda$ of the form 
$\mathcal X_{[a,b)}\cdot\delta_g$
is in the smooth domain of the derivation 
$\delta=\mbox{ad}(|\D|)$.
We remind the reader that the {\rm KMS} condition on the modular automorphism 
group of the state
$\psi$, \cite{Ta}, (for $t=i$) is: $\psi(xy)=\psi(\s_i(\pi(y))x)=
\psi(\s(y)x)$
for $x,y\in\pi(\mathcal{Q}^\lambda),$ where $\sigma(y)=\Delta^{-1}(y).$ 

\begin{lemma} The group of modular
automorphisms  of the von Neumann algebra 
$\pi(\mathcal{Q}^\lambda)^{\prime\prime}$
is given on the generators by  
\be 
\s_t(\pi(f\cdot\delta_g)):=\Delta^{it}\pi(f\cdot\delta_g)\Delta^{-it}=
\pi(\Delta^{it}(f\cdot\delta_g))=|g|^{it}\pi(f\cdot\delta_g)
=\det(g)^{it}\pi(f\cdot\delta_g).
\label{eq:kms-flow}
\ee
\end{lemma}
\begin{proof} This is immediate from Lemma \ref{tomita2} if we note that 
$|g|=\det(g).$
\end{proof}

\begin{cor}\label{DeltaD} With $\mathcal{Q}^\lambda$ acting on 
$\mathcal{H}:=\LL^2(\mathcal{Q}^\lambda)$ and 
with $\D$ the generator of the natural unitary implementation of the gauge 
action of $\T^1$ on $\mathcal{Q}^\lambda,$  we have
$\Delta=\lambda^{\D}\  {\rm  or  }\ 
e^{it\D}=\Delta^{it/\log \lambda}.$
\end{cor}

To simplify notation, we let $A=\mathcal{Q}^\lambda$  and
$F=F^\lambda=A^\gamma$, the fixed point algebra for the $\T^1$ gauge
action, $\gamma.$ For convenience we will suppress the 
notations $\D\otimes 1_k$ and so on. 
The algebras $A_c,F_c$ are defined as the finite linear
span of the generators.
Right multiplication makes $A$ into a right $F$-module, and
similarly $A_c$ is a right module over $F_c$. We define an
$F$-valued inner product $(\cdot|\cdot)_R$ on both these modules
by
 $ (a|b)_R:=\Phi(a^*b).$
\begin{defn}\label{mod} Let $X$ be the right $F$ $C^*$-module obtained by
completing $A$ (or $A_c$) in the norm
$$ 
\Vert x\Vert^2_X:=\Vert (x|x)_R\Vert_F=\Vert
\Phi(x^*x)\Vert_F.
$$
\end{defn}
The algebra $A$ acting by left multiplication on  $X$
provides a representation of $A$ as adjointable operators on $X$.
Let $X_c$ be the copy of $A_c\subset X$. The $\T^1$ action on $X_c$ is
unitary and extends to $X$, \cite{CNNR,pr}.
For all $k\in\Z$, the  projection operator onto the $k$-th spectral subspace
of the $\T^1$ action is also denoted (somewhat carelessly) $\Phi_k$ on $X$:  
$$
\Phi_k(x)=\frac{1}{2\pi}\int_{\T^1}z^{-k}u_z(x)d\theta,
\ \ z=e^{i\theta},\ \ x\in X.
$$ 
Observe
 that $\Phi_0$ restricts to $\Phi$ on $A$
and on generators of $\mathcal{Q}^\lambda$ we have 
\be
\label{proj}\Phi_k(f\cdot\delta_g)=\left\{\begin{array}{ll} f\cdot\delta_g & 
\mbox{if $|g|=\lambda^k$}\\ 0 & \mbox{otherwise}\end{array}\right.
\ee
Of course $\mathcal L^2(\mathcal{Q}^\lambda)$ and $X$ have a common dense 
subspace 
$\mathcal Q_c^\lambda$ on which these projections are identical.
Let $A_k=\Phi_k (A)$ and observe from (\ref{proj}) that 
$\overline{A_k^*A_k}= F=\overline{A_kA_k^*}$ so that the gauge action
$\gamma$ on $\mathcal{Q}^\lambda$  has {\bf full spectral subspaces}.

We quote the following result from \cite{pr}, the proof in our case is the same.
\begin{lemma}\label{phiendo} The operators $\Phi_k$ are adjointable 
endomorphisms of the $F$-module $X$ such that $\Phi_k^*=\Phi_k=\Phi_k^2$ and
$\Phi_k\Phi_l=\delta_{k,l}\Phi_k$. If $K\subset\Z$ then the sum
$\sum_{k\in K}\Phi_k$ converges strictly to a projection in the
endomorphism algebra. The sum $\sum_{k\in\Z}\Phi_k$ converges to
the identity operator on $X$. For all $x\in X$, the sum
$x=\sum_{k\in\Z}\Phi_kx=\sum_{k\in\Z}x_k$ converges in $X$.
\end{lemma}
The unbounded operator of the next proposition is of course
the generator of the
$\T^1$ action on $X$.
We refer to Lance's book, \cite[Chapters 9,10]{L}, for information
on unbounded operators on $C^*$-modules.

\begin{prop}\label{dee}\cite{pr} Let $X$ be the right $C^*$-$F$-module of
Definition \ref{mod}. Define $\D:X_\D\subset X$ to be the linear space
$$ 
X_\D=
\{x=\sum_{k\in\Z}x_k\in X:\Vert \sum_{k\in\Z}k^2(x_k|x_k)_R\Vert<\infty\}.
$$
For $x\in X_\D$ define $ \D(x)=\sum_{k\in\Z}kx_k.$  Then $\D:X_\D\to X$ is a
is self-adjoint, regular operator on $X$.
\end{prop}

This should be compared to the following Hilbert space version.

\begin{prop}
The generator $\D$ of the one-parameter unitary group $\{u_z\;|\;z\in\T^1\}$
on $\mathcal L^2(\mathcal{Q}^\lambda,\psi)$ has eigenspaces given by the 
ranges of the
$\Phi_k$ and $\D(x)=kx$ iff $\Phi_k(x)=x.$ In particular
$$
{\rm dom}(\D)=\{x=\sum_k x_k\;|\;\Phi_k(x_k)=x_k\;\;\text{and}\;\; 
\sum_{k}k^2\Vert x_k\Vert^2 < \infty\},
$$
and $\D(\sum_k x_k)=\sum_k kx_k.$
\end{prop} 
\noindent{\bf Remark.}
On generators in $\mathcal{Q}^\lambda$ regarded as elements of either
$X$ or $L^2(\mathcal{Q}^\lambda,\psi)$
we have $\D(f\cdot\delta_g)=(\log_\lambda(|g|))f\cdot\delta_g.$

To continue, we recall the underlying right $C^*$-$F^\lambda$-module, $X$,
which is the completion of $\mathcal{Q}^\lambda$ for the norm $\Vert
x\Vert_X^2=\Vert\Phi(x^*x)\Vert_{F^\lambda}$.
Introduce the rank one operators on $X:$
$\Theta^{R}_{x,y}$ by $\Theta^{R}_{x,y}z=x(y|z)_R$. Then using the operators
$S_{k,m}$ defined above, we obtain formulas for the projections
$\Phi_k$ similar to those of 
\cite[Lemma 4.7]{pr} with some important differences. First recall 
\cite[Lemma 3.5]{CPR2}. 

\begin{lemma}
Any $F^\lambda$-linear endomorphism $T$ of the module $X$
which preserves the copy of $\mathcal{Q}^\lambda$ inside $X$, extends uniquely 
to a 
bounded 
operator on the Hilbert space $\HH=\LL^2(\mathcal{Q}^\lambda).$
\end{lemma}

In particular, the finite rank endomorphisms of the pre-$C^*$ module
$\mathcal{Q}^\lambda_c$ (acting on the left) 
satisfy this condition, and we denote the algebra of all these
endomorphisms by $End_F^{00}(\mathcal{Q}^\lambda_c)$.

\begin{lemma}\label{Phiformula}Compare \cite[Lemma 4.7]{pr}.
The following formulas hold in both 
$\mathcal L(X)$ and in $\mathcal B(\HH).$\\
{\rm (1)} For $k\geq 0,$ we have 
$$
\Phi_0=\Theta^R_{e,e}\;\;{\rm while\;\;for}\;\;k>0,\;\;
\Phi_k=\sum_{m=0}^{m_k}\Theta^R_{S_{k,m},S_{k,m}}.
$$
{\rm (2)} For $-k<0,$ we have
$$
\Phi_{-k}=\Theta^R_{S^*_{k,m},S^*_{k,m}}\;\;{\rm for\;\;any}
\;\;m=0,1,...,m_k-1\;\;{\rm and\ also\;\;for}\;\; m_k\;\;{\rm if}\;\; 
\lambda^{-k}=m_k+1.
$$
\end{lemma}

\begin{proof}
Since both $\Phi_k$ and the finite rank endomorphisms satisfy the hypotheses
of the previous lemma, the first statement of this lemma will follow from
calculations done on generators. The following calculations are based
on the formulas in Lemma \ref{Smusub}.\\
(1) Let $k>0$ and let $x=\sum_l x_l$ be a finite sum of generators, $x_l$
satisfying $\Phi_l(x_l)=x_l.$ Then
\bean 
\sum_{m=0}^{m_k}\Theta^R_{S_{k,m},S_{k,m}}(x) 
&=&\sum_l\sum_{m=0}^{m_k}\Theta^R_{S_{k,m},S_{k,m}}(x_l)
=\sum_l\sum_{m=0}^{m_k}S_{k,m}\Phi(S^*_{k,m}x_l)=
\sum_{m=0}^{m_k}S_{k,m}\Phi(S^*_{k,m}x_k)\\
&=&\sum_{m=0}^{m_k}S_{k,m}S^*_{k,m}x_k=ex_k=x_k=\Phi_k(x).
\eean
For $k=0$ this is a similar but far easier calculation.\\
(2) Let $-k<0$ and let $x=\sum_l x_l$ be a finite sum of generators as above.
 Then, for $0\leq m<m_k$
\bean 
\Theta^R_{S^*_{k,m},S^*_{k,m}}(x) 
&=&\sum_l\Theta^R_{S^*_{k,m},S^*_{k,m}}(x_l)
=\sum_lS^*_{k,m}\Phi(S_{k,m}x_l)=
S^*_{k,m}\Phi(S_{k,m}x_{-k})\\
&=&S^*_{k,m}S_{k,m}x_{-k}=ex_{-k}=x_{-k}=\Phi_{-k}(x).
\eean
\end{proof}

We recall the following result discussed in Section 3 of \cite {CNNR}
(a `bare hands' proof can be given by the method in \cite{CPR2}).
\begin{prop}\label{tildetrace} Let $\cn$ be the von Neumann algebra
$  \cn=(End^{00}_F(\mathcal Q_c^\lambda))'',$
where we take the commutant inside $\B(\HH)$. Then $\cn$ is 
semifinite, and there exists a faithful,
semifinite, normal trace $\tilde\tau:\cn\to\C$ such that for all rank
one endomorphisms $\Theta^{R}_{x,y}$ of $\mathcal{Q}^\lambda_c$, 
$$
\tilde\tau(\Theta^{R}_{x,y})=(\tau\circ\Phi)(y^*x),\ \ \ x,y\in 
\mathcal{Q}^\lambda_c.
$$
In addition, $\D$ is affiliated to $\cn$ and $\pi(\mathcal{Q}^\lambda)$ 
is a subalgebra of $\cn$.
\end{prop}
The fact that $\tilde\tau(\Phi_k)=\lambda^{-k}$
implies that with respect to the trace $\tilde\tau$ we can not expect 
$\mathcal D$ to satisfy a finite summability criterion. We solve this problem 
exactly as in \cite{CPR2}.

\begin{defn} 
We define a new weight on $\cn^+$: let $T\in\cn^+$ then
$\tau_\Delta(T):=\sup_N\tilde\tau(\Delta_N T)$ where 
$\Delta_N=\Delta(\sum_{|k|\leq N}\Phi_k).$ 
\end{defn}

\noindent{\bf Remarks}. Since $\Delta_N$ is 
$\tilde\tau$-trace-class, we see that $T\mapsto\tilde\tau(\Delta_N T)$
is a normal positive linear functional on $\cn$ and hence 
$\tau_\Delta$ is a normal weight on $\mathcal N^+$ which is easily seen to 
be faithful and semifinite. 

As in \cite{CPR2}, we now give another way to define 
$\tau_\Delta$ which is not only conceptually useful but also makes a
number of important properties straightforward to verify. Many proofs
require only trivial notation changes and the substitution of $n^{\pm}$
with $\lambda^{\mp}.$

\noindent {\bf Notation}. Let $\cm$ be the relative commutant
in $\mathcal N$ of the operator $\Delta$. Equivalently, $\cm$ is the
relative commutant of the set of spectral projections $\{\Phi_k | k\in\Z\}$ 
of $\D.$  Clearly, $\cm=\sum_{k\in\Z}\;\Phi_k\cn\Phi_k.$

\begin{defn} As $\tilde\tau$ restricted to each  $\Phi_k\cn\Phi_k$ is a 
faithful finite trace with $\tilde\tau(\Phi_k)=\lambda^{-k}$
we define $\widehat\tau_k$ on $\Phi_k\cn\Phi_k$ to be $\lambda^k$
times the restriction of $\tilde\tau.$ Then, 
$\widehat\tau:=\sum_k\widehat\tau_k$
on $\cm=\sum_{k\in\Z}\Phi_k\cn\Phi_k$ is a faithful normal semifinite
trace $\widehat\tau$ with $\widehat\tau(\Phi_k)=1$ for all $k.$ 
\end{defn}
We use $\widehat\tau$ to 
give an alternative  expression for $\tau_\Delta$ below

\begin{lemma}\label{commutes} An element $m\in\cn$ is in $\cm$ if and
only if it is in the fixed point algebra of the action, $\s_t^{\tau_\Delta}$
on $\cn$ defined for $T\in\cn$ by 
$\s_t^{\tau_\Delta}(T)=\Delta^{it}T\Delta^{-it}.$ Both  $\pi(F^\lambda)$ and 
the 
projections $\Phi_k$ belong to $\cM$. The map $\Psi :\cn\to\cm$ defined by
$\Psi(T)=\sum_k \Phi_k T\Phi_k$ is a conditional expectation onto $\cm$
and $\tau_{\Delta}(T)=\widehat\tau(\Psi(T))$ for all $T\in\cn^+.$ That is,
$\tau_{\Delta}= \widehat\tau\circ\Psi$ so that $\widehat\tau(T)=
\tau_{\Delta}(T)$ for 
all $T\in\cm^+.$ Finally, if one of $A,B\in\cm$ is $\widehat\tau$-trace-class
and $T\in\cn$ then $\tau_\Delta(ATB)=\tau_\Delta(A\Psi(T)B)=
\widehat\tau(A\Psi(T)B).$
\end{lemma}

\begin{proof}The proof is the same as the proof of \cite[Lemma 3.9]{CPR2}
with $\lambda^k$ in place of $n^{-k}.$
\end{proof}

\begin{lemma} The modular automorphism group $\s_t^{\tau_\Delta}$ of
$\tau_\Delta$ is inner and given by
$\s_t^{\tau_\Delta}(T)=\Delta^{it}T\Delta^{-it}$. The weight
$\tau_\Delta$ is a {\rm KMS} weight for the group $\s_t^{\tau_\Delta}$, and 
$\s_t^{\tau_\Delta}|_{\mathcal{Q}^\lambda}=\s_t^{\tau\circ\Phi}.$
\end{lemma}

\begin{proof} This follows from: \cite[Thm 9.2.38]{KR}, which gives
us the KMS properties of $\tau_\Delta$: the modular group is inner since
$\Delta$ is affiliated to $\cn.$ The final
statement about the restriction of the modular group to $\mathcal{Q}^\lambda$ 
is clear.
\end{proof}

We now have the key lemma:

\begin{lemma}\label{tracesplit}
Suppose $g$ is a function on $\R$ such that $g(\D)$ is $\tau_\Delta$ trace-class
in $\cm$, then for all $f\in F^\lambda$ we have
$$
\tau_\Delta(\pi(f)g(\D))=\tau_\Delta(g(D))\tau(f)=
\tau(f)\sum_{k\in\Z} g(k).
$$
\end{lemma}

\begin{proof}
First note that $\tau_\Delta(g(\D))=\widehat\tau(\sum_{k\in\Z}g(k)\Phi_k)
=\sum_{k\in\Z}g(k)\widehat\tau(\Phi_k)=\sum_{k\in\Z}g(k).$
We first do the computation for $f\in F^\lambda_c$ so that all the sums are 
finite.
Now, 
$$
\tau_\Delta(\pi(f)g(\D))=\widehat\tau(\pi(f)\sum_{k\in\Z}g(k)\Phi_k)
=\sum_{k\in\Z}g(k)\widehat\tau(\pi(f)\Phi_k)
$$
$$
=\sum_{k\in\Z}g(k)\widehat\tau_k(\pi(f)\Phi_k)
=\sum_{k\in\Z}g(k)\lambda^k\tilde\tau(\pi(f)\Phi_k).
$$
So it suffices to see for each $k\in\Z$, we have 
$\tilde\tau(\pi(f)\Phi_k)=\lambda^{-k}\tau(f).$

Now, by Theorem \ref{III-lambda} $\pi(F^\lambda)^{\prime\prime}$ is a 
type $II_1$ factor on $\HH$ whose unique trace say $Tr$ (with norm one) 
extends the trace $\tau$ on $F^\lambda$ in the sense that $Tr(\pi(f))=\tau(f).$
Since the projection $\Phi_k$
is in the commutant of the factor $\pi(F^\lambda)^{\prime\prime}$ the map
$$
T\in \pi(F^\lambda)^{\prime\prime}\mapsto T\Phi_k=\Phi_k T\Phi_k
$$ 
is a normal isomorphism
by \cite[Chapter 1, section 2, Prop. 2]{Dix} and so it has a unique
normalised trace also given by $Trace(T\Phi_k)=Tr(T).$ 
But $\tilde\tau(T\Phi_k)$ is a trace on
$\Phi_k\pi(F^\lambda)^{\prime\prime}\Phi_k=\pi(F^\lambda)^{\prime\prime}\Phi_k$ 
and so must be $\tilde\tau(\Phi_k)=\lambda^{-k}$ times the unique norm one
trace. That is, we get the required formula: 
$$
\tilde\tau(\pi(f)\Phi_k)=\lambda^{-k}Trace(\pi(f)\Phi_k)
=\lambda^{-k}Tr(\pi(f))=\lambda^{-k}\tau(f).
$$

So for $f\in F^\lambda_c,$ we have the formula: 
$$
\tau_\Delta(\pi(f)g(\D))=\tau_\Delta(g(D))\tau(f)=
\sum_{k\in{\Z}}g(k)\tau(f).
$$
Now, the right hand side is a norm-continuous function of $f$.
To see that the left side is norm-continuous we do it in more generality.
Let $T\in\cn$, then since $\widehat\tau$ is a trace on $\cm$ we get:
$$
|\tau_{\Delta}(Tg(\D))|=|\widehat\tau(\Psi(Tg(\D))|=
|\widehat\tau(\Psi(T)g(\D))|\leq
\Vert \Psi(T)\Vert \widehat\tau(|g(\D)|)\leq
\Vert T\Vert \widehat\tau((|g(\D)|)=
\Vert T\Vert \tau_{\Delta}(|g(\D)|).
$$
That is the left hand side is norm-continuous in $T$ and so we have 
the formula:
$$
\tau_\Delta(\pi(f)g(\D))=\tau_\Delta(g(\D))\tau(f)=
\sum_{k\in{\Z}}g(k)\tau(f)
$$ 
for all $f\in F^\lambda.$
\end{proof}

\begin{prop}\label{dixycomp} (i) We have 
$(1+\D^2)^{-1/2}\in\LL^{(1,\infty)}(\cM,\tau_\Delta)$. That is,
$\tau_\Delta((1+\D^2)^{-s/2})<\infty$ for all $s>1.$ Moreover, for
all $f\in F^\lambda$
$$
\lim_{s\to 1^+}(s-1)\tau_{\Delta}(\pi(f)(1+\D^2)^{-s/2})=2\tau(f)
$$
so that $\pi(f)(1+\D^2)^{-1/2}$ is a measurable operator in the sense of 
\cite{C}.

(ii)  For $\pi(a)\in \pi(\mathcal{Q}^\lambda)\subset \cn$ the following (ordinary)
limit exists and
$$
\widehat\tau_\omega(\pi(a))=\frac{1}{2}\lim_{s\to 1^+}(s-1)
\tau_\Delta(\pi(a)(1+\D^2)^{-s/2})=\tau\circ\Phi(a),
$$
the original {\rm KMS} state $\psi=\tau\circ\Phi$ on $\mathcal{Q}^\lambda.$
\end{prop}
\begin{proof} (i) This proof is identical to \cite[Proposition 3.12]{CPR2}.\\
(ii) This proof is the same as \cite[Proposition 3.14]{CPR2}
with $\mathcal{Q}^\lambda , F^\lambda$ replacing $O_n, F.$ 
\end{proof}

\begin{defn}\label{modspectrip} 
The triple $(\A,\HH,\D)$ along with $\gamma,\ \psi,\ \cn,\ \tau_\Delta$ 
satisfying properties (0) to (3) below
is called the  {\bf modular spectral triple} of the dynamical system 
$(\mathcal{Q}^\lambda,\gamma,\psi)$ \\
0) The $*$-subalgebra $\A=\mathcal Q_c^\lambda$ of the algebra 
$\mathcal{Q}^\lambda$ is
faithfully represented in $\cn$ with the latter acting on the  Hilbert space  
$\HH=\LL^2(\mathcal{Q}^\lambda,\psi)$,\\
1) there is a faithful normal semifinite weight $\tau_\Delta$
on $\cn$ such that the modular automorphism group of $\tau_\Delta$ is an
inner automorphism group $\sigma_t$ (for $t\in\C$) of (the Tomita algebra of) 
$\cn$ with
$\sigma_i|_\A=\s$ in the sense that $\s_i(\pi(a))=\pi(\s(a)),$
where $\s$ is the automorphism $\s(a)=\Delta^{-1}(a)$ on $\A$,\\
2) $\tau_\Delta$ restricts to a faithful semifinite trace $\widehat\tau$ 
on $\cM=\cn^\s$, with a faithful normal projection $\Psi: \cn\to\cm$
satisfying $\tau_\Delta=\widehat\tau\circ\Psi$ on $\cn$,\\
3) with $\D$ the generator of the one parameter group implementing the
gauge action of $\T$ on $\HH$ we have:\;\;
 $[\D,\pi(a)]$ extends to a bounded operator (in $\cn$) for all $a\in\A$
and for $\lambda$ in the resolvent set of $\D$,
$(\lambda-\D)^{-1} \in \K(\cM,\tau_\Delta)$, where  
$\K(\cM,\tau_\Delta)$ is the ideal of compact operators in $\cM$ relative to
$\tau_\Delta$. In particular, $\D$ is affiliated to $\cM$.
\end{defn} 

For matrix 
algebras $\A= \mathcal Q_c^\lambda\otimes M_k$ over $\mathcal Q_c^\lambda$,
$(\mathcal Q_c^\lambda\otimes M_k,\HH\otimes M_k ,\D\otimes Id_k)$
is also a modular spectral triple in the obvious fashion. 

We need some technical lemmas
for the discussion in the next Section. A function $f$ from a complex domain 
$\Omega$ into a Banach space $X$ is called {\bf holomorphic} if it is complex
differentiable in norm on $\Omega.$ The following is proved in 
\cite[Lemma 3.15]{CPR2}.
\begin{lemma}

(1) Let $\mathcal B$ be a $C^*$-algebra and let 
$T\in\mathcal B^+.$ The mapping $z\mapsto T^z$ is holomorphic (in operator norm)
in the half-plane $Re(z)>0.$\\
(2) Let $\mathcal B$ be a von Neumann algebra with faithful normal 
semifinite trace $\phi$ and let $T\in\mathcal B^+$ be in 
$\LL^{(1,\infty)}(\mathcal B,\phi).$ Then, the mapping $z\mapsto T^z$ is 
holomorphic (in trace norm) in the half-plane $Re(z)>1.$\\
(3) Let $\mathcal B$, and $T$ be as in item (2) and let $A\in \mathcal B$
then the mapping $z\mapsto \phi(A T^z)$ is holomorphic for $Re(z)>1.$
\end{lemma}

\begin{lemma}\label{phiclosed}
In these modular spectral triples $(\A,\HH,\D)$ for matrices over the
algebras $\mathcal{Q}^\lambda$ we have
$(1+\D^2)^{-s/2}\in\LL^1(\cM,\tau_\Delta)$ for all
$s> 1$ and for $x\in \cn,$ 
$\tau_\Delta(x(1+\D^2)^{-r/2})$ is holomorphic for 
$Re(r)>1$ and we have for $a\in \mathcal Q_c^\lambda$,
$\tau_\Delta([\D,\pi(a)](1+\D^2)^{-r/2})=0,$ for $Re(r)>1.$
\end{lemma}
\begin{proof} We include a brief proof since there are some small but 
important differences from \cite[Lemma 3.16]{CPR2}. 
Since the eigenvalues for $\D$ are precisely the set of
integers, and the projection $\Phi_k$ on the eigenspace with eigenvalue $k$
satisfies $\tau_\Delta(\Phi_k)=1,$ it is clear that 
$(1+\D^2)^{-s/2}\in\LL^1(\cM,\tau_\Delta).$
Now, $\tau_\Delta(x(1+\D^2)^{-r/2})=
\widehat\tau(\Psi(x)(1+\D^2)^{-r/2})$ is holomorphic for $Re(r)>1$
by item (3) of the previous lemma.

To see the last statement, we observe that 
$\tau_\Delta([\D,\pi(a)](1+\D^2)^{-r/2})=
\tau_\Delta(\Psi([\D,\pi(a)])(1+\D^2)^{-r/2}),$ so it suffices to see that
$\Psi([\D,\pi(a)])=0$ for $a\in \A=\mathcal Q_c^\lambda.$   To this end, 
let $a=f\cdot\delta_g$ where $det(g)=\lambda^n$
is one of the linear generators of $\mathcal Q_c^\lambda.$  Then by calculating 
the action
of the operator $\D\pi(f\cdot\delta_g)$ on the linear 
generators $f_i\cdot\delta_{h_i}$ of the Hilbert space, $\HH$, 
we obtain:
$$
\D\pi(f\cdot\delta_g)=n\pi(f\cdot\delta_g)
+\pi(f\cdot\delta_g)\D\;\;\;{\rm that\;\; is}\;\;\;
[\D,\pi(f\cdot\delta_g)]=
\log_{\lambda}(|g|)\pi(f\cdot\delta_g).
$$
More generally,
$$
[\D,\pi(\sum_{i=1}^m c_i f_i\cdot\delta_{h_i})]=
\sum_{i=1}^m c_i(\log_{\lambda}(|h_i|))\pi(f_i\cdot\delta_{h_i}).
$$
If we apply $\Psi$ to this equation, we see that 
$\Psi(\pi(f_i\cdot\delta_{h_i}))=\pi(\Phi(f_i\cdot\delta_{h_i}))=0$
whenever $\log_{\lambda}(|h_i|)\neq 0,$ and so the whole sum is $0$. We also
observe that $[\D,\pi(a)]\in \pi(\mathcal Q_c^\lambda)$ for all 
$a\in \mathcal Q_c^\lambda.$ 
This is not
too surprising since $\D$ is the generator of the action $\gamma$ of $\T$ on
$\mathcal{Q}^\lambda.$
\end{proof}

\subsection{Modular $K_1$}

We now make appropriate modifications to \cite[Section 4]{CPR2}) using 
\cite{CNNR}
introducing elements of these modular spectral triples 
$(\A,\HH,\D)$ (where $\A$ is a matrix algebra over 
$\mathcal Q_c^\lambda$) that will have a
well defined pairing with our Dixmier functional
 $\widehat\tau_{\omega}$. Let
$A=\mathcal{Q}^\lambda$.
Following \cite{HR} we say that a unitary (invertible, projection,...) 
in the $n\times n$ matrices  over $\mathcal{Q}^\lambda$ for some $n$ is 
a unitary (invertible, projection,...) {over} $\mathcal{Q}^\lambda$.
We write $\s_t$ for the automorphism $\s_t\otimes Id_n$ of $\A$.

\begin{defn} Let $v$ be a partial isometry in the $*$-algebra $\A$. We say 
that $v$
satisfies the {\bf modular condition} with respect to $\s$ 
if the operators
$ v\s_t(v^*)$ and $v^*\s_t(v)$
are in the fixed point algebra $F\subset\A$ for all $t\in\Rl$.
Of course, any partial isometry in $F$
is a modular partial isometry. 
\end{defn}

\begin{lemma}\label{noinv} (\cite[Lemma 4.8]{CPR2}) Let  $v\in \A$ be
a modular partial
isometry. Then we have
$$
u_v=\bma 1-v^*v & v^*\\ v & 1-vv^*\ema
$$
is a modular unitary over $\A$. Moreover there is a modular homotopy
$u_v\sim u_{v^*}$.
\end{lemma}

Note that in \cite{CPR2} we used a different approach which is implied by the 
one
given here. In \cite{CPR2} we defined modular unitaries in terms of the
regular automorphism:
$$
\pi(\s(a))=\pi(\Delta^{-1}(a))=\Delta^{-1}\pi(a)\Delta=\s_i(\pi(a)).
$$
That is we said that a unitary in $\A$ is modular if $u\s(u^*)$ and $u^*\s(u)$ 
are 
in the fixed point algebra.

{\bf Examples}.\\ 
(1) For $k,j>0$ recall $S_{k,m}\in \mathcal Q_c^\lambda$ with $m<m_k$ 
(see Definition
\ref{Sk,m}) we write 
$P_{k,m}=S_{k,m}S_{k,m}^*=\mathcal X_{[m\lambda^k,(m+1)
\lambda^k)}\cdot\delta_1$ which is in clearly $F^\lambda$. Then 
for each
$\{k,m\},\{j,n\}$ we have a unitary
$$ 
u_{\{k,m\},\{j,n\}}=\bma 1-P_{k,m} & S_{k,m} S_{j,n}^*\\ S_{j,n} S_{k,m}^* &
1-P_{j,n}\ema.
$$ 
It is simple to check that this a self-adjoint
unitary satisfying the modular condition, and that $\tau(P_{k,m})=\lambda^k$
and $\tau(P_{j,n})=\lambda^j.$ These examples behave very much like the 
$S_\mu S_\nu^*$ examples of \cite{CPR2}.\\

(2) For $k,j>0$ consider the ``leftover'' partial isometries $S_{k,m_k}$
and $S_{j,m_j}$ of Definition 3.13 which we will denote by $S_k$ and $S_j$ to 
lighten the notation. We let $v_{j,k}=S_jS_k^*$ and calculate its range and 
initial projections which are both in $F^\lambda$:
$$
P_j = S_jS_k^*S_kS_j^*=\mathcal X_{[m_j\lambda^j , m_j\lambda^j + 
\lambda^j(\lambda^{-k} - m_k))}\cdot\delta_1,\;\;{\rm and}
$$
$$
P_k = S_kS_j^*S_jS_k^*=\mathcal X_{[m_k\lambda^k , m_k\lambda^k + 
\lambda^k(\lambda^{-j} - m_j))}\cdot\delta_1.
$$

We note for future reference that:
$$
\tau(P_j)=\lambda^j(\lambda^{-k} - m_k)\;\;{\rm and}\;\;
\tau(P_k)=\lambda^k(\lambda^{-j} - m_j).
$$
We also note that we have a modular unitary $u_{j,k}$:
$$
u_{j,k}=\bma 1-P_{k} & S_{k} S_{j}^*\\ S_{j} S_{k}^* & 1-P_{j}\ema.
$$ 
Define the modular $K_1$ group  as follows.

\begin{defn}\label{def:modular K1} Let $K_1(A,\s)$ be the abelian group with one
generator $[v]$ for each partial isometry $v$ over~$A$ satisfying the
modular condition and with the following relations: 
\bean 
1)&& [v]=0\
\mbox{if}\ v\ \mbox{is over}\ F,\nno 2)&& [v]+[w]=[v\oplus w],\nno 3)&&
\mbox{if }v_t,\ t\in[0,1],\ \mbox{is a continuous path of modular
partial isometries in some matrix algebra over }A\nno && \mbox{then}\
[v_0]=[v_1].
\eean
\end{defn}

One could use modular unitaries as in \cite{CPR2} in place of these modular
partial isometries. 

The following can now be seen to hold.
\begin{lemma}\label{centre} (Compare \cite[Lemma 4.9]{CPR2}) Let 
$(\A,\HH,\D)$ be 
our modular spectral triple relative
to $(\cn,\tau_\Delta)$ 
and set $F=\A^\s$ and $\s:\A\to\A$.
Let $L^\infty(\Delta)=L^\infty(\D)$ be the von Neumann algebra generated by 
the spectral projections of $\Delta$
then $L^\infty(\Delta)\subset{\mathcal Z}(\cM)$.
Let $v\in\A$ be a partial isometry with $vv^*,\,v^*v\in F$. 
Then $\pi(v)Q\pi(v^*)\in\cM$ and 
$\pi(v^*)Q\pi(v)\in\cM$ for all spectral 
projections $Q$ 
of $\D$, if and only if $v$ is modular. That is, $\pi(v)\Delta \pi(v^*)$ 
and $\pi(v^*)\Delta \pi(v)$
(or $\pi(v)\D \pi(v^*)$ and $\pi(v^*)\D \pi(v)$) 
are both affiliated to $\cM$ if and only if $v$ is modular.
\end{lemma}

Thus we see that modular
partial isometries conjugate $\Delta$ to an operator affiliated to $\cM$, and so 
$v\Delta v^*$ commutes with $\Delta$ (and $v\D v^*$ commutes with $\D$). 

We will next show  that there is an analytic pairing between (part of) modular
$K_1$ and modular spectral triples. To do this, we are going to use
the analytic formulae for spectral flow 
in \cite{CP2}.

\subsection{The mapping cone algebra}
Our aim in the remainder is to calculate an index pairing explicitly for
the matrix algebras $\A$ over the smooth subalgebra $\mathcal Q_c^\lambda$
of $\mathcal{Q}^\lambda$. 
In the following few pages we will sometimes abuse notation and write $a$ 
in place of $\pi(a)$ for $a\in\A$ in order to make our formulae more readable. 
Whenever we do this, however, we will use $\s_i(\cdot)=\Delta^{-1}(\cdot)\Delta$
the spatial version of the algebra homomorphism, $\s$. We will generally use
the spatial version $\s_i$ when in the presence of operators not in $\pi(\A).$

 We briefly review some results from \cite{CNNR}, that provide
an interpretation of the modular index pairing
given by the spectral flow.

If $F\subset A$ is a sub-$C^*$-algebra of the $C^*$-algebra $A$, then
the mapping cone algebra for the inclusion is 
$$
M(F,A)=\{f:\R_+=[0,\infty)\to A: f\ \mbox{is continuous and vanishes at
infinity},\ f(0)\in F\}.
$$
When $F$ is an ideal in $A$ it is known that $K_0(M(F,A))\cong
K_0(A/F)$, \cite{Put1}.  In general, $K_0(M(F,A))$ is the set of
homotopy classes of partial isometries $v\in M_k(A)$ with range and
source projections $vv^*,\ v^*v$ in $M_k(F)$, with operation the
direct sum and inverse $-[v]=[v^*]$. All this is proved in \cite{Put1}. 

It is shown in \cite{CNNR} that there is a natural map that injects 
$K_1(A,\sigma)$ into 
$K_0^\T(M,F)$, the equivariant $K$-theory of the mapping cone algebra. Note 
that the
$\T$ action on $A$ lifts in the obvious way to the mapping cone.
Now, it was shown in \cite{CPR1} that certain Kasparov
$A,F$-modules extend to Kasparov $M(F,A),F$-modules, and this was 
extended to the equivariant case in \cite{CNNR}.
Importantly the theory applies to
the equivariant Kasparov module coming from a circle action.
The extension is explicit, namely  there is a pair
$(\hat{X},\hat\D)$ which is a graded unbounded Kasparov module for
the mapping cone algebra $M(F,A)$ constructed using a generalised APS
construction, \cite{APS3}.

If $v$ is a partial isometry in  $M_k(\A)$, setting 
$$
e_v(t)=\bma
1-\frac{vv^*}{1+t^2} & -iv\frac{t}{1+t^2}\\iv^*\frac{t}{1+t^2}&
\frac{v^*v}{1+t^2}\ema,
$$
defines $e_v$ as a projection
over $M(F,A)$.

Then in \cite{CNNR} we showed that if $v\in \A$ is a modular partial isometry 
we have
\bea
\la [e_v]-\left[\bma 1 & 0\\ 0 & 0\ema\right],
[(\hat{X},\hat{\D})]\ra&=&-{\rm Index}(PvP:v^*vP(X)\to vv^*P(X))\in K_0(F)\nno
&=&{\rm Index}(Pv^*P:vv^*P(X)\to v^*vP(X))\in K_0^\T(F).
\label{pisomindex}
\eea

We thus obtain an index map $K_1(A,\sigma)\to K_0^\T(F)$. The latter may
be thought of as the ring of Laurent polynomials $K_0(F)(\chi,\chi^{-1})$
where we think of $\chi,\chi^{-1}$ as generating the representation ring of $\T$.
We may obtain a real valued invariant from this map by evaluating
$\chi$ at $e^{-\beta}$ where $\beta$ is the inverse temperature of our KMS state
and applying the trace to the resultant element of $K_0(F)$.
Then one of the main results of \cite{CNNR} is that the real valued invariant
so obtained is identical with the spectral flow invariant of the next subsection.
However the general theory of \cite{CNNR} does not tell us the range of this
index map and it is the latter that is of interest for these explicit 
calculations.

\subsection{A local index formula for the algebras $\mathcal{Q}^\lambda$}

Using the fact that we have full spectral subspaces we know from \cite{CNNR} 
that there is a formula for spectral flow which is analogous to the local 
index formula in noncommutative geometry. We remind the reader that 
$\tau_\Delta=\widehat\tau\circ\Psi$ where $\Psi:\cn\to\cm$ is the canonical
expectation, so that $\tau_\Delta$ {\bf restricted to} $\cm$ is $\widehat\tau.$

\begin{thm}\label{first} (Compare \cite[Theorem 5.5]{CPR2}) Let
$(\A,\HH,\D)$ be the $(1,\infty)$-summable, 
modular spectral triple for the algebra $\mathcal{Q}^\lambda$
we have constructed previously.
 Then for any modular
partial isometry $v$
and for any
Dixmier trace $\widehat{\tau}_{\tilde\omega}$ associated to $\widehat\tau$, 
we have
spectral flow as an actual limit
$$
sf_{\widehat\tau}(vv^*\D,v\D v^*)=\frac{1}{2}\lim_{s\to 1+}
(s-1)\widehat\tau(v[\D,v^*](1+\D^2)^{-s/2})
=\frac{1}{2}\widehat{\tau}_{\tilde\omega}(v[\D,v^*](1+\D^2)^{-1/2})
=\tau\circ\Phi(v[\D,v^*]).
$$
The functional on $\A\otimes\A$ defined by
$a_0\otimes a_1\mapsto
\frac{1}{2}\lim_{s\to 1^+}(s-1)\tau_\Delta(a_0[\D,a_1](1+\D^2)^{-s/2})$
is a $\s$-twisted $b,B$-cocycle (see the proof below for the definition). 
\end{thm}

\noindent{\bf Remark}. Spectral flow in this setting is independent of the path 
joining the endpoints of unbounded self adjoint operators affiliated to 
$\mathcal M$ however
it is not obvious that this is enough to show that it is constant on homotopy 
classes of modular unitaries. This latter fact is true but the 
proof is lengthy so we refer to \cite{CNNR}.

\begin{thm} We let 
$(\mathcal Q_c^\lambda\otimes M_2,\HH\otimes\C^2,\D\otimes 1_2)$ 
be the modular spectral triple of $(\mathcal Q_c^\lambda\otimes M_2).$\\
{\rm (1)} Let $u$ be  a modular unitary defined in Section 5 of the form
$$ 
u_{\{k,m\},\{j,n\}}=\bma 1-P_{k,m} & S_{k,m} S_{j,n}^*\\ S_{j,n} S_{k,m}^* &
1-P_{j,n}\ema.
$$
Then the spectral flow is positive being given by
\bean sf_{\tau_\Delta}(\D, u\D u^*)&=&
(k-j)(\lambda^j-\lambda^k)
\in \Z[\lambda]\subset\Gamma_\lambda.\eean
{\rm (2)} Let $u$ be a modular unitary defined in Section 5 of the form:
$$
u_{j,k}=\bma 1-P_{k} & S_{k} S_{j}^*\\ S_{j} S_{k}^* &
1-P_{j}\ema,
$$  
where $S_{k} S_{j}^*=S_{k,m_k} S_{j,m_j}^*$  and $P_k$ and $P_j$ are its 
range and initial projections, respectively. Then the spectral flow is 
 given by
\bean 
sf_{\tau_\Delta}(\D, u\D u^*)&=&
(k-j)[\lambda^j(\lambda^{-k}-m_k)-\lambda^k(\lambda^{-j}-m_j)]
\in \Gamma_\lambda.
\eean
\end{thm}

\begin{proof} We have already observed that these are, in fact modular
unitaries.
For the computations we use a calculation from the proof of 
Lemma \ref{phiclosed}
to get in example {\rm (1)}:
\bean 
&&u[\D\otimes 1_2 ,u]
=\bma 1-P_{k,m} & S_{k,m} S_{j,n}^*\\ S_{j,n} S_{k,m}^* &
1-P_{j,n}\ema\bma 0 & {[\D ,S_{k,m} S_{j,n}^*]}\\
{[\D ,S_{j,n} S_{k,m}^* ]} & 0\ema\\
&=&\bma 1-P_{k,m} & S_{k,m} S_{j,n}^*\\ S_{j,n} S_{k,m}^* &
1-P_{j,n}\ema\bma 0 & (k-j)S_{k,m} S_{j,n}^*\\
(j-k)S_{j,n} S_{k,m}^* & 0\ema=(k-j)\bma -P_{k,m} & 0\\ 0 &
P_{j,n}\ema.
\eean
So using Theorem \ref{first} and our previous computation of the
Dixmier trace, Proposition \ref{dixycomp}, and the fact that
$P_{k,m}=S_{k,m}S_{k,m}^*=\mathcal X_{[m\lambda^k,(m+1)
\lambda^k)}\cdot\delta_1$ so that $\tau(P_{k,m})=\lambda^k$
we have
$$
sf_{\tau_\Delta}(\D,u_{k,m}\D u_{k,m})=(k-j)\tau(P_{j,n}-P_{k,m})
=(k-j)(\lambda^j-\lambda^k).
$$
This number is always positive as the reader may check, and is
contained in $\Z[\lambda].$

The computations in example {\rm (2)} 
are similar and use the fact that
$P_k =\mathcal X_{[m_k\lambda^k , m_k\lambda^k + 
\lambda^k(\lambda^{-j} - m_j))}\cdot\delta_1,$ so that
$\tau(P_k)=\lambda^k(\lambda^{-j} - m_j)\in\Gamma_\lambda.$ In these 
examples, the spectral flow is {\bf not} contained in the smaller polynomial
ring, $\Z[\lambda].$
\end{proof}
\begin{rems*}
The observation of \cite{CPR2} that the twisted residue cocycle formula for
spectral flow is calculating Araki's relative entropy of two KMS states 
\cite{Ar} also applies to the examples in this subsection.
\end{rems*}
\noindent{\bf Acknowledgements} We would like to thank 
Nigel Higson, Ryszard Nest, Sergey Neshveyev, Marcelo Laca, Iain Raeburn 
and Peter Dukes for advice and comments. 
The first and fourth named authors
were supported by the Australian Research
Council. The second and third named authors acknowledge the support of
NSERC (Canada).


\begin{thebibliography}{0000000}

\bibitem[Ar]{Ar} H. Araki,
{\em Relative entropy of states of von Neumann algebras}, Publ. RIMS, Kyoto
Univ., {\bf 11} (1976) 809--833 and {\em Relative entropy for states of von 
Neumann algebras II}, Publ. RIMS, Kyoto Univ., {\bf 13} (1977) 173--192.

\bibitem[APS3]{APS3}  M.F. Atiyah, V.K. Patodi, I.M. Singer, {\em Spectral
asymmetry and Riemannian geometry. III},
Math. Proc. Camb. Phil. Soc. {\bf 79} (1976) 71--99.

\bibitem[BR1]{BR} O. Bratteli, D. Robinson, {\em Operator algebras and
quantum statistical mechanics 1}, Springer-Verlag, 2nd Ed, 1987.

\bibitem[BR2]{BR2}  O. Bratteli, D. Robinson, {\em Operator algebras and
quantum statistical mechanics 2}, Springer-Verlag, 2nd Ed, 1987.

\bibitem[CNNR] {CNNR} A.L. Carey,
R. Nest, S. Neshveyev, A. Rennie, {\em Twisted cyclic theory, equivariant 
$KK$-theory and KMS states}, to appear in J. reine angew. Math.

\bibitem[CP1]{CP1} A. L. Carey, J. Phillips, {\em Unbounded Fredholm
modules and spectral flow}, Canadian J. Math., {\bf 50} (4) (1998) 673--718.

\bibitem[CP2]{CP2} A. L. Carey, J. Phillips,
{\em Spectral flow in $\theta$-summable
Fredholm modules, eta invariants and the JLO cocycle}, $K$-Theory, 
{\bf 31} (2004) 135--194.

\bibitem[CPR1]{CPR1} A.L. Carey, J. Phillips, A. Rennie, {\em A noncommutative
Atiyah-Patodi-Singer index theorem in $KK$-theory}, to appear in 
J. reine angew. Math.

\bibitem[CPR2]{CPR2} A.L. Carey, J. Phillips, A. Rennie, {\em Twisted cyclic 
theory
and an index theory for the gauge invariant {\rm KMS} state on Cuntz algebras},
to appear in Journal of $K$-Theory.

\bibitem[CPS2]{CPS2} A.L. Carey, J. Phillips, F. Sukochev,
{\em Spectral flow and Dixmier traces},
 Adv. Math, {\bf 173} (2003) 68--113.

\bibitem[CPRS2]{CPRS2} A.L. Carey, J. Phillips, A. Rennie, F. Sukochev,
{\em The local index formula in semifinite von Neumann algebras I:
Spectral Flow},  Adv. in Math. {\bf 202} (2006) 451--516.

\bibitem[CRSS]{CRSS} A. L. Carey, A. Rennie, A. Sedaev, F. Sukochev, 
{\em Dixmier 
traces and asymptotics of zeta functions}, {\bf somewhere}

\bibitem[CRT]{CRT}  A.L. Carey, A. Rennie, K. Tong, {\em Spectral flow 
invariants 
and twisted cyclic theory from the Haar state on  $SU_q(2)$}, to appear 

\bibitem[C0]{C0} A. Connes, {\em Une classification des facteurs de type III},
Annales Scientifiques de le'Ecole Norm. Sup., 4em serie t. 6 (1973) 18--252.

\bibitem[C]{C} A. Connes, {\em Noncommutative geometry}, Academic Press, 1994.

\bibitem[Cu]{Cu} J. Cuntz, {\em Simple $C^*$-algebras generated by isometries}, 
Commun. Math. Phys, {\bf 57} (1977) 173--189.

\bibitem[Cu1]{Cu1} J. Cuntz, {\em $C^*$-algebras associated with the 
$ax+b$-semigroup over $\mathbb N$}
in{\em 
K-Theory and 
Noncommutative 
Geometry}, Eds, G. Cortinas, J. Cuntz, M Karoubi, R. Nest, C.A. Weibel,
EMS Series of Congress Reports, Volume 2.

\bibitem[D]{Dix} J. Dixmier, {\em von Neumann algebras},
North-Holland, 1981.

\bibitem[E]{E} G. Elliott, {\em Some simple $C^*$-algebras constructed as
crossed products with discrete outer automorphism groups}, Publ. RIMS Kyoto
Univ., {\bf 16} (1980) 299-311. 

\bibitem[FK]{FK} T. Fack and H. Kosaki, {\em Generalised $s$-numbers of
$\tau$-measurable operators}, Pac. J. Math., {\bf 123} (1986) 269--300.

\bibitem[HR]{HR} N. Higson, J. Roe, {\em Analytic $K$-homology},
Oxford University Press, 2000.

\bibitem[KNR]{NR} J. Kaad, R. Nest, A. Rennie, {\em $KK$-theory and spectral 
flow in von Neumann algebras}, 
arXive:math.OA/0701326.

\bibitem[KR]{KR} R.V. Kadison, J. R. Ringrose, {\em Fundamentals of the 
theory of operator algebras, Vol. II: advanced theory}, Academic Press, 1986.


\bibitem[K]{K} G. G. Kasparov, {\em The operator $K$-functor and
extensions of $C^*$-algebras}, Math. USSR. Izv. {\bf 16} No. 3
(1981) 513--572.

\bibitem[KMT]{KMT} J. Kustermans, G. Murphy, L. Tuset, {\em Differential
calculi over quantum groups and twisted cyclic cocycles},
J. Geom. Phys., {\bf 44} (2003) 570--594.

\bibitem[L]{L} E. C. Lance,
{\em Hilbert $C^*$-modules}, Cambridge University Press, Cambridge,
1995.

\bibitem[LS]{LS} M. Laca, J. Spielberg, {\em Purely infinite $C^*$-algebras
from boundary actions of discrete groups}, Journal f\"{u}r die reine und
ang. Mathematik, {\bf 480} (1996) 125--139.

\bibitem[LSS]{LSS} S. Lord, A. Sedaev, F. A. Sukochev, {\em Dixmier
traces as singular symmetric functionals and applications to
measurable operators}, Journal of Functional Analysis, {\bf 224}  no.1 (2005)
 72--106.
\bibitem[PR]{pr} D. Pask, A. Rennie, {\em The noncommutative geometry of
graph $C^*$-algebras I: The index theorem}, Journal of Functional Analysis, 
{\bf 233} (2006) 92--134.

\bibitem[PhR]{PhR} J. Phillips, I. Raeburn, {\em Semigroups of isometries,
Toeplitz algebras and twisted crossed products}, J. Int. Equat. and Op. Th.,
{\bf 17} (1993) 579-602.

\bibitem[Ped]{Ped} G. K. Pedersen, {\em $C^*$-algebras and their automorphism
groups}, London Math. Soc. monographs {\bf 14}, Academic Press,
London 1979.

\bibitem[PT]{PT} G. K. Pedersen, M. Takesaki, {\em The Radon-Nikodym theorem
for von Neumann algebras}, Acta Math., {\bf 130} (1973) 53--87.

\bibitem[Put1]{Put1} I. Putnam, {\em An excision theorem for the
$K$-theory of $C^*$-algebras}, J. Operator Theory, {\bf 38} (1997) 151--171.

\bibitem[Put2]{Put2} I. Putnam, {\em On the $K$-theory of $C^*$-algebras of
principal groupoids}, Rocky Mountain J. Math., {\bf 28} no. 4 (1998) 1483--1518.

\bibitem[RS]{RS} M. R\o rdam and E. St\o rmer, {\em Classification of
nuclear $C^*$-algebras. Entropy in operator algebras}, Encyclopedia of
Mathematical Sciences, {\bf 126} (2002), Springer, Berlin. 

\bibitem [Sc]{Sc} C. Schochet, {\em Topological methods for $C^*$-algebras II:
geometric resolutions and the K\"{u}nneth formula}, Pac. J. Math., {\bf 98}
No. 2 (1982) 443--458.

\bibitem[Ta]{Ta} M. Takesaki, {\em Tomita's theory of modular Hilbert algebras
and its applications}, Lecture Notes in Mathematics, {\bf 128} (1970), Springer,
Berlin.

\bibitem[T]{T} J. Tomiyama, {\em On the projection of norm one in 
$W^*$-algebras}, Proc. Japan Acad., {\bf 33} (1957) 608--612.

\end{thebibliography}
\end{document}